\documentclass{amsart}
\usepackage{graphicx}
\usepackage{color,amssymb}
\usepackage{stmaryrd,esint}
\newcommand{\dt}{\partial_t}
\newcommand{\dx}{\partial_x}

\newcommand{\eps}{\varepsilon}
\newcommand{\av}[1]{\langle#1\rangle}

\newcommand{\jump}[1]{\llbracket#1\rrbracket}

\newcommand{\abs}[1]{\vert#1\vert}

\newcommand{\uzeta}{\underline{\zeta}}
\newcommand{\zetap}{\zeta_+}
\newcommand{\tzeta}{\underline{\zeta}}
\newcommand{\tzetap}{{\underline{\zeta}{}_+}}
\newcommand{\tzetam}{\underline{\zeta}{}_-}
\newcommand{\tzetapm}{\underline{\zeta}{}_\pm}
\newcommand{\dtzetapm}{\dot{\underline{\zeta}}{}_\pm}
\newcommand{\dtzetap}{\dot{\underline{\zeta}}{}_+}
\newcommand{\dtzetam}{\dot{\underline{\zeta}}{}_-}
\newcommand{\uh}{\underline{h}}
\newcommand{\dzetapm}{\dot{\zeta}_\pm}

\newcommand{\zetam}{\zeta_-}
\newcommand{\zetapm}{\zeta_\pm}

\newcommand{\cE}{{\mathcal E}}

\newcommand{\dsp}{\displaystyle}

\newcommand{\mfsw}{{\mathfrak f}_{\rm sw}}

\newcommand{\RR}{{\mathbb R}}

\newtheorem{proposition}{Proposition}[section]

 \theoremstyle{remark}

\newtheorem{remark}{Remark}[section]

\title{A numerical method for wave-structure interactions in the Boussinesq regime }

\begin{document}

\markboth{Geoffrey Beck, David Lannes, Lisl Weynans}{A numerical method for wave-structure interactions in the Boussinesq regime}

%
%

\maketitle

\begin{center}
Geoffrey Beck, David Lannes, Lisl Weynans
\end{center} 

\begin{abstract}
The goal of this work is to study waves interacting with partially immersed objects allowed to move freely in the vertical direction, and in a regime in which the propagation of the waves is described by the one dimensional Boussinesq-Abbott system. The problem can be reduced to a transmission problem for this Boussinesq system, in which the transmission conditions between the components of the domain at the left and at the right of the object are determined through the resolution of coupled forced ODEs in time satisfied by the vertical displacement of the object  and the average discharge in the portion of the fluid located under the object. We propose a new extended formulation in which these ODEs are complemented by two other forced ODEs satisfied by the trace of the surface elevation at the contact points. The interest of this new extended formulation is that the forcing terms are easy to compute numerically and that the surface elevation at the contact points is furnished for free. Based on this formulation, we propose a second order scheme that involves a generalization of the MacCormack scheme with nonlocal flux and a source term, which is coupled to a second order Heun scheme for the ODEs. In order to validate this scheme, several explicit solutions for this wave-structure interaction problem are derived and can serve as benchmark for future codes. As a byproduct, our method provides a second order scheme for the generation of waves at the entrance of the numerical domain for the Boussinesq-Abbott system.
\end{abstract}



\section{Introduction}

\subsection{Presentation of the problem}

While the first studies of the interactions of waves with floating structures go back at least to John's paper \cite{John1}, or to the phenomenological integro-differential equation derived by Cummins to describe the linear motion of floating structures \cite{Cummins},  this research field became increasingly active in recent years. A first reason for this renewed interest  is related to  the development of renewable marine energies as one of the tools for energy transition. Indeed, several devices of offshore wind-turbines and wave-energy convertors involve partially immersed structures \cite{Babarit}. 

A second reason for the recent mathematical activity on wave-structure interactions is that this has been made technically feasible thanks to the recent progresses on the mathematical understanding of the propagation of water waves. The initial value problem in domains without boundaries (${\mathbb R}^d$ or ${\mathbb T}^d$) is now well understood for the full water waves (also called free-surface Euler) equations, as well as for asymptotic models in shallow water (such as the nonlinear shallow water equations, the Boussinesq systems, the Serre-Green-Naghdi equations). Recently, the initial value problem has also been studied in domains with a boundary. When the fluid domain is delimited by vertical sidewalls, the water waves equations have been studied in \cite{ABZ}; in the case of non-vertical sidewalls, this problem has been considered in \cite{Poyferre,MingWang} for the water waves equations, and in  \cite{LannesMetivier} for the shallow water and Green-Naghdi equations. The initial boundary value problem, in which one imposes initial and boundary datas, has also been investigated for the Boussinesq equations \cite{BLM,LannesWeynans}. These advances make it more  realistic to address the issues raised by the presence of a partially immersed object. 

From the numerical point of view, efficient numerical codes based on shallow water models have been developed recently and can be used to address realistic submersion issues (see for instance \cite{Funwave,Uhaina}); here also, it is now a reasonable prospect to address the  specific difficulties raised by wave-structure interactions.

The present paper is a contribution to the theoretical and numerical understanding of these interactions, inasmuch as it provides a precise description of the motion of a partially immersed object allowed to move freely in the vertical direction under the action of waves described by a nonlinear dispersive model (the standard Boussinesq-Abbott system), see Figure \ref{fig-floating}. It has to be considered as a partial (affirmative) answer to the wider question:
can the efficient modelling of waves based on shallow water models be extended to allow the presence of floating structures? If this happens to be true, the gain in computational time would allow to investigate the behavior of many floating structures (the so-called farms of wave-energy convertors or offshore wind turbines), as well as their impact on the wave fields, which can have significant consequences in coastal regions. Answering such question is out of reach for CFD methods that can be used to describe the behavior of one wave-energy convertor, and also, to a lower extent, for potential methods (see for instance \cite{EL,GKC}). On the other hand, the linear methods based on Cummins' equation used in commercial softwares such as Wamit neglect the nonlinear effects that can be important \cite{PGR}, especially in shallow water, and are unable to provide a precise description of the impact of a wave-farm on the wave-field.

 \begin{figure}[h!]
        \centering \includegraphics[width=8cm]{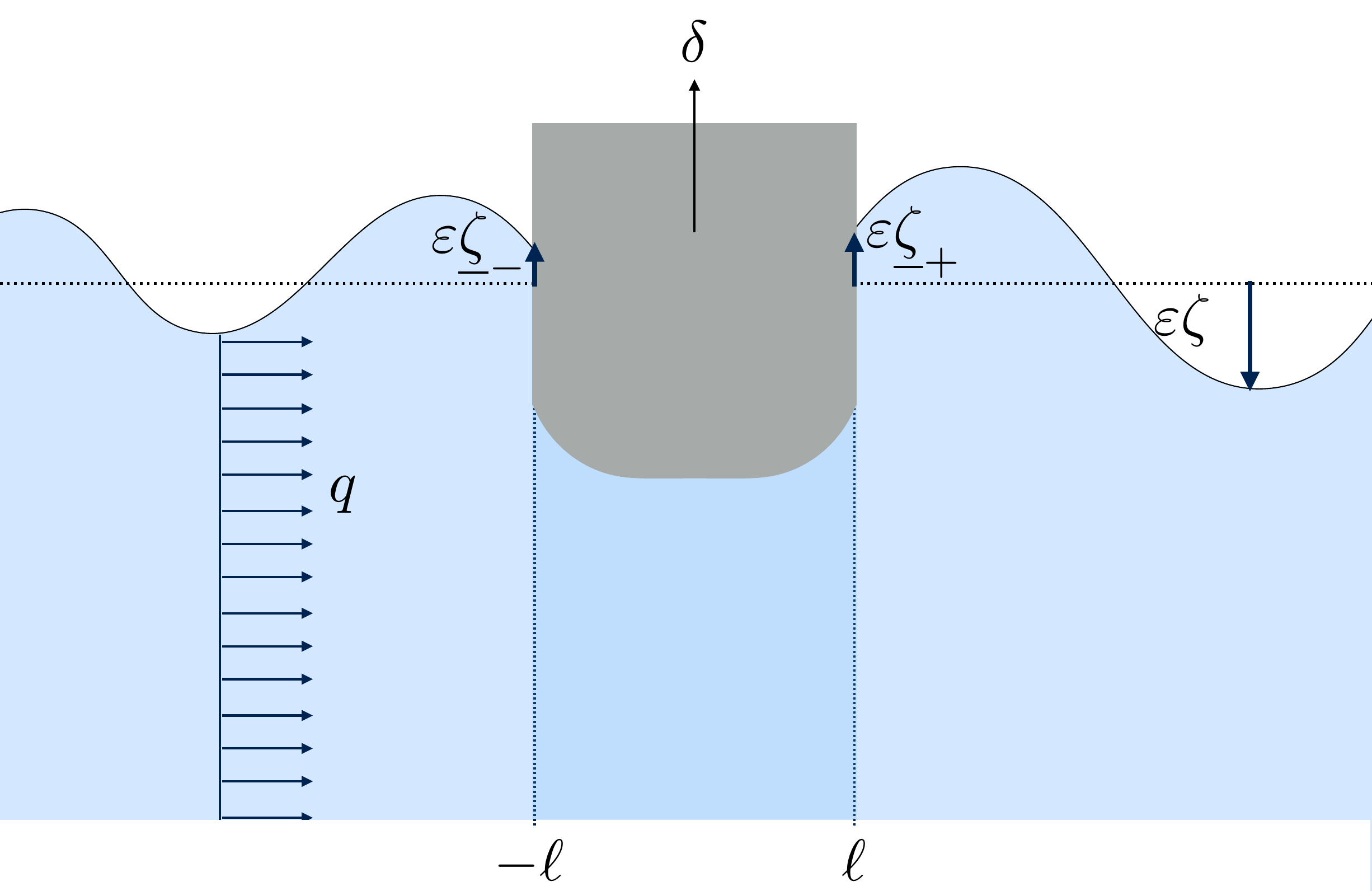}
        \caption{The floating object}
        \label{fig-floating}
\end{figure}

The presence of a floating structure in a shallow water model can be taken into account following the approach proposed in \cite{Lannes_float} where the horizontal plane is decomposed into two regions: the interior region (below the floating object), and the exterior region (below the free surface waves). In the exterior regions, the standard (depth integrated) shallow water model is used, while in the interior region, an additional pressure term is present. This pressure term corresponds to the pressure exerted by the fluid on the object (and which eventually makes it move through Newton's equations), and can be understood as the Lagrange multiplier associated with the constraint that, under the object, the surface elevation of the waves is constrained as it must by definition coincide with the bottom of the object. It is possible to relax this constraint by approximating the pressure term by a pseudo-compressible relaxation; one can then use the same kind of asymptotic preserving schemes as for the low-Mach limit in compressible gases. This approach has been used in the present context in \cite{GPSW1,GPSW2}, and is also relevant for other instances of partially congested flows \cite{PerrinSaleh,DalibardPerrin,BianchiniPerrin}. In this paper, we rather consider the original (non relaxed) problem, which requires to understand precisely the coupling between the interior and exterior regions. 

It turns out that this wave-structure interaction problem can be reduced to an initial boundary value problem for the wave model in the exterior region, with non standard boundary (or transmission) conditions. In the case where the horizontal dimension is $d=1$, the object has vertical walls located at $x=\pm\ell$ as in Figure \ref{fig-floating} and is only allowed to move vertically,  and if the wave model is given by the nonlinear shallow water equations, it was shown in \cite{Lannes_float} that this transmission problem takes the form (in dimensionless variables, see Section \ref{sectaugmented} for details),
$$
\begin{cases}
\dt \zeta + \dx q=0,\\
\dt q +\eps \dx (\frac{1}{h}q^2)+h\dx \zeta=0,
\end{cases}
\quad \mbox{ on } (-\infty,-\ell)\cup (\ell,+\infty),
$$
where $\zeta$ is the elevation of the surface and $q$ the horizontal discharge.
Using the notation $\jump{q}=q(\ell)-q(-\ell)$ and $\av{q}=\frac{1}{2}(q(-\ell)+q(\ell))$, the transmission conditions are given by
$$
\jump{q}=-2\ell \dot \delta \quad\mbox{ and }\quad \av{q}=\av{q_{\rm i}}
$$
where the function $\delta$ and $\av{q_{\rm i}}$ (representing respectively the vertical displacement of the object and the mean discharge under the object) solve an ODE of the form
$$
\frac{d}{dt}\theta=F(\theta,\zeta_{\vert_{x=-\ell}}, \zeta_{\vert_{x=\ell}}),
$$
with $\theta=(\delta,\dot\delta,\av{q_{\rm i}})$ and $F$ a smooth function of no importance at this stage of the discussion. The coupling acts in two ways: it is necessary to known $\theta$ to solve the transmission problem for $(\zeta,q)$, and it is necessary to know the solution $(\zeta,q)$ to determine the forcing term $F(\theta,\zeta_{\vert_{x=-\ell}}, \zeta_{\vert_{x=\ell}})$ in the ODE for $\theta$. A key point in the mathematical analysis of this problem is the regularity of the traces $\zeta_{\vert_{x=\pm \ell}}$. Such a control is furnished by the construction of a Kreiss symmetrizer, as shown in \cite{IguchiLannes} where general initial boundary value problems, possibly with a free boundary, are considered for a wide class of hyperbolic systems; they include the above transmission problem as well as the more complex free boundary problem one has to deal with when the lateral boundaries of the object are not vertical. 

Numerically, the evaluation of the  traces $\zeta_{\vert_{x=\pm \ell}}$ also requires a careful treatment which relies on the Riemann invariants associated with the nonlinear shallow water equations  \cite{Lannes_float}; we also refer to  \cite{BEER} for a higher order scheme, to \cite{Bocchi3} where a wave-energy device is simulated using this approach (the oscillating water column), and to   \cite{SuTucsnak} where controlability issues were also addressed. The more complex case of an object freely floating and with non-vertical walls (and therefore nontrivial dynamics for the contact points) has been solved theoretically in \cite{IguchiLannes}, and numerically in \cite{Haidar} using ALE methods to treat the evolution of the contact points. 

A variant of the above wave-structure interaction problem for the viscous nonlinear shallow water equations was also considered in \cite{Maity} and an extension to the case of horizontal dimension $d=2$ with radial symmetry has been considered theoretically in \cite{Bocchi} and the so-called decay test (or return to equilibrium) investigated in the same configuration under an additional assumption of linearity in \cite{Bocchi2}. Let us also mention \cite{Parisot} where the dynamics of trapped air pockets are studied.

We propose here an extension in another direction. The principal drawback of the nonlinear shallow water equations is that they neglect the dispersive effects that play an important role is some important situations (they allow for instance the existence of solitary waves). The most simple models that generalize the nonlinear shallow water equations by adding dispersive terms are the Boussinesq equations (see \cite{LannesModeling} for a recent review on shallow water models). Replacing the nonlinear shallow water equations by the so called Boussinesq-Abbott system in the above example, one obtains the following transmission problem (see Section  \ref{sectaugmented} for more details),
$$
\begin{cases}
\dt \zeta + \dx q=0,\\
(1-\kappa^2\dx^2)\dt q +\eps \dx (\frac{1}{h}q^2)+h\dx \zeta=0,
\end{cases}
\quad \mbox{ on } (-\infty,-\ell)\cup (\ell,+\infty),
$$
and with transmission conditions
$$
\jump{q}=-2\ell \dot \delta \quad\mbox{ and }\quad \av{q}=\av{q_{\rm i}}
$$
where the function $\delta$ and $\av{q_{\rm i}}$ (representing respectively the vertical displacement of the object and the mean discharge under the object) solve an ODE of the form
$$
\frac{d}{dt}\theta=F_\kappa(\theta,\zeta_{\vert_{x=-\ell}}, \zeta_{\vert_{x=\ell}}, \kappa^2 (\dt^2 \zeta)_{\vert_{x=-\ell}},\kappa^2 (\dt^2 \zeta)_{\vert_{x=\ell}}  ),
$$
with $\theta=(\av{q_{\rm i}},\dot\delta, \delta)^{\rm T}$ and $F_\kappa$ a smooth function of no importance at this stage of the discussion. 

The differences with the nondispersive case considered above are the operator $(1-\kappa^2\dx^2)$ applied in front of $\dt q$ in the evolution equations and a contribution of the trace of $\dt^2\zeta$ to the forcing term in the ODE for $\theta$. Formally, the nondispersive case is obtained by setting $\kappa=0$, but the mathematical and numerical differences are considerable.
Contrary to the hyperbolic case mentioned above, there is no general theory for initial boundary value problems associated with nonlinear dispersive systems and this is why several approximations have been used to bypass this issue. In \cite{BEER}, wave-structure interactions using a Boussinesq model was used, but the issue at the boundary was avoided by using the (dispersionless) nonlinear shallow water equations in a small region around the object;  in \cite{MIS} the behavior at the boundary was approximated at second order using Bessel expansions and matched asymptotics;  in \cite{Karambas}, Boussinesq type equations where computed in the whole domain, neglecting the singularities of the surface elevation and of the discharge at the contact line, while the presence of the object is taken into account by adding an additional pressure term in the interior region.  There are also approximate methods based on sponge layers and artificial source terms which are often used to generate waves at the entrance of the numerical domain \cite{WeiKirbySinha}. Such methods are far too rough to be used in the present case, where a precise description of the waves at the contact points is needed; indeed, as shown in \cite{BLM}, the behavior at the contact points can be quite complex and exhibit dispersive boundary layers. This is why a new method to handle non-homogeneous initial boundary value problems for the Boussinesq equations was proposed in \cite{LannesWeynans} and numerically implemented with an order $1$ scheme. 

The approach used in the present paper allows us to treat the issues related to the initial boundary value problem for Boussinesq-type equations without any approximation; a byproduct of independent interest of the present paper is that it furnishes a second order method for the generation of waves at the numerical boundary of the fluid domain for the Boussinesq equations, hereby complementing the first order generation scheme of \cite{LannesWeynans}.

As for  the hyperbolic case discussed previously, the control of the traces of $\zeta$ (and a fortiori of $\dt^2\zeta$) is a key ingredient of the analysis, both from the PDE and numerical perspectives; however, due to the presence of dispersion, there are no such things as a Kreiss symmetrizer or Riemann invariants to help us. To solve these issues, we propose in this paper an  {\it extended formulation} of the equations. After remarking that the traces $\zeta_{\vert_{x=\pm\ell}}$ solve a second order forced ODE, we introduce new unknowns $\uzeta_\pm$ defined as the solutions of this ODE. This allows us to replace $\zeta_{\vert_{x=\pm\ell}}$ by $\uzeta_\pm$ in the forcing term $F_\kappa$ in the ODE for $\theta$, hereby avoiding the computation of the traces. The resulting extended formulation just consists in replacing the above ODE for $\theta$ by the higher dimensional ODE 
$$
\frac{d}{dt}\Theta={\mathcal G}(\Theta,(R_1 {\mathfrak f}_{\rm sw})_{\vert_{x=\pm\ell}})
$$
with $\Theta=(\av{q_{\rm i}},\dot\delta,\dot{\uzeta_+},\dot{\uzeta_-},\delta,\uzeta_+,\uzeta_-)^{\rm T}$ and ${\mathcal G}$ is a smooth mapping. This ODE is forced by the terms $(R_1 {\mathfrak f}_{\rm sw})_{\vert_{x=\pm\ell}}$ which depend on the solution $(\zeta,q)$ of the Boussinesq equations in the fluid domain; the precise meaning of this term will be given in Section  \ref{sectaugmented}, the important thing being that the control  of the trace of the quantities $R_1 {\mathfrak f}_{\rm sw}$ at $x=\pm\ell$ does not raise any theoretical nor numerical difficulty.

\medbreak

The second step of our approach consists in transforming this extended  {\it transmission} problem into an {\it initial value} problem coupled with forced ODEs: this means that we do no longer have to bother with the boundary conditions (which are automatically propagated by the flow). This new formulation can be written as a system of conservation laws with nonlocal flux and an exponentially localized source term,
$$
\dt U+\dx ({\mathfrak F}_\kappa(U))={\mathcal S}_\pm(\Theta,(R_1 {\mathfrak f}_{\rm sw})_{\vert_{x=\pm\ell}}){\mathfrak b}(x\mp\ell) \quad\mbox{in }\quad \pm (\ell,\infty),
$$
where ${\mathfrak F}_\kappa$ denotes the nonlocal flux, while ${\mathcal S}_\pm$ is a smooth function of its arguments and ${\mathfrak b}$ an exponentially localized function. The above ODE for $\Theta$ allows one to compute the source term in this system of nonlocal conservation laws; conversely, the resolution of this system allows one to compute the forcing terms $(R_1 {\mathfrak f}_{\rm sw})_{\vert_{x=\pm\ell}}$ in the ODE for $\Theta$: the coupling acts therefore both ways.

One of the advantages of this new formulation is that we can implement on it a second order scheme that couples a MacCormack predictor corrector scheme (generalized to handle nonlocal fluxes and a source term)  for the computation of the waves, and a second order Heun scheme for the computation of the forced ODEs. We also exhibit several exact explicit solutions that we use to study the convergence of our code and that are of independent interest.

\subsection{Organization of the paper}

In Section \ref{sectaugmented}, we derive the formulation of the problem our numerical scheme is based on: we first recall in \S \ref{sectreduc} the reduction of \cite{BeckLannes} to a transmission problem for the Boussinesq equation on the two connected components of the exterior region, and then show in \S \ref{secttraces} that the traces of the surface elevation at the contact points satisfy a forced second order ODE that we use to write the new augmented formulation of the transmission problem in \S \ref{sectaugmform}; this transmission problem is finally rewritten as an initial boundary value problem in \S \ref{secttransfIVP}.

The numerical schemes are presented in Section \ref{sectschemes}. The initial value problem obtained in the previous section is a set of two conservation equations with nonlocal flux and an exponentially decaying source term whose coefficient is found by solving a set of forced second order ODEs. We propose two numerical schemes based on an abstract formulation of these equations. The first one, described in \S \ref{sectLF}, is of first order and is an adaptation of the Lax-Friedrichs scheme to the present context. The second one, studied in \S \ref{sectMC}, is of second order. It is based on the MacCormack predictor-corrector scheme for the two conservation PDEs (with adaptations to handle the nonlocal flux and the source term), and on  a Heun scheme for the ODE part.

Numerical simulation are then presented in Section \ref{sectnumsim}. We investigate several configurations exploring different aspects of the coupling between the Boussinesq equations and the forced ODEs used in the transmission conditions. Wave generation is considered in \S \ref{sectWG}, and is of independent interest as it provides a way to generate waves at the entrance of the numerical domain for the Boussinesq equations. The return to equilibrium test in which an object oscillates vertically after being released from an out of equilibrium position is studied in \S \ref{sectRTE}; in the linear case, an explicit solution is exhibited and computed via Laplace transforms, and this solution is used to assess the precision of our scheme. Interactions of waves with a fixed object are then investigated in  \S \ref{sectfixed}; here also, an exact solution is derived in the linear case and used for validation. The most general configuration of waves interacting with an object allowed to move freely in the vertical direction is then considered in \S \ref{sectfreely}.

Throughout this article, we work with an abstract and concise formulation of the equations. The precise equations, with the expressions of the various coefficients involved, is postponed to Appendix \ref{appcoeff}.

\subsection{Notation}

- The horizontal axis ${\mathbb R}$ is decomposed throughout this paper into an {\it interior region} ${\mathcal I}=(-\ell,\ell)$ and an {\it exterior region} ${\mathcal E}={\mathcal E}_+\cup {\mathcal E}_-$ with ${\mathcal E}_-=(-\infty,-\ell)$ and ${\mathcal E}_+= (\ell,\infty)$, and two {\it contact points} ${x=\pm \ell}$. \\
- For any function $f\in C(\overline{\cE})$, we denote
$$
f_\pm=f_{\vert_{x=\pm\ell}},\qquad
\jump{f} =f_+-f_-
\quad\mbox{ and }\quad
\av{f}=\frac{1}{2}\big(f_++f_-\big).
$$
- If $f\in C^1([0,T])$, we sometimes use the notation $\dot f=\frac{d}{dt} f$.\\
- We denote by $\mfsw$ the momentum flux associated with the nonlinear shallow water equations,
\begin{equation}\label{defmfsw}
\mfsw=\left( \eps \frac{q^2}{h}+\frac{h^2-1}{2\eps} \right).
\end{equation}

\section{An augmented formulation of the wave-structure interaction equations}\label{sectaugmented}

The goal of this section is to derive the augmented formulation of the wave-structure equations that we shall use in Section \ref{sectschemes} to propose numerical schemes. We first sketch in \S \ref{sectreduc} the main steps of the analysis of \cite{BeckLannes} that led to a formulation of the problem as a transmission problem between the two connected components of the fluid domain, and with transmission conditions determined through the resolution of an ODE forced by a source term involving  the traces at the contact points of the surface elevation and of their second order time derivative. We then remark in \S \ref{secttraces} that these traces solve themselves a second order ODE, but which is forced by a source term which is easier to compute. This observation is the key ingredient that allows us to derive in \S \ref{sectaugmform} an augmented formulation. It has the same structure as the formulation derived in \S \ref{sectreduc}, namely, it is a transmission problem coupled with a forced ODE. The crucial difference is that this ODE does no longer require the computation of the traces of the surface elevation at the contact points and that it can easily be computed numerically. Finally, we show  in \S \ref{secttransfIVP} that this augmented transmission problem can be rewritten as an initial value problem, which is the structure the numerical schemes of Section \ref{sectschemes} are based on.

\subsection{Reduction to a transmission problem coupled with scalar ODEs}\label{sectreduc}

We remind here the main steps of the derivation of the equations describing the interactions of a partially immersed object with one-dimensional waves in a regime where these waves can correctly be described by the Boussinesq-Abbott equations (see \cite{LannesModeling}); the object is assumed to have vertical sidewalls and can be either fixed, in forced vertical motion, or allowed to float freely in the vertical direction under the action of the waves. 

In dimensionless variables, the equations involve two coefficients $\eps$ and $\mu$, respectively called nonlinearity and shallowness parameters, and that are defined as
$$
\eps=\frac{\mbox{typical amplitude of the waves}}{\mbox{typical depth}}
\quad\mbox{ and }\quad
\mu=\Big(\frac{\mbox{typical depth}}{\mbox{typical horizontal scale}}\Big)^2;
$$
in the {\it weakly nonlinear shallow water regime} in which the Boussinesq-Abbott equations are known to provide a good approximation of the motion of the waves, one has
$$
\mu \ll 1 \quad\mbox{ and }\quad \eps=O(\mu);
$$
these conditions are assumed throughout this article. For the sake of conciseness, we also introduce the parameter $\kappa$ as
$$
\kappa=\big( \frac{\mu}{3} \big)^{1/2};
$$
this parameter plays an important role as it measures the size of the dispersive boundary layers that appear in the analysis of mixed initial boundary-value problems for the Boussinesq equations, which are a dispersive perturbation of an hyperbolic system \cite{BLM}.

\medbreak

As displayed in Figure \ref{fig-floating}, in dimensionless coordinates, the surface of the fluid is parametrized at time $t$ by the function $x\in\RR \mapsto \eps \zeta(t,x)$, and the horizontal discharge (the vertical integral of the horizontal component of the velocity field) at time $t$ and position $x$ is denoted $q(t,x)$. We also sometimes denote by $h$ the water depth, $h=1+\eps\zeta$. Finally, we denote by $\underline{P}(t,x)$ the pressure at the surface of the fluid, namely, $\underline{P}=P(t,x,\eps\zeta(t,x))$ if $P$ denotes the pressure field in the fluid.
 
 Regarding the solid object, we denote by $\pm\ell$ the position of its vertical sidewalls and by $\eps\zeta_{\rm w}$ the parametrization on $(-\ell,\ell)$ of its bottom (the subscript "w" stands for "wetted part"); we also denote by $\eps\delta(t)$ the vertical deviation of the object from its equilibrium position, and by $h_{\rm eq}$ the water depth at rest. These quantities are related through
 $$
 \zeta_{\rm w}(t,x)=\delta(t)+\frac{1}{\eps}\big(h_{\rm eq}(x)-1 \big).
 $$
 
 \noindent

\noindent
 {\bf N.B.} For the sake of simplicity, we assume throughout this article that 
the center of mass is located at $\{x=0 \}$
and that $h_{\rm eq}(x)$ is an even function. \\

The Boussinesq-Abbott equations for the motion of the waves are given for $t>0$, $x\in \RR$ by
\begin{equation}\label{BoussinesqAbbott}
\begin{cases}
\dt \zeta+\dx q=0,\\
(1-\kappa^2\dx^2)\dt q +\eps \dx \big( \frac{1}{h}q^2\big)+h\dx\zeta=-\frac{1}{\eps} h \dx \underline{P}
\end{cases}
\end{equation}
(with $h=1+\eps\zeta$). We now have to distinguish between the {\it exterior} region $\cE=(-\infty,-\ell)\cup (\ell,\infty)$ where the surface of the water is in contact with the air and the {\it interior} region ${\mathcal I}=(-\ell,\ell)$ where it is in contact with the object,
\begin{itemize}
\item In the {\it{exterior}} region $\cE$, the surface elevation $\zeta$ is free, but the surface pressure $\underline{P}$ is constrained, assumed to be equal to the (constant) atmospheric pressure $P_{\rm atm}$,
$$
\underline{P}(t,x)=P_{\rm atm} \quad \mbox{ for } t>0, \quad x\in \cE;
$$
the right-hand side in the second equation of \eqref{BoussinesqAbbott} therefore vanishes.
\item In the {\it{interior}} region $\cE$, it is the reverse: the surface elevation is constrained because it has to coincide with the bottom of the object,
$$
\zeta(t,x)=\zeta_{\rm w}(t,x) \quad \mbox{ for } t>0, \quad x\in {\mathcal I},
$$
but there is no constraint on the surface pressure $\underline{P}$ which, under the general approach of \cite{Lannes_float}, can be understood as the Lagrange multiplier associated with the constraint on the surface elevation. Plugging the constraint equation in the first equation of \eqref{BoussinesqAbbott} one directly gets that
$$
q(t,x)=-x\dot \delta+\av{q_{\rm i}}(t),
$$
where $\av{q_{\rm i}}$ is a time dependent function corresponding to the average discharge over the interior region. Using this relation and applying $\dx$ to the second equation in \eqref{BoussinesqAbbott} provides an elliptic equation for $\underline{P}$,
$$
-\dx \big( \frac{1}{\eps} h_{\rm w} \dx \underline{P}\big)=-\ddot\delta+\dx\big[h_{\rm w}\dx\zeta_{\rm w}+\eps\dx\big(\frac{1}{h_{\rm w}}(-x\dot\delta+\av{q_{\rm i}}\big)^2)\big],
$$
for $ x\in (-\ell,\ell)$ and
with $h_{\rm w}=h_{\rm eq}+\eps\delta$. If we know the boundary values of $\underline{P}$ at $x=-\ell+0$ and $x=\ell-0$, this elliptic equation can be solved and it provides an expression for $\underline{P}$ in terms of $h_{\rm eq}$, $\delta$, $\av{q_{\rm i}}$ and of this boundary data. Using this expression in the second equation of \eqref{BoussinesqAbbott} then provides an expression for $\frac{d}{dt}\av{q_{\rm i}}$ in terms of the same quantities.
\end{itemize}
We also need coupling conditions at the contact points $x=\mp\ell$ between the exterior and interior region. There are two of them,
\begin{itemize}
\item Continuity of the horizontal discharge. Taking into account the expression of the discharge derived above in the interior region, this condition yields
\begin{equation}\label{BCdischarge}
q(t,-\ell-0)=\ell\dot\delta+\av{q_{\rm i}}(t)\quad\mbox{ and }\quad q(t,\ell+0)=-\ell\dot\delta+\av{q_{\rm i}}(t).
\end{equation}
\item Conservation of the total energy. Imposing conservation of the total (i.e., fluid+solid) energy classically provides the boundary data needed to solve the elliptic equation derived for the surface pressure in the interior region \cite{Maity,Bocchi,BLM}. We refer to \cite{BeckLannes} for the derivation of these boundary data in the present context, but do not provide it here explicitly for the sake of conciseness.
\end{itemize}

To summarize, we have the standard Boussinesq-Abbott equation in the exterior region, with boundary condition on the discharge $q$ at $\mp\ell$ that are given in terms of two functions of time, namely, $\av{q_{\rm i}}$ and $\delta$. As said above, the fact that the elliptic equation for the pressure in the interior region has been solved provides an evolution equation for $\av{q_{\rm i}}$; the last thing to do is therefore to determine $\delta$. If the object is fixed or in forced motion then $\delta$ is given; otherwise, it is of course given by Newton's equation. The three cases can be considered simultaneously by allowing an external force to be applied to the solid (if the solid is fixed or in forced motion, this external force $F_{\rm ext}$ represents the vertical force exerted on the solid to maintain it fixed or with the desired motion). The outcome of this analysis, as shown\footnote{The presence of the external force is not taken into account in that reference. It is however straightforward to add it in Newton's equation; note that in the present dimensionless setting, the force has been nondimensionalized by $2\ell\rho g$.} in Theorem 3.1 of \cite{BeckLannes} is that the wave-structure interaction problem under consideration can be reduced to a transmission problem. Using the notations
$$
\av{f} =\frac{1}{2}\big( f(\ell)-f(-\ell)\big) \quad\mbox{ and }\quad \jump{f} =f(\ell)-f(-\ell)
$$
for all $f\in C((-\infty,-\ell]\cup[\ell,\infty))$, this transmission problem can be written
\begin{equation}\label{transm1}
\begin{cases}
\dt \zeta+\dx q=0,\\
(1-\kappa^2\dx^2)\dt q +\dx \mfsw=0
\end{cases}
\quad\mbox{ for }\quad t>0, \quad x\in \cE
\end{equation}
where $\mfsw$ is the shallow water momentum flux given by \eqref{defmfsw}, and with transmission conditions across the floating object given by
\begin{equation}\label{transm2}
\av{q}=\av{q_{\rm i}}\quad \mbox{ and } \quad \jump{q}=-2\ell \dot\delta,
\end{equation}
where $\av{q_{\rm i}}$ and $\delta$ are functions of time solving
\begin{align}
\label{transm3}
\alpha(\eps\delta)\frac{d}{dt}\av{q_{\rm i}}+\eps \alpha'(\eps\delta)\dot\delta\av{q_{\rm i}}&=-\frac{1}{2\ell}\jump{\zeta+{\mathfrak G}},\\
\label{transm4}
\tau_\kappa(\eps\delta)^2\ddot\delta+\delta-\eps\beta(\eps\delta)\dot\delta^2-\eps\frac{1}{2}\alpha'(\eps\delta)\av{q_{\rm i}}^2&=\av{\zeta+{\mathfrak G}}+F_{\rm ext},
\end{align}
where ${\mathfrak G}$ is the function defined on $\cE$ by
\begin{equation}\label{mathG}
{\mathfrak G}=\eps\frac{1}{2}\frac{q^2}{h^2}-\kappa^2\frac{1}{h}\dx\dt q,
\end{equation}
and where the explicit expression of the functions $\alpha$, $\tau_\kappa$ and $\beta$, of no importance at this point of the discussion, are provided in \S \ref{appexpcoeff} of Appendix \ref{appcoeff}. We just want to emphasize that the coefficient $\tau_\kappa(\eps\delta)^2$ in front of $\ddot\delta$ in \eqref{transm4} takes into account the contribution of the added mass effect  (when a solid moves in a fluid, not only must it accelerate its own mass but also the mass of the fluid around it). 

The initial value problem corresponding to \eqref{transm1}-\eqref{transm4} is studied and solved in \cite{BeckLannes}. Its structure is that of a transmission problem coupled with a set of ODEs on $\av{q_{\rm i}}$ and $\delta$. This coupling acts in both ways: on the one hand, it is necessary to know $\av{q_{\rm i}}$ and $\delta$ in order to solve the transmission problem \eqref{transm1}-\eqref{transm2} and on the other hand, one needs to know the solution $(\zeta,q)$ of this transmission problem to compute the source term in the right-hand side of \eqref{transm3}-\eqref{transm4}. From the numerical view point, this last step is not easy to treat since one has to compute the numerical trace of $\zeta$ and $\dt\dx q$ at the contact points $x=\pm\ell$. The key ingredient we propose here to overcome this difficulty is to work with an augmented formulation of the problem, with additional functions of time involved in the system of ODEs for $\delta$ and $\av{q_{\rm i}}$, but where the computation of such traces is no longer needed.

\subsection{The trace equations}\label{secttraces}

The source terms in the right-hand sides of \eqref{transm3}-\eqref{transm4} involve the trace of $\zeta+{\mathfrak G}$ at $x=\pm \ell$, with ${\mathfrak G}$ given by \eqref{mathG}. Since \eqref{transm2} implies that $q_{\vert_{x=\pm\ell}}=\mp \ell \dot\delta +\av{q_{\rm i}}$ and remarking that one deduces from the first equation of \eqref{transm1} that $\dx\dt q=-\dt^2 \zeta$, we have
$$
{\mathfrak G}_{\vert_{x=\pm \ell}}=\eps\frac{1}{2}\Big(\frac{\mp \ell \dot\delta +\av{q_{\rm i}}}{1+\eps\zeta_\pm}\Big)^2+\kappa^2\frac{1}{1+\eps\zeta_\pm}\ddot\zeta_\pm,
$$
with $\zeta_{\pm}:=\zeta_{\vert_{x=\pm\ell}}$. The difficulty therefore lies in the computation of the trace of $\zeta$ at $x=\pm \ell$ and of their second time derivative. The augmented formulation consists in treating $\zeta_\pm$ as a new unknown function of time instead of getting it by taking the traces of $\zeta$ at the contact points. This is made possible by the following proposition which provides a second order ODE satisfied by $\zeta_+$ and $\zeta_-$. This requires first the introduction of the Dirichlet and Neumann inverses of the operator $(1-\kappa^2\dx^2)$ on $\cE$, respectively denoted by $R_0$ and $R_1$. They are defined for all $F\in L^2(\cE)$ by
$$
R_0 F= u \quad \mbox{ with } \begin{cases} (1-\kappa^2 \dx^2)u= F &\mbox{ on }\cE,
\\ u_{\vert_{x=\pm\ell}}=0, \end{cases}
$$
and
\begin{equation}\label{defR1}
R_1 F= v \quad \mbox{ with } \begin{cases} (1-\kappa^2 \dx^2)v= F &\mbox{ on }\cE,\\
 \dx v_{\vert_{x=\pm\ell}}=0. \end{cases}
\end{equation}
We can now state the following proposition. Note that the ODEs satisfied by $\zeta_\pm$ only make sense in the presence of dispersion  ($\kappa >0$). 

\begin{proposition}\label{prop1} Let $f$ and $g$ be two continuous functions of time. If $(\zeta,q)$ is a smooth solution to 
\begin{equation}\label{pb1-1-gen}
\begin{cases}
\dt \zeta+\dx q=0,\\
(1-\kappa^2\dx^2)\dt q +\dx \mfsw=0,
\end{cases}
t>0, \quad x\in\cE
\end{equation}
with $\mfsw$ as in \eqref{defmfsw} and with transmission conditions
\begin{equation}\label{pb1-2-gen}
\av{q}(t)=f(t)\quad\mbox{ and }\quad \jump{q}(t)=2g(t), \qquad t>0
\end{equation}
then $\zeta_\pm=\zeta_{\vert_{x=\pm\ell}}$ solve the ODEs
\begin{equation}\label{trace-equation-gen}
\begin{cases}
\dsp \dt^2 \zetap +  \frac{1}{\kappa^2}  \zetap+ \frac{\eps }{\kappa^2} \big( \frac{1}{2}\zetap^2 + \frac{(f+g)^2}{1+\eps\zetap}\big)   =\frac{1}{\kappa^2} ({R}_1 \mfsw)_++\frac{1}{\kappa} (\dot f+\dot g),\\
\dsp \dt^2 \zetam +  \frac{1}{\kappa^2}  \zetam+ \frac{\eps }{\kappa^2} \big( \frac{1}{2}\zetam^2 + \frac{(f-g)^2}{1+\eps\zetam}\big)   =\frac{1}{\kappa^2} ({R}_1 \mfsw)_--\frac{1}{\kappa} (\dot f-\dot g),
\end{cases}
\end{equation}
where we used the notation  $({R}_1 \mfsw)_\pm= ({R}_1 \mfsw)_{\vert_{x=\pm\ell}}$.
\end{proposition}
\begin{proof}
Applying $R_0$ to the second equation in \eqref{pb1-1-gen} and using the boundary condition \eqref{pb1-2-gen}, one gets
$$
\dt q+R_0\dx \mfsw=(\dot f \pm \dot g) \exp(- \frac{\abs{x\mp \ell}}{\kappa})\quad \mbox{ on }\quad \cE^\pm.
$$
Remarking further that $R_0\dx=\dx R_1$, the problem is therefore reduced to
\begin{equation}\label{pb1reduced0}
\begin{cases}
\dt \zeta+\dx q=0,\\
\dt q+\dx R_1\mfsw=(\dot f\pm \dot g) \exp(- \frac{\abs{x\mp\ell}}{\kappa}).
\end{cases}
\end{equation}
Differentiating with respect to $x$ the second equation of \eqref{pb1reduced0} and using the fact that $\dt\dx q=-\dt^2 \zeta$, one gets
$$
-\dt^2 \zeta+\dx^2 R_1 \mfsw=\mp \frac{1}{\kappa} (\dot f\pm \dot g) \exp\big( -\frac{1}{\kappa}(\abs{x\mp\ell})\big).
$$
Since moreover $\dx^2 =-\frac{1}{\kappa^2}(1-\kappa^2 \dx^2)+\frac{1}{\kappa^2}$, we deduce that
$$
-\dt^2 \zeta- \frac{1}{\kappa^2}  \mfsw +\frac{1}{\kappa^2} R_1 \mfsw=\mp \frac{1}{\kappa} (\dot f \pm \dot g) \exp\big( -\frac{1}{\kappa}(\abs{x\mp\ell})\big).
$$
Taking the trace at $x=\pm \ell$, and substituting ${\mfsw}_{\vert_{x=\pm\ell}}= \zetapm+\eps \big( \frac{1}{2}\zetapm^2 + \frac{(f\pm g)^2}{1+\eps\zetapm}\big)$,
we obtain the equations stated in the proposition.
\end{proof}

\subsection{The augmented formulation}\label{sectaugmform}

Proposition \ref{prop1} can be applied to the wave-structure interaction system \eqref{transm1}-\eqref{transm4} with $f=\av{q_{\rm i}}$ and $g=-\ell\dot\delta$. Together with \eqref{transm3}-\eqref{transm4}, this shows that $\av{q_{\rm i}}$, $\delta$, $\zetap$ and $\zetam$ solve the second order differential system
\begin{equation}\label{GencoupleODE}
{\mathcal M}[\eps\delta,\eps\zetapm]
\frac{d}{dt}\begin{pmatrix} \dsp  \av{q_{\rm i}} \\  \dsp \dot\delta \\ \dsp \dot{{{\zeta}}}{}_+  \\ \dsp  \dot{{\zeta}}{}_- \end{pmatrix}
+ 
\begin{pmatrix} \dsp \frac{1}{2\ell} \jump{\zeta}\\  \dsp \delta-\av{\zeta} \\ \dsp  {{{\zeta}}}{}_+  \\ \dsp   {{\zeta}}{}_- \end{pmatrix}
= \eps {\boldsymbol {\mathfrak Q}}[\eps\delta,\eps\zetapm](\av{q_{\rm i}},\dot\delta,\zetapm)
+ \begin{pmatrix}  0 \\ F_{\rm ext} \\\dsp ({R}_1 \mfsw)_+ \\ \dsp  ({R}_1 \mfsw)_-  \end{pmatrix},
\end{equation}
where ${\mathcal M}[\eps\delta,\eps\zetapm]$ is the invertible matrix
\begin{equation}\label{defMtot}
{\mathcal M}[\eps\delta,\eps\zetapm]:=
\left(
\begin{array}{cc|cc}
\alpha(\eps\delta)& 0  & \frac{\kappa^2}{2 \ell }\frac{1}{1+\eps\zetap} &  - \frac{\kappa^2}{2 \ell }\frac{1}{1+\eps\zetam} \\ 
0 & \tau_\kappa(\eps\delta)^2 & -  \frac{1}{2} \frac{\kappa^2}{1+\eps\zetap} & -  \frac{1}{2} \frac{\kappa^2}{1+\eps\zetam} \\ \hline
-{\kappa} & \ell \kappa & \kappa^2 &0  \\
\kappa & \ell \kappa & 0 & \kappa^2 
\end{array}
\right),
\end{equation}
while ${\boldsymbol {\mathfrak Q}}[\eps\delta,\eps\zetapm](\av{q_{\rm i}},\dot\delta,\zetapm)$ is a four-dimensional vector whose entries are quadratic forms in $(\av{q_{\rm i}},\dot\delta,\zetapm)$ with coefficients depending on $\eps\delta$, $\eps\zetap$ and $\eps\zetam$ (the exact expression of these terms is of no importance at this point, and we refer the reader to Appendix \ref{appcoeff}). The second order differential system \eqref{GencoupleODE} can classically be transformed into a first order ODE on  $\av{q_{\rm i}}$, $\dot\delta$, $\dot\zetap$, $\dot\zetam$, $\delta$, $\zetap$ and $\zetam$ with forcing terms $(R_1\mfsw)_\pm$ and $F_{\rm ext}$ (see Appendix \ref{appGC}). The augmented formulation is obtained by replacing $\zetap$ and $\zetam$ by two additional unknowns $\tzetap$ and $\tzetam$ in this first order ODE. It reads therefore
\begin{equation}\label{transm1augm}
\begin{cases}
\dt \zeta+\dx q=0,\\
(1-\kappa^2\dx^2)\dt q +\dx\mfsw=0
\end{cases}
\quad\mbox{ for }\quad t>0, \quad x\in \cE,
\end{equation}
where $\mfsw$ is as in \eqref{defmfsw}, and with transmission conditions across the floating object given by
\begin{equation}\label{transm2augm}
\av{q}=\av{q_{\rm i}}\quad \mbox{ and } \quad \jump{q}=-2\ell \dot\delta,
\end{equation}
where $\av{q_{\rm i}}$ and $\delta$ are functions of time determined by the first order ODE 
\begin{equation}\label{ODEaugm}
\frac{d}{dt}\Theta={\mathcal G}\big(\Theta, (R_1\mfsw)_+,(R_1\mfsw)_-,F_{\rm ext}\big),
\end{equation}
with $\Theta:=\big(\av{q_{\rm i}},\dot\delta,\dot\tzetap,\dot\tzetam,\delta,\tzetap,\tzetam\big)^{\rm T}$ and where ${\mathcal G}$ is a smooth function of its arguments and whose exact expression is given in Appendix \ref{appcoeff}. It is a consequence of Proposition \ref{propIVP} below that if the initial data for $\tzetapm$ and $\dtzetapm$ are chosen appropriately, then $\zetapm=\tzetapm$ and $\dzetapm=\dtzetapm$ for all times, as expected (see Proposition \ref{propIVP} below).

\begin{remark}
The difference between the augmented formulation \eqref{transm1augm}-\eqref{ODEaugm} and the original formulation \eqref{transm1}-\eqref{transm4} lies in the ODE used to determine the functions $\delta$ and $\av{q_{\rm i}}$ involved in the transmission conditions. In the original formulation, one has a first order $3$-dimensional ODE (on $\delta$, $\dot \delta$ and $\av{q_{\rm i}}$), which is forced by $F_{\rm ext}$, $(R_1 \mfsw)_{\vert_{x=\pm\ell}}$, $\zeta_{\vert_{x=\pm\ell}}$ and $\frac{d^2}{dt^2}\zeta_{\vert_{x=\pm\ell}}$. In the augmented formulation, the ODE is of higher dimension, namely, it is a first order $7$-dimensional ODE (on $\delta$, $\dot \delta$,  $\av{q_{\rm i}}$, $\tzetapm$ and $\dtzetapm$), but it is forced only by $F_{\rm ext}$ and $(R_1 \mfsw)_{\vert_{x=\pm\ell}}$. These two quantities do not raise any difficulty since $F_{\rm ext}$ is a given external force and $(R_1 \mfsw)_{\vert_{x=\pm\ell}}$ can easily be computed numerically (see \S \ref{sectdiscR} below), contrary to the traces of $\zeta$ and $\dt^2 \zeta$ at the contact points that appear in the original formulation and that are very delicate to compute.
\end{remark}

\subsection{Transformation into an initial  value problem}\label{secttransfIVP}

We reformulate in this section the wave-structure {\it transmission} problem \eqref{transm1augm}-\eqref{ODEaugm} in the form of an {\it initial value} problem that is easier to handle from a numerical point of view. This formulation is the new augmented formulation (with additional variables $\tzetapm$) we shall base our numerical schemes on. In \cite{BeckLannes} (see also the lecture notes \cite{LannesBressanone}), the well-posedness of the standard formulation \eqref{transm1}-\eqref{transm4} is proved, and it could similarly be obtained for the augmented formulation; for the sake of conciseness, we do not give here such a result and just prove that both formulations have the same regular solutions, and that the additional variables $\tzetapm$ coincide with the traces $\zeta_{\vert_{x=\pm\ell}}$ under certain compatibility conditions on the initial data.   We use the following notation for the source term in the reformulated momentum equation, 
\begin{equation}\label{defS}
{\mathcal S}_\pm\big(\Theta, (R_1\mfsw)_\pm,F_{\rm ext}\big):={\mathcal G_1}\big(\Theta, (R_1\mfsw)_\pm,F_{\rm ext}\big)\mp \ell {\mathcal G_2}\big(\Theta, (R_1\mfsw)_\pm,F_{\rm ext}\big),
\end{equation}
where ${\mathcal G}_1$ and ${\mathcal G}_2$ denote the first two components of the mapping ${\mathcal G}$ in the right-hand side of the ODE \eqref{ODEaugm}.
We also recall that we denote $\cE^-=(-\infty,-\ell)$ and $\cE^+=(\ell,\infty)$ the two connected components of the fluid domain $\cE$. 
\begin{proposition}\label{propIVP}
Let $(\zeta,q)$ and $\Theta=\big(\av{q_{\rm i}},\dot\delta,\dtzetap,\dtzetam,\delta,\tzetap,\tzetam\big)^{\rm T}$ be a regular solution to the transmission problem \eqref{transm1augm}-\eqref{ODEaugm} with initial data $U=(\zeta^{\rm in},q^{\rm in})$ and $\Theta^{\rm in}$. Then $(\zeta,q)$ and $\Theta$ also solve the initial value problem
\begin{equation}\label{IVP1}
\begin{cases}
\dt \zeta+\dx q=0,\\
\dt q+\dx R_1\mfsw ={\mathcal S}_\pm\big(\Theta, (R_1\mfsw)_\pm,F_{\rm ext}\big)\exp(-\frac{1}{\kappa}\abs{x\mp \ell})
\end{cases}
\quad\mbox{ on }\cE^\pm,
\end{equation}
and 
\begin{equation}\label{IVP2}
\frac{d}{dt}\Theta={\mathcal G}\big(\Theta, (R_1\mfsw)_+,(R_1\mfsw)_-,F_{\rm ext}\big).
\end{equation}
The converse is true, provided that the initial data satisfy the compatibility conditions
\begin{equation}\label{CI1}
\av{q^{\rm in}}=\Theta^{\rm in}_1, \qquad \jump{q^{\rm in}}=-2\ell\Theta^{\rm in}_2.
\end{equation}
If moreover the initial data also satisfy
\begin{equation}\label{CI2}
 \zeta^{\rm in}_{\vert_{x=\ell}}=\Theta^{\rm in}_6, \qquad \zeta^{\rm in}_{\vert_{x=-\ell}}=\Theta^{\rm in}_7, 
 \qquad -(\dx q^{\rm in})_{\vert_{x=\ell}}=\Theta^{\rm in}_3, \qquad -(\dx q^{\rm in})_{\vert_{x=-\ell}}=\Theta^{\rm in}_4,
\end{equation}
then for all times, one has $\zeta_{\vert_{x=\pm\ell}}=\tzetapm$.
\end{proposition}
\begin{remark}\label{remODE}
The proposition deals with the most general situation to cover in a unified way all the situations considered in this article. It can be simplified in various cases, as shown in Appendix \ref{appcoeff}. For instance,
\begin{itemize}
\item When the object is freely floating, one takes $F_{\rm ext}\equiv 0$.
\item When the data and the object are symmetric with respect to the vertical axis $x=0$, then $\av{q_{\rm i}}=0$.
\item If the object is in forced motion, i.e. if $\delta\equiv\delta_{\rm forced}$ for some given function $\delta_{\rm forced}$, then the ODE \eqref{IVP2} can be reduced to a $5$-dimensional ODE on $\big(\av{q_{\rm i}},\dot{\underline{\zeta}}_+,\dot{\underline{\zeta}}_+,{\underline{\zeta}}_+,{\underline{\zeta}}_-\big)^{\rm T}$. Note that in this situation, an external force is needed to maintain the object fixed (the exact expression of this force is derived in Remark \ref{remcontrol}).
\end{itemize}
\end{remark}

\begin{proof}
Let us first prove the direct implication. Proceeding as in the proof of Proposition \ref{prop1}, we can rewrite the second equation of \eqref{transm1augm} on each component $\cE^\pm$ of $\cE$ under the form
$$
\dt q +\dx R_1\mfsw=\frac{d}{dt}(q_{\vert_{x=\pm \ell}})\exp(-\frac{1}{\kappa}\abs{x\mp \ell})\quad\mbox{ on }\quad \cE^\pm.
$$
From the transmission conditions \eqref{transm2augm}, we have $q_{\vert_{x=\pm\ell}}=\mp\ell\delta+\av{q_{\rm i}}$ so that the result follows from the observation that, owing to \eqref{ODEaugm}, one has 
$$
\frac{d}{dt}\av{q_{\rm i}}={\mathcal G}_1(\Theta, (R_1\mfsw)_\pm,F_{\rm ext})
\quad\mbox{ and }\quad \frac{d}{dt}\dot\delta={\mathcal G}_2(\Theta, (R_1\mfsw)_\pm,F_{\rm ext}).
$$
Conversely, if $(\zeta,q)$ solves \eqref{IVP1}, it suffices to apply $(1-\kappa^2\dx^2)$ to the second equation to show that $(\zeta,q)$ solves \eqref{transm1augm}. The equation \eqref{ODEaugm} on $\Theta$ is the same as \eqref{IVP2}, so that the only thing we need to prove is that the transmission conditions \eqref{transm2augm} hold. Taking the trace of the second equation of \eqref{IVP1} at the contact points and taking the average and the jump, we find that
$$
\frac{d}{dt}\av{q}={\mathcal G_1}(\Theta,(R_1\mfsw)_\pm,F_{\rm ext}) \quad\mbox{ and }\quad \frac{d}{dt}\jump{q}=-2\ell{\mathcal G}_2(\Theta,(R_1\mfsw)_\pm,F_{\rm ext}),
$$
or equivalently (from the definition of ${\mathcal G}$),
$$
\frac{d}{dt}\av{q}=\frac{d}{dt}\av{q_{\rm i}} \quad\mbox{ and }\quad \frac{d}{dt}\jump{q}=\frac{d}{dt}\big(-2\ell \dot\delta\big).
$$
This shows that the time derivative of the transmission conditions \eqref{transm2augm} are satisfied; the compatibility conditions \eqref{CI1} show moreover that the transmission condition condition is satisfied at $t=0$. It is therefore satisfied for all times.\\
For the last assertion, we can use Proposition \ref{prop1} to show that $\zeta_{\vert_{x=\pm\ell}}$ and $\tzetapm$ satisfy the same second order ODE in time. The additional condition \eqref{CI2} ensures that these initial data and the initial value of the first time derivative coincide (we also used the first equation of \eqref{transm1augm} to substitute $\frac{d}{dt}(\zeta_{\vert_{x=\pm\ell}})=-(\dx q)_{\vert_{x=\pm\ell}}$). They are therefore identical for all times.
\end{proof}

\section{Numerical schemes}\label{sectschemes}

We present in this section one first order and one second order numerical scheme for the resolution of the augmented formulations derived in this article. We explain these schemes for the general formulation \eqref{IVP1}-\eqref{IVP2}. 
We recall that these equations are conservation laws with a nonlocal flux and and an exponentially localized source term,
\begin{equation}\label{WE2_BNL}
\dt U+\dx \big({\mathfrak F}_\kappa(U) \big)={\mathcal S}_\pm(\Theta,(R_1\mfsw)_-,(R_1\mfsw)_+,F_{\rm ext}) 
{\mathfrak b}(x\mp\ell)
\qquad \mbox{ in }{\mathcal E}_\pm
\end{equation}
with $U=(\zeta,q)^T$ and where ${\mathfrak F}_\kappa(U)$ is the nonlocal flux given by
\begin{equation}\label{deftzetaNL}
{\mathfrak F}_\kappa(U)=\big(q,R_1 \mfsw \big)^T,
\end{equation}
while the source terms ${\mathcal S}_\pm$ are as in \eqref{defS} and ${\mathfrak b}$ is the shape of source term
\begin{equation}\label{sourceNL}
{\mathfrak b}(x) =\big( 0,  \exp(- \frac{\abs{x}}{\kappa}) \big)^{\rm T}.
\end{equation}
The quantity $\Theta$ is defined as $
\Theta=\big( \av{q_{\rm i}}, \dot \delta,\dtzetap,\dtzetam,\delta,\tzetap,\tzetam\big)^{\rm T}$
and solves a system of $7$ first order ODEs forced by $(R_1\mfsw)_+$, $(R_1\mfsw)_-$ and $F_{\rm ext}$,
\begin{equation}\label{ODE1ThetaNL}
\frac{d}{dt}\Theta={\mathcal G}\big(\Theta,(R_1\mfsw)_+,(R_1\mfsw)_- ,F_{\rm ext}\big).
\end{equation}
\begin{remark}
As explained in Remark \ref{remODE}, in some of the examples considered in this paper, the ODE \eqref{ODE1ThetaNL} can be reduced to a possibly lower dimensional ODE; we refer to Appendix \ref{appcoeff} where such simplifications are derived.
\end{remark}

\subsection{Notations}

We gather here the main notations used to write our numerical schemes. We first set our notations for the discretized quantities, and then explain how we define the discrete version of the nonlocal operator $R_1$ defined in \eqref{defR1}.
\subsubsection{Discretization}
    \begin{figure}[ht!]
 	\center
 	\includegraphics[height=1.3cm]{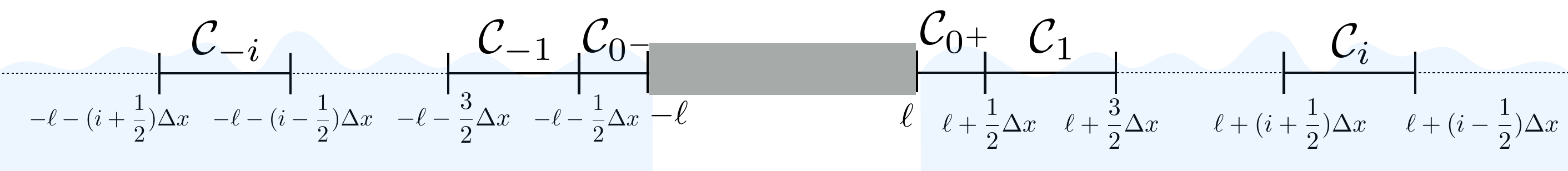}
 	\caption{Space discretzation}
 	\label{discret}
 \end{figure}
We denote by $\Delta x$ the mesh size and decompose the two components $\cE^-$ and $\cE^+$ of the exterior domain into a disjoint union of cells (see figure \ref{discret}),
$$
\cE^-=\Big(\bigcup_{i=-\infty}^{-1} {\mathcal C}_i \Big)\cup {\mathcal C}_{0^-}
\quad \mbox{ and }\quad 
\cE^+={\mathcal C}_{0^+}\cup \Big(\bigcup_{i=1}^\infty {\mathcal C}_i \Big)
$$
with
$$
{\mathcal C}_i=(x_{i-1/2},x_{i+1/2})\mbox{ if } i\neq 0
\quad\mbox{ and }\quad
{\mathcal C}_{0^-}=(x_{-1/2},-\ell), \qquad {\mathcal C}_{0^+}=(\ell,x_{1/2})
$$
and where
$$
x_{i+1/2}=-\ell+(i+1/2)\Delta x \quad \mbox{ if }i<0
\quad\mbox{ and }\quad
x_{i-1/2}=\ell+(i-1/2)\Delta x \quad \mbox{ if }i>0.
$$
\begin{remark}
Of course, the numerical domain is of finite size but we work with large enough domains so that the influence of the left and right boundaries of the numerical domain are not seen in the computations near the solid object. For the sake of clarity, we do not mention these boundaries in the presentation of the numerical scheme. 
\end{remark}

We also write  $\Delta t>0$  the time stepping and denote by
$$
U^n=U(n\Delta t) ,\qquad  \Theta^n= \Theta( n \Delta t )
\quad\mbox{ and }\quad F_{\rm ext}^n=F_{\rm ext}(n\Delta t)
$$
the values of $U=(\zeta,q)^{\rm T}$, of the $\RR^7$-valued vector $\Theta$  involved in the ODE \eqref{ODE1ThetaNL}, and of the external force $F_{\rm ext}$ at each time step. We further denote by $U^n_i$ ($i\in \big({\mathbb Z}\backslash\{0\}\big)\cup \{0^-,0^+\}$) the approximation of ${\mathcal U}^n$ in the middle of the cell ${\mathcal C}_i$ furnished by the numerical scheme.

\subsubsection{About the nonlocal operator $R_1$}\label{sectdiscR}

The equations \eqref{WE2_BNL}-\eqref{ODE1ThetaNL} involve the quantities $R_1\mfsw$ and $(R_1\mfsw)_\pm$, where we recall that $R_1$ is the inverse of $(1-\kappa^2\dx^2)$ on $\cE^-\cup \cE^+$ with Neumann boundary condition at $\pm\ell$, as defined in \eqref{defR1}, and that $(R_1\mfsw)_\pm$ stands for the trace of $R_1\mfsw$ at $\pm\ell$.

We keep the same notation  $R_1$ for the discrete inverse of the operator $(1-\kappa\dx ^2)$ with homogeneous Neumann condition at the boundary. We use here a standard centered second order finite difference approximation for the discretization of $\dx^2$. More precisely, if $F=(f_i)_{\abs{i}\geq 1}$, we denote by $R_1 F$ the vector $R_1 F=V$ where $V=(v_i)_{\abs{i}\geq 1}$ is given by the resolution of the equations
$$
v_i -\kappa^2\frac{v_{i+1}-2 v_i +v_{i-1}}{\Delta x^2}=f_i,\qquad \abs{i}\geq 2
$$
while, for $i=\pm1$ a second order discretization of the Neumann boundary condition leads to
$$
v_{-1} -\kappa^2 \frac{2}{3}\frac{v_{-2}- v_{-1}}{\Delta x^2}=f_{-1}
\quad\mbox{ and }\quad
v_1 -\kappa^2 \frac{2}{3}\frac{v_{2}- v_1}{\Delta x^2}=f_1.
$$
Similarly, we still denote by $({R}_1F)_\pm$ the discrete version of the traces ${R}_1F$ at the boundaries; they are naturally defined by the second order approximation
\begin{equation}\label{defR1pm}
\big({R}_1 F\big)_-=\frac{4}{3}v_{-1}-\frac{1}{3}v_{-2} \quad \mbox{ and }\quad \big({R}_1 F\big)_+=\frac{4}{3}v_{1}-\frac{1}{3}v_2.
\end{equation}

\subsection{A first order scheme}\label{sectLF}

We propose here an adaptation of the Lax-Friedrichs scheme for the conservation laws with nonlocal flux \eqref{WE2_BNL}. This scheme is an extension of the scheme used in \cite{LannesWeynans} for the numerical simulation of the Boussinesq equations with generating boundary condition (i.e. with data on $\zeta$ at the entrance of the numerical domain). It reads
\begin{equation}\label{LF1}
\frac{U_i^{n+1}-U_i^n}{\Delta t}+\frac{1}{\Delta x}\big( {\mathfrak F}^n_{\kappa,i+1/2}-{\mathfrak F}^n_{\kappa,i-1/2} \big)={\mathcal S}_\pm^n{\mathfrak b}_{i}, \qquad \pm i\geq 1, \quad n\geq 0,
\end{equation}
with 
\begin{equation}\label{LF2}
{\mathcal S}_\pm^n={\mathcal S}_\pm(\Theta^n,(R_1\mfsw^n)_+,(R_1\mfsw^n)_-,F_{\rm ext}^n)
\quad\mbox{ and }\quad
{\mathfrak b}_i={\mathfrak b}(i\Delta x)
\quad\mbox{ if } \pm i>0;
\end{equation}
the discrete flux correspond to the Lax-Friedrichs scheme,
\begin{equation}\label{LF3}
\begin{cases}
{\mathfrak F}_{\kappa,i+1/2}^n=\frac{1}{2}\big( {\mathfrak F}_{\kappa,i+1}^n+{\mathfrak F}_{\kappa,i}^n\big)-\frac{\Delta x}{2\Delta t}\big( U_{i+1}^n-U_{i}^n \big) &\mbox{if } i\leq -2,\\
{\mathfrak F}_{\kappa,i-1/2}^n=\frac{1}{2}\big( {\mathfrak F}_{\kappa,i}^n+{\mathfrak F}_{\kappa,i-1}^n\big)-\frac{\Delta x}{2\Delta t}\big( U_i^n-U_{i-1}^n \big) &\mbox{if } i\geq 2,
\end{cases}
\end{equation}
with the notations
$$
{\mathfrak F}_{\kappa,i}^n=\big( q_i^n, (R_1\mfsw^n)_i \big)^{\rm T}
\quad\mbox{ and }\quad
\mfsw^n=\mfsw(U^n);
$$
finally, for $i=\pm 1$, we must adapt \eqref{LF3} in the following way,
\begin{equation}\label{LF4}
\begin{cases}
{\mathfrak F}_{\kappa,-1/2}^n=\frac{1}{2}\big( {\mathfrak F}_{\kappa,0^-}^n+{\mathfrak F}_{\kappa,-1}^n\big)-\frac{\Delta x}{2\Delta t}\big( U_{0^-}^n-U_{-1}^n \big) \\
{\mathfrak F}_{\kappa,1/2}^n=\frac{1}{2}\big( {\mathfrak F}_{\kappa,1}^n+{\mathfrak F}_{\kappa,0^+}^n\big)-\frac{\Delta x}{2\Delta t}\big( U_1^n-U_{0^+}^n \big) 
\end{cases}
\end{equation}
with 
\begin{equation}\label{LF5}
{\mathfrak F}_{\kappa,0^\pm}^n=\big( q^n_{0^\pm}, (R_1\mfsw^n)_\pm \big)^{\rm T};
\end{equation}
the component $(R_1\mfsw^n)_\pm$ is computed according to \eqref{defR1pm}, but 
we still need to define  $q_{0^\pm}^n$. By definition $q^n_{0^\pm}$ is the approximation at time $n\Delta t$ of the trace of the discharge $q$ at $\pm \ell$. From the transmission conditions \eqref{transm2augm} of the continuous problem, we have $q_{\vert_{x=\pm\ell}}=\av{q_{\rm i}}\mp \ell \dot\delta$. Recalling also that $\av{q_{\rm i}}$ and $\dot\delta$ are respectively the first and second components of $\Theta$, this relation can be rewritten $q_{\vert_{x=\pm\ell}}=\Theta_1\mp\ell \Theta_2$. At the discrete level, this leads to the following definition for $q^n_{0^\pm}$,
\begin{equation}\label{LF6}
q^n_{0^\pm}=\Theta^n_1\mp\ell\Theta^n_2.
\end{equation}
The equation (\ref{ODE1ThetaNL}) is discretized with a first-order explicit Euler scheme:
\begin{equation}\label{eulerexplicit}
\frac{\Theta^{n+1} -\Theta^n }{\Delta t }=   {\mathcal G}\big(\Theta^n,(R_1\mfsw^n)_+,(R_1\mfsw^n)_- ,F_{\rm ext}^n\big).
\end{equation}

The equations \eqref{LF1}-\eqref{eulerexplicit} furnish an induction relation that allows to compute $U^{n+1}$ and $\Theta^{n+1}$ in terms of $U^n$ and $\Theta^n$. It need of course to be initiated with initial data that are taken of the form
\begin{equation}\label{LF7}
U_i^0=(\zeta^{\rm in}_i,q^{\rm in})_i^{\rm T} \qquad ( i\in \big({\mathbb Z}\backslash\{0\}\big)\cup \{0^-,0^+\}),
\end{equation} 
with $\zeta^{\rm in}$ and $q^{\rm in}$ describing the initial wave field in the exterior domain, and
\begin{equation}\label{LF8}
\Theta^0=\big(\av{q_{\rm i}}^{\rm in}, \delta^{(1)}, \underline{\zeta}_+^{(1)}, \underline{\zeta}_-^{(1)}, \delta^{(0)}, \underline{\zeta}_+^{(0)}, \underline{\zeta}_-^{(0)} \big)^{\rm T}
\end{equation}
satisfies the discrete version of the compatibility conditions of Proposition \ref{propIVP}, namely,
\begin{equation}\label{LF9}
\av{q^{\rm in}}=\av{q_{\rm i}}^{\rm in},\qquad \jump{q^{\rm in}}=-2\ell \delta^{(1)},\qquad \underline{\zeta}_\pm^{(0)}=\zeta_\pm^{\rm in}, \qquad
\underline{\zeta}^{(1)}_\pm=-(\dx q^{\rm in})_\pm.
\end{equation}

\subsection{A second order scheme}\label{sectMC}
We propose here an adaptation of the MacCormack scheme  for the conservation laws with nonlocal flux \eqref{WE2_BNL}, 
coupled with  a second-order Heun integration scheme for the system of 7 first-order ODEs (\ref{ODE1ThetaNL}). 
Both are predictor-corrector schemes.
We use the same notations as in the previous subsection and can decompose the scheme into four main steps:\\

- {\it Prediction step for the MacCormack scheme}. This reads
\begin{equation}\label{predictionMC}
\frac{U_i^{n,*}-U_i^n}{\Delta t}+\frac{1}{\Delta x}\big( {\mathfrak F}^n_{\kappa,i}-{\mathfrak F}^n_{\kappa,i-1} \big)={\mathcal S}_+^n{\mathfrak b}_i, \qquad   i > 1, \quad n\geq 0,
\end{equation}
with  $U_i^{n,*}=(\zeta_i^{n,*},q_i^{n,*})^{\rm T}$. We use a symmetric scheme with respect to $x=0$ so that, for negative values of $i$, we use a forward rather than backward derivative for the flux,
\begin{equation}
\frac{U_i^{n,*}-U_i^n}{\Delta t}+\frac{1}{\Delta x}\big( {\mathfrak F}^n_{\kappa,i+1}-{\mathfrak F}^n_{\kappa,i} \big)={\mathcal S}_-^n{\mathfrak b}_i, \qquad  i < -1,
\quad n\geq 0,
\end{equation}
For $i = 1$ and $i=-1$ it reads
\begin{align}
\frac{U_1^{n,*}-U_1^n}{\Delta t}+\frac{1}{\Delta x}\big( {\mathfrak F}^n_{\kappa,1}-{\mathfrak F}^n_{\kappa,0^+} \big)&={\mathcal S}_+^n{\mathfrak b}_1,\\
\frac{U_{-1}^{n,*}-U_{-1}^n}{\Delta t}+\frac{1}{\Delta x}\big( {\mathfrak F}^n_{\kappa,0^-}-{\mathfrak F}^n_{\kappa,-1} \big)&={\mathcal S}_-^n{\mathfrak b}_{-1},  
\end{align}
for $n\geq 0$ and with ${\mathfrak F}_{\kappa,0^\pm}^n$ as in \eqref{LF5}. \\

- {\it Prediction step for the Heun scheme}. This step is  similar to a first-order explicit Euler scheme,
\begin{equation}\label{predictionHeun}
\frac{\Theta^{n,*} -  \Theta^n }{\Delta t} =  {\mathcal G}\big(\Theta^n,(R_1\mfsw^n)_+,(R_1\mfsw^n)_- ,F_{\rm ext}^n\big).
\end{equation}

- {\it Corrector step for the MacCormack scheme}. With the quantities computed in the previous steps,  we define
\begin{equation}
\mfsw^{n,*}=\mfsw(U^{n,*}) \quad  \mbox{ and } \quad q^{n,*}_{0^\pm}=\Theta^{n,*}_1\mp\ell\Theta^{n,*}_2
\end{equation}
as well  as an intermediate non-local flux and an intermediate source term,
\begin{align}\label{sourcestar}
{\mathfrak F}_{\kappa,i}^{n,*} &=\big( q_i^{n,*}, (R_1\mfsw^{n,*})_i \big)^{\rm T}  \qquad  \abs{i}\geq 1, \quad n\geq 0, \\
{\mathcal S}_\pm^{n,*} &={\mathcal S}_\pm(\Theta^{n,*},(R_1\mfsw^{n,*})_+,(R_1\mfsw^{n,*})_-,F_{\rm ext}^{n}) \quad n\geq 0.
\end{align}
The correction step  for the MacCormack scheme then reads
\begin{equation}\label{correctionMC}
\frac{U_i^{n+1}-U_i^n}{\Delta t}+\frac{  {\mathfrak F}^n_{\kappa,i} -{\mathfrak F}^n_{\kappa,i-1}+ {\mathfrak F}^{n,*}_{\kappa,i+1}  -{\mathfrak F}^{n,*}_{\kappa,i}  }{ 2\Delta x} = \frac{ {\mathcal S}_+^n  + {\mathcal S}_+^{n,*}}{2}{\mathfrak b}_i\qquad i\geq 1,
\end{equation}
for $ n\geq 0$. Here again, we take a symmetric scheme so that for $i\leq -1$, we take a forward difference of ${\mathfrak F}^n$ and a backward difference of ${\mathfrak F}^{n,*}$,
\begin{equation}\label{correctionMCneg}
\frac{U_i^{n+1}-U_i^n}{\Delta t}+\frac{  {\mathfrak F}^n_{\kappa,i+1} -{\mathfrak F}^n_{\kappa,i}+ {\mathfrak F}^{n,*}_{\kappa,i}  -{\mathfrak F}^{n,*}_{\kappa,i-1}  }{ 2\Delta x} = \frac{ {\mathcal S}_-^n  + {\mathcal S}_-^{n,*}}{2}{\mathfrak b}_i\qquad i\leq -1;
\end{equation}
 in particular, there is no need to defined boundary values ${\mathfrak F}_{\kappa,0^\pm}^{n,*}$, of the intermediate flux.\\
 
 -{\it Correction step for the Heun scheme}. This reads, for $n\geq 0$,
\begin{equation}\label{correctionHeun}
\frac{\Theta^{n+1} -  \Theta^n }{\Delta t}  =  \frac{  {\mathcal G}\big(\Theta^n,(R_1\mfsw^n)_\pm ,F_{\rm ext}^n\big) + {\mathcal G}\big(\Theta^{n,*},(R_1\mfsw^{n,*})_\pm ,F_{\rm ext}^{n+1}\big) }{2} .
\end{equation}

The initial data have the same form as for the first order scheme described in the previous subsection.

\section{Numerical simulations}\label{sectnumsim}

We have seen in \S \ref{sectreduc} that the wave-structure interaction problem under consideration in this paper can be reduced to a transmission problem potentially coupled to two forced ODEs for the vertical displacement $\delta$ of the object and the mean discharge $\av{q_{\rm i}}$ under the object. 

We first consider in \S \ref{sectWG} a situation where this coupling is absent. This corresponds to the case where a wave is generated in a wave tank by moving the object vertically with a prescribed motion. This example is of particular interest since it provides an efficient way to generate waves for the Boussinesq equations at the entrance of a numerical domain if we have at our disposal time series of the horizontal discharge at the boundary, hereby extending the result of \cite{LannesWeynans} where data on the surface elevation were used.

We then consider in \S \ref{sectRTE} the return to equilibrium problem (also called decay test or drop test by engineers) which consists in releasing an object from an out of equilibrium position and to observe its oscillations. These examples involve the coupling of the transmission problem with the ODE on $\delta$. In the linear case, we are able to derive exact explicit solutions that we compute to check the numerical convergence of our scheme; the nonlinear case is then investigated and the importance of the dispersive effects pointed out by comparing with simulations based on the nonlinear shallow water equations instead of the Boussinesq system.

We then investigate in \S \ref{sectfixed} a configuration where the transmission problem is coupled to the interior discharge $\av{q_{\rm i}}$, namely, the interaction of waves with a fixed partially immersed object. Here again, we derive an explicit exact solution in the linear case that we use to validate that this coupling is also of second order. The nonlinear case is then considered.

Finally, a configuration involving the most general coupling (with both $\delta$ and $\av{q_{\rm i}}$ is considered in \S \ref{sectfreely}; it consists in the interaction of a solitary wave with an object freely floating in the vertical direction.

\subsection{Wave generation}\label{sectWG}

The first physical configuration we consider consists in creating waves in a fluid initially at rest by moving up and down a partially immersed object. By symmetry, it is enough to consider the waves in the right component $\cE^+=(\ell,\infty)$ of the fluid domain. As shown in \S \ref{appfixforced} of Appendix \ref{appcoeff}, the mathematical formulation of this problem is a particular case of the following initial boundary value problem with boundary condition on the discharge $q$, namely, 
\begin{equation}\label{pb1-1}
\begin{cases}
\dt \zeta+\dx q=0,\\
(1-\kappa^2\dx^2)\dt q +\dx \mfsw=0,
\end{cases}
t>0, \quad x>\ell
\end{equation}
with $\mfsw$ as in \eqref{defmfsw} and with boundary condition
\begin{equation}\label{pb1-2}
q_{\vert_{x=\ell}}(t)=g(t), \qquad t>0
\end{equation}
and initial condition
\begin{equation}\label{pb1-3}
(\zeta,q)_{\vert_{t=0}}(x)=(\zeta^{\rm in},q^{\rm in})(x), \qquad x>\ell,
\end{equation}
and where $g$, $\zeta^{\rm in}$ and $q^{\rm in}$ are some given functions satisfying the compatibility condition
\begin{equation}\label{compat}
q^{\rm in}(x=0)=g(t=0),
\end{equation}
which is obviously necessary to obtain solutions that are continuous at the origin in time and space. This problem is somehow symmetric to the one considered in \cite{LannesWeynans} where a boundary condition on $\zeta$ rather than $q$ was considered and where a first order scheme was proposed. 
\begin{remark}
For the wave generation problem, one has $(\zeta^{\rm in},q^{\rm in})=(0,0)$ and $g(t)=-\ell \dot\delta_{\rm forced}$, where $\delta_{\rm forced}$ is the prescribed vertical displacement of the center of mass of the object.
\end{remark}
Contrary to the other physical configurations we consider in this article, the wave generation problem (or more generally, the initial boundary value problem \eqref{pb1-1}-\eqref{compat}) does not require the resolution of an ODE to determine the boundary data on the discharge. The formulation as an initial value problem given in Proposition \ref{propIVP} then reduces to
\begin{equation}\label{pb1reduced}
\begin{cases}
\dt \zeta+\dx q=0,\\
\dt q+\dx R_1\mfsw={\mathcal S}_+(t) \exp(- \frac{x-\ell}{\kappa})
\end{cases}
\quad\mbox{ with }\quad {\mathcal S}_+(t)=\dot g(t),
\end{equation}
for $x>\ell$ and with initial condition \eqref{pb1-3} satisfying \eqref{compat}. \\
The numerical scheme presented in \S \ref{sectLF} and \S \ref{sectMC} can be simplified by skipping the second and fourth step related to the Heun scheme, and by taking simply for the first-order scheme:
$$ {\mathcal S}^n=\frac{g^{n+1}-g^n}{\Delta t},$$ 
and for the second-order scheme:
 $$ {\mathcal S}^n=\frac{g^{n+1}-g^{n-1}}{2 \Delta t} \quad \mbox{ and } \quad  {\mathcal S}^{n,*}=\frac{g^{n+2}-g^{n}}{2 \Delta t}. $$

The wave generation problem gives us the opportunity to validate our numerical code with a nonlinear case. The Boussinesq-Abbott equations admits solitary waves solutions of the form
$$
(\zeta,q)(t,x)=\big(\zeta_c(x-x_0-ct), c\zeta_c(x-x_0-ct)\big),
$$
where $c>0$, $x_0\in \RR$ and $\zeta_c$ is a smooth, even and fastly decaying function. These solutions can be used to test the precision of the code. For the Boussinesq-Abbott equations, there is no explicit formula for $\zeta_c$ and it is determined by the resolution of a nonlinear second order ODE, namely,
\begin{equation}\label{EDO-solition}
c^2\kappa^2\zeta_c''-c^2 \frac{\zeta}{1+\eps\zeta}+\frac{\eps^2\zeta^2+2\eps\zeta}{2\eps}=0,
\end{equation}
with 
$$
c^2=\frac{\eps}{6}\frac{3\zeta_{\rm max}^2+\eps\zeta_{\rm max}^3}{\zeta_{\rm max}-\frac{\ln(1+\eps\zeta_{\rm max})}{\eps}}
$$
(see for instance \cite{LannesWeynans} for more details on the computations); these formula furnish a family of solitary waves parametrized by their maximal amplitude $\zeta_{\rm max}$. Solving the above ODE with a standard high precision ODE solver provides us a solution to \eqref{pb1-1} that we use to assess the precision of the numerical solution obtained with our numerical scheme for \eqref{pb1reduced} with discharge boundary data $g(t)=c\zeta_c(0-x_0-ct)$ and initial data $(\zeta^{\rm in},q^{\rm in})=(\zeta_c(x-x_0-0), c\zeta_c(x-x_0-0)\big)$.

\begin{remark}
For very fine meshes, spurious oscillations may appear. 
These oscillations are reminiscent of the oscillations that appear when using dispersive schemes (such as the Lax-Wendroff or MacCormack schemes) to simulate shock waves. Flux-limiters methods are typically used to control this phenomenon \cite{Leveque}. Here, these oscillations are created at the boundary, whose position is fixed, and we use a very simple efficient method consisting in adding an artificial viscosity on a finite number $n_0$ of cells near the boundary. More precisely, in the right-component of the fluid domain (the left component is treated symmetrically) we add the following term in the right-hand side of the first component of \eqref{correctionMCneg},
$$
\nu \frac{\Delta x}{\Delta t} \big( \zeta^n_{i+1}-2\zeta^n_i+\zeta^n_{i-1}\big) \qquad 1\leq i\leq n_0
$$
with $\nu>0$ a fixed coefficient, that we take equal to $2.136$. This corresponds to an artificial viscosity $\nu \frac{(\Delta x)^3}{\Delta t} \dx^2 \zeta$; for a fixed ration $\Delta x/ \Delta t$, this viscosity is of order $2$ and therefore does not alter the overall second order of the MacCormack scheme.
\end{remark}

We  choose  for this test $\zeta_{{\rm max}}=1$ and $3\kappa^2= \eps = 0.3$. Once $c$ is computed, we solve the differential equation \eqref{EDO-solition}  with a high order numerical method in order to obtain our reference solution.
 The size of the computational domain is $L = 6$.
The space step is computed as $\Delta_x = L/N$, with $N = 200,240,300,400$. 
We take a constant time step $\Delta_t = 0.8 \, \Delta_x$.
The maximum of the soliton is initially located on the left of the computational domain, at $x = -L/2$, so that the initial datum in the small domain is almost zero, and then the soliton propagates inside it.
The numerical results at final time $T_f = 20$ are presented for both schemes on Figures \ref{soliton_LF} and \ref{soliton_MC}, showing respectively a first-order and a second-order convergence.

These results correspond to a final time where the soliton has completely entered the computational domain, so that the influence of the dispersive boundary layer due to the generating condition at the left side of the domain is nearly zero. 
However, if one would choose a final time where the soliton is still entering the computational domain, then one would notice that the error for the variable $\zeta$ is only first-order in the vicinity of the left boundary of the computational domain, while it is still second-order for the variable $q$. This first-order error is probably due to a lack of accuracy in the numerical evaluation of the spatial derivative of $q$ near the left boundary.

 \begin{figure}[!ht]
\centering
\begin{tabular}{cc}
   \includegraphics[height=50mm]{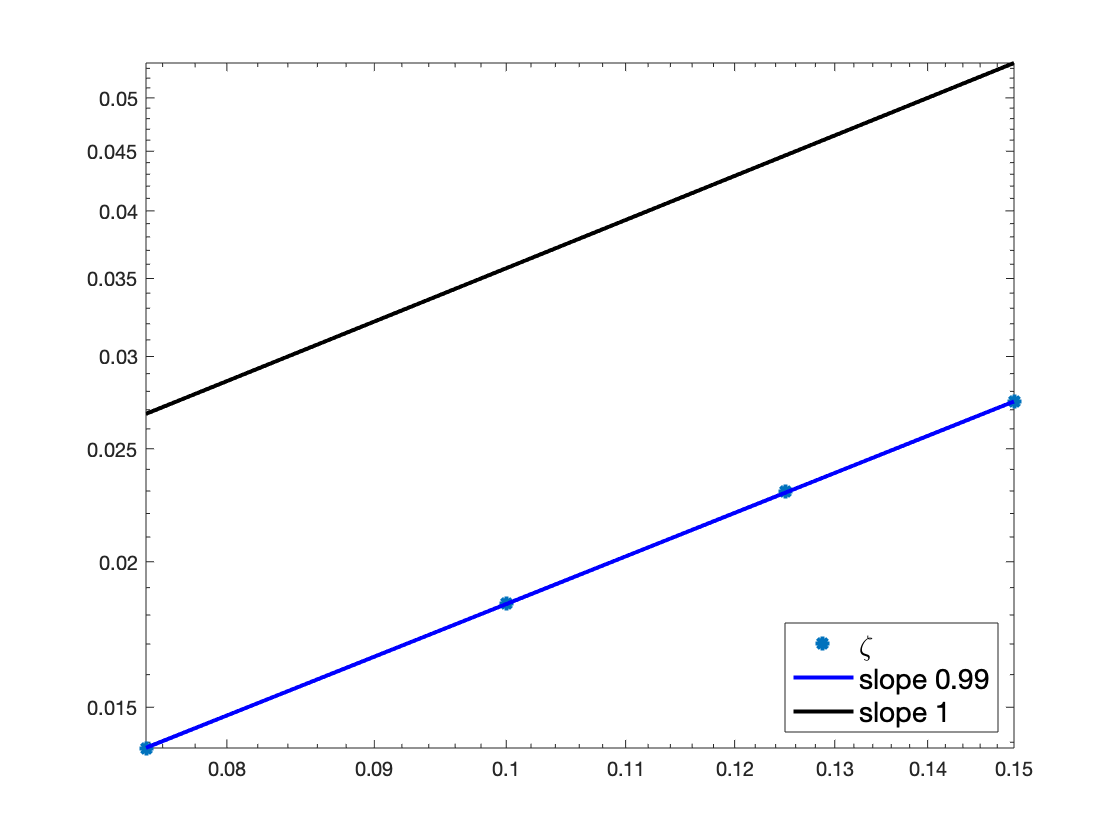}  & \includegraphics[height=50mm]{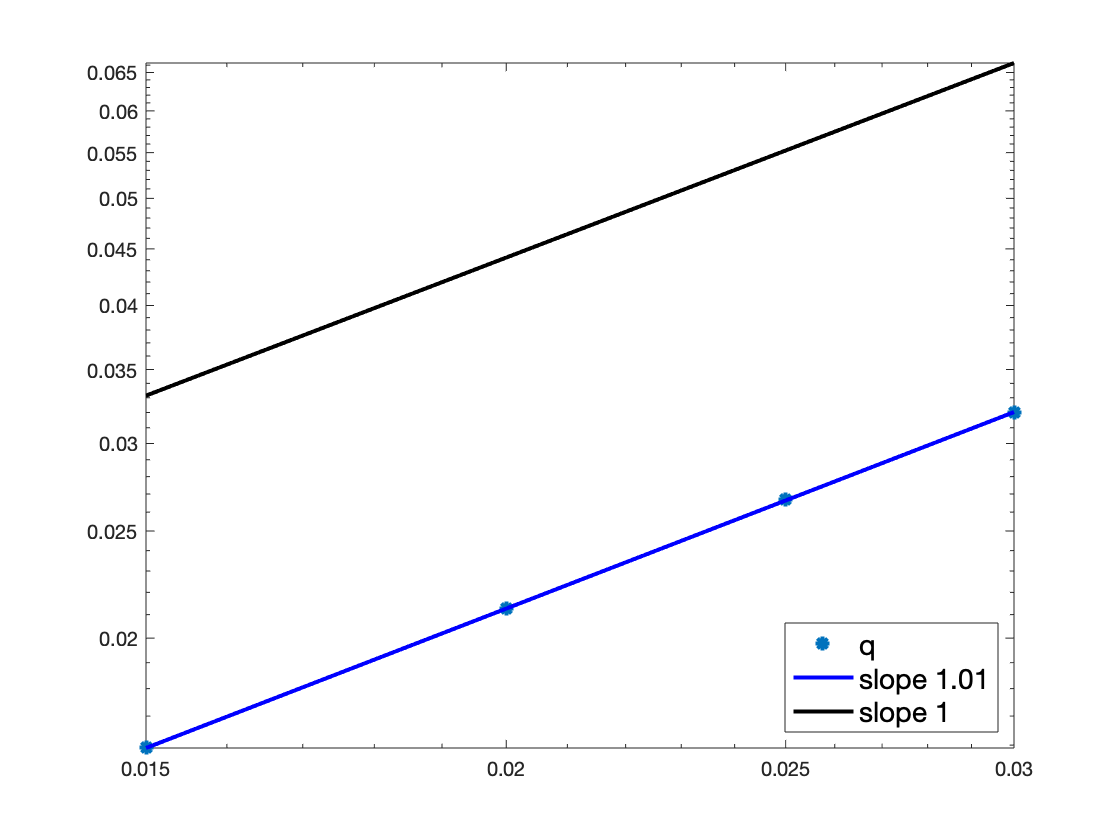}
\end{tabular}
  \caption{Solitary wave with generating condition on $q$, $L^{\infty}$ error for $3\kappa^2 = 0.3$ with first-order scheme, left: $\zeta$, right: $q$.}
  \label{soliton_LF}
     \end{figure}

 \begin{figure}[!ht]
\centering
\begin{tabular}{cc}
   \includegraphics[height=50mm]{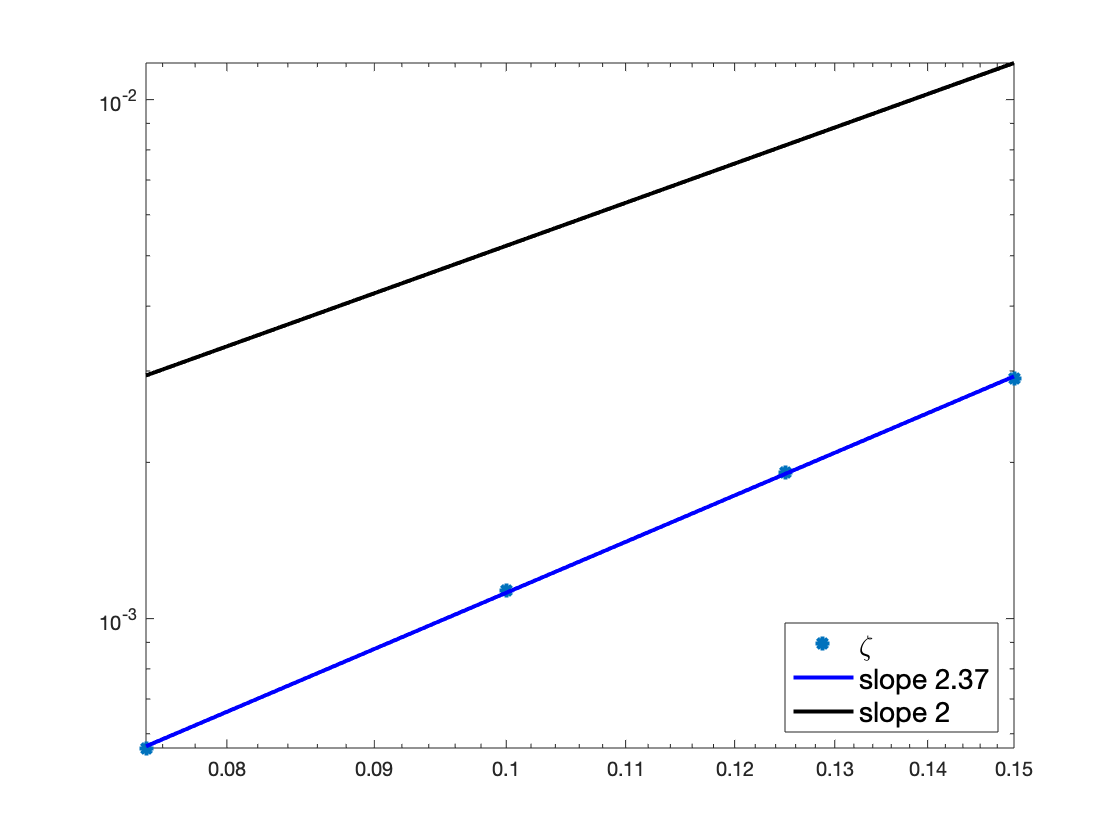}  & \includegraphics[height=50mm]{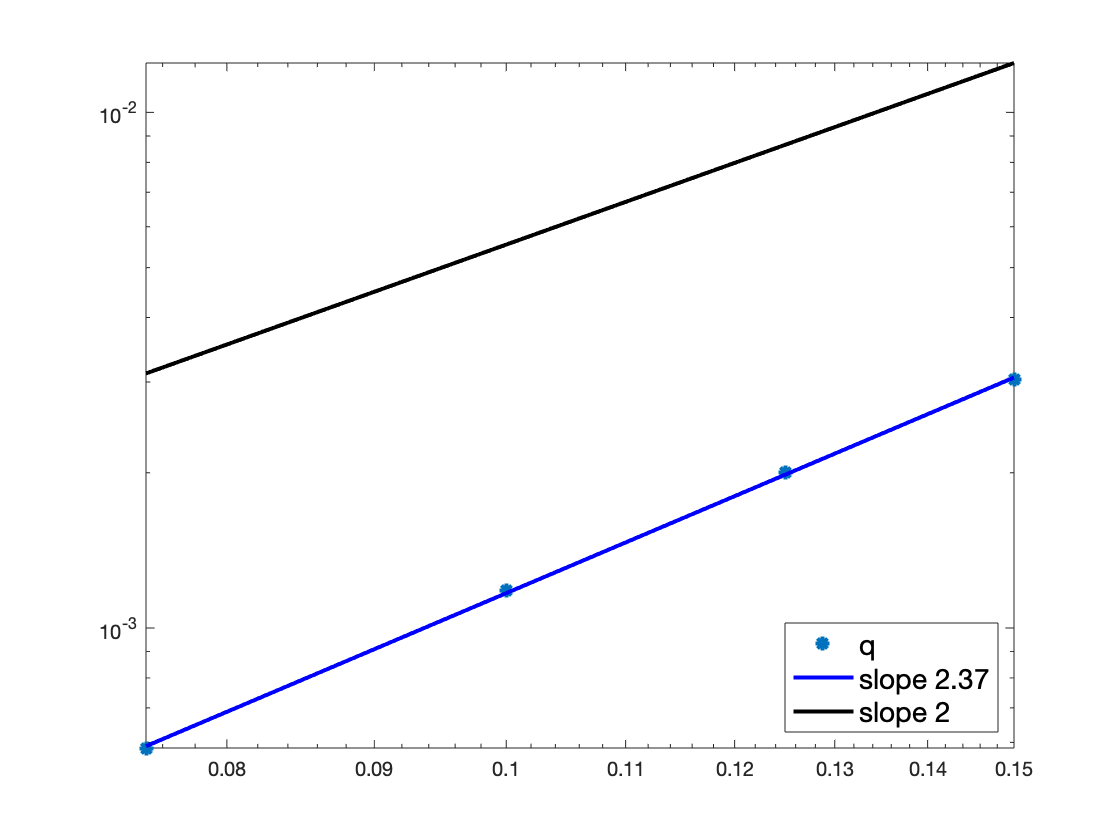}
\end{tabular}
  \caption{Solitary wave with generating condition on $q$, $L^{\infty}$ error for $3\kappa^2 = 0.3$ with second-order scheme, left: $\zeta$, right: $q$.}
  \label{soliton_MC}
     \end{figure}

\subsection{Return to equilibrium}\label{sectRTE}

We consider here the return to equilibrium problem (also called decay test), which consists in dropping the floating object from an out of equilibrium position and to let it oscillate vertically and stabilize towards its equilibrium position. This is a problem of practical importance because it is used by engineers to characterize some buoyancy properties of the solid, and theoretically because it leads to simpler equations than the general wave-structure equations. For instance, in the nonlinear non dispersive case ($\eps\neq 0$, $\kappa=0$), it is possible to show that the dimensionless vertical displacement $\eps\delta$ of the solid with respect to its equilibrium position  is fully described by a second order nonlinear scalar ODE \cite{Lannes_float,BeckLannes} and that in the linear dispersive case ($\eps=0$, $\kappa\neq 0$) it is governed by a second order linear integro-differential equation \cite{BeckLannes}. In the nondispersive case, similar equations have also been derived  in the presence of viscosity \cite{Maity} in the linear case, as well as in the $2D$ radial and partially linear case \cite{Bocchi2}. In  the presence of nonlinear {\it and} dispersive effects ($\eps\neq 0$, $\kappa\neq 0$), it does not seem possible to derive such a simple equation for the motion of the solid and the wave-structure equations must therefore be solved. 

As for the wave generation problem, there is a symmetry in this problem which allows to consider only the right part $\cE^+$ of the fluid domain and the governing wave-structure interaction equations reduce to an initial boundary value problem of the form \eqref{pb1-1} with $g=-\ell\dot\delta$. The difference is that the vertical displacement $\delta$ is no longer a given function $\delta$ but is found through the resolution of Newton's equation (see \eqref{transm4} for its general expression). Since this equation involves the trace of $\zeta$ at the contact point, we have to work with the augmented formulation provided by Proposition \ref{propIVP}. Since in this particular case, one has $F_{\rm ext}\equiv 0$, $\av{q_{\rm i}}\equiv 0$ and $\zeta_+=\zeta_-$, the $7$-dimensional ODE on $\Theta$ can be simplified into a simpler $4$-dimensional ODE (see \S \ref{appfixforced} in Appendix \ref{appcoeff}). The interest of this test case is that, since the interior discharge identically vanishes,  it allows us to investigate specifically the coupling between the waves and the vertical displacement of the object. We first consider in \S \ref{sectcvrte} the linear case for which explicit solutions exist and can be used to investigate the precision of the code, and then show in \S \ref{sectNLrte} some simulations in the nonlinear case.

\subsubsection{Convergence error in the linear case}\label{sectcvrte}
We first consider the linear case ($\eps=0$) since in the case, it was shown in \cite{BeckLannes} that the evolution of $\delta$ can be found by solving a linear second order integro-differential equation, namely,
$$
\big(\tau_\kappa(0)^2+\ell\kappa\big)\ddot\delta+\ell {\mathcal K}^1_\kappa * \dot\delta+\delta=0,
$$
with initial conditions $\delta(0)=\delta_0$ and $\dot\delta(0)=0$ and where $\tau_\kappa(\cdot)$ is defined in Appendix \ref{appcoeff} and the kernel ${\mathcal K}_\kappa^1$ is given in terms of the first Bessel function $J_1$ by the relation
$$
{\mathcal K}^1_\kappa(t)=\frac{1}{t}J_1\big(\frac{t}{\kappa}\big).
$$
The solution of this integro-differential equation is given explicitly by taking the Laplace transform (denoted with a hat),
\begin{equation}\label{exRTE}
\widehat{\delta}(s)=\frac{\tau_\kappa(0)^2s+\ell\sqrt{1+\kappa^2 s^2}}{\tau_\kappa(0)^2 s^2+s\ell \sqrt{1+\kappa^2 s^2}+1}\delta_0, \qquad s\in {\mathbb C}_0, 
\end{equation}
where ${\mathbb C}_0$ is the half-plane of complex numbers $s$ such that $\Re s>0$. The vertical displacement deduced from the exact formula \eqref{exRTE}, and denoted $\delta_{\rm exact}$ is compared with the surface elevation $\delta$ found by solving the wave-structure equations using the numerical schemes presented in \S \ref{sectschemes}. In order to discard possible numerical errors in the computation of the inverse Laplace transform one has to apply to \eqref{exRTE} two different inversion methods (the Euler and Talbot methods \cite{talbot_euler}); we impose that they match up to $10^{-4}$ terms to consider the solution provided as relevant to be considered as an exact solution for our convergence studies.\\

In our numerical tests we chose $3\kappa^2 = 0.3$ or $3\kappa^2 = 0.1$, $h_{eq} = 1-0.3$, $l = 4$ and the size of the computational domain $L = 30$.
 The space steps $\Delta_x$ were computed as $\Delta_x = (L-l)/(N+1)$, with $N=300,400,500,600$ for the first-order scheme and $N = 60,120,240,320$ for the second-order scheme.
 The time step was computed as $\Delta_t = 0.9 \Delta_x$.
The numerical results at final time $T_f = 15$ computed with the first-order and the second-order schemes show respectively a first-order convergence, see Figure \ref{return_linear_LF} and a second-order convergence, see Figure \ref{return_linear_MC}.
On Figure \ref{return_linear_compaLFMC} we compare the numerical results for the two schemes for $3\kappa^2 = 0.3$ and $N = 100$, showing evidence that it is advantageous to use the second-order scheme.

 \begin{figure}[!ht]
\centering
\begin{tabular}{cc}
   \includegraphics[height=50mm]{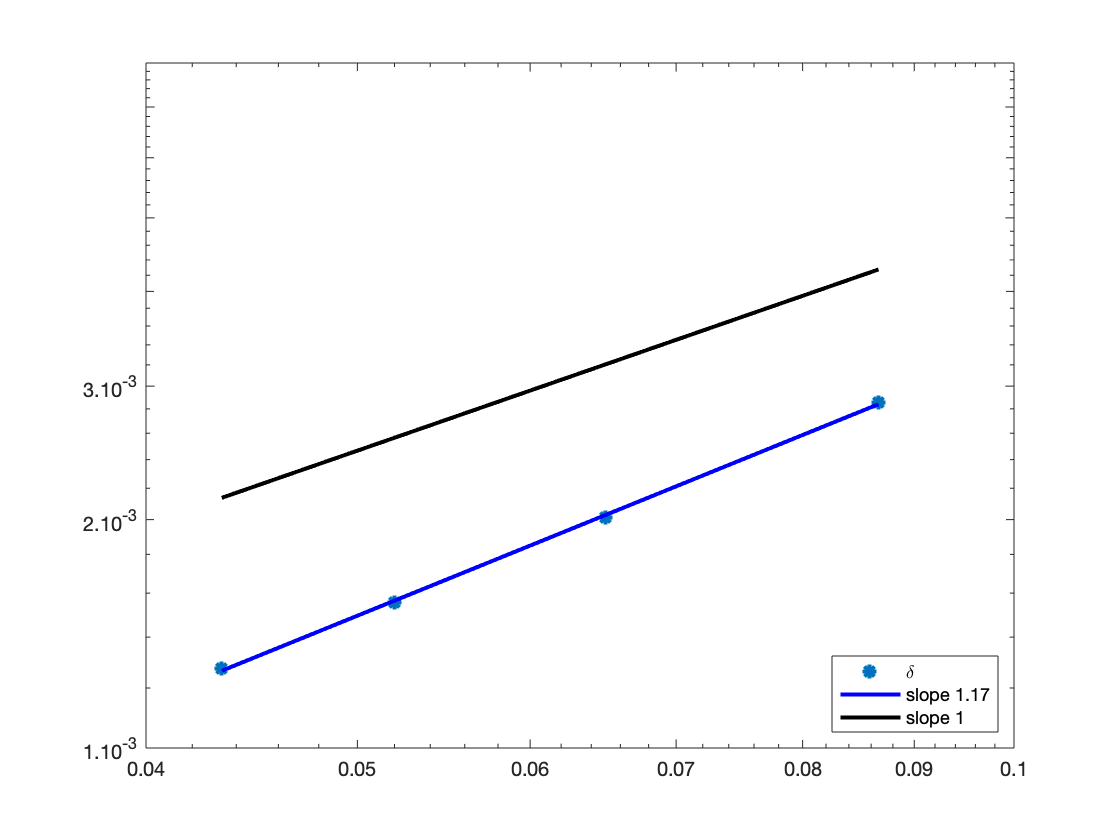}  & \includegraphics[height=50mm]{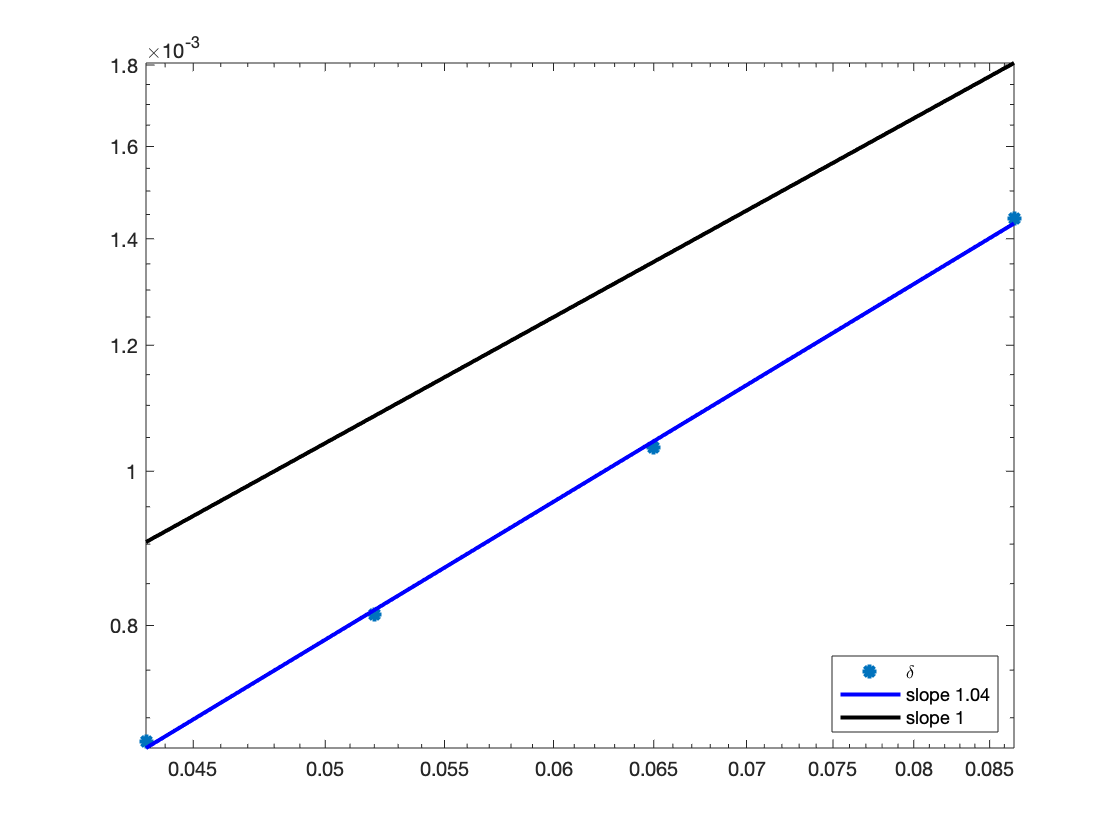}
\end{tabular}
  \caption{Return to equilibrium, linear case: convergence results for $\delta$ with the first-order scheme, $3\kappa^2 = 0.1$ (left) and $3\kappa^2 = 0.3$ (right).}
  \label{return_linear_LF}
     \end{figure}

 \begin{figure}[!ht]
\centering
\begin{tabular}{cc}
   \includegraphics[height=50mm]{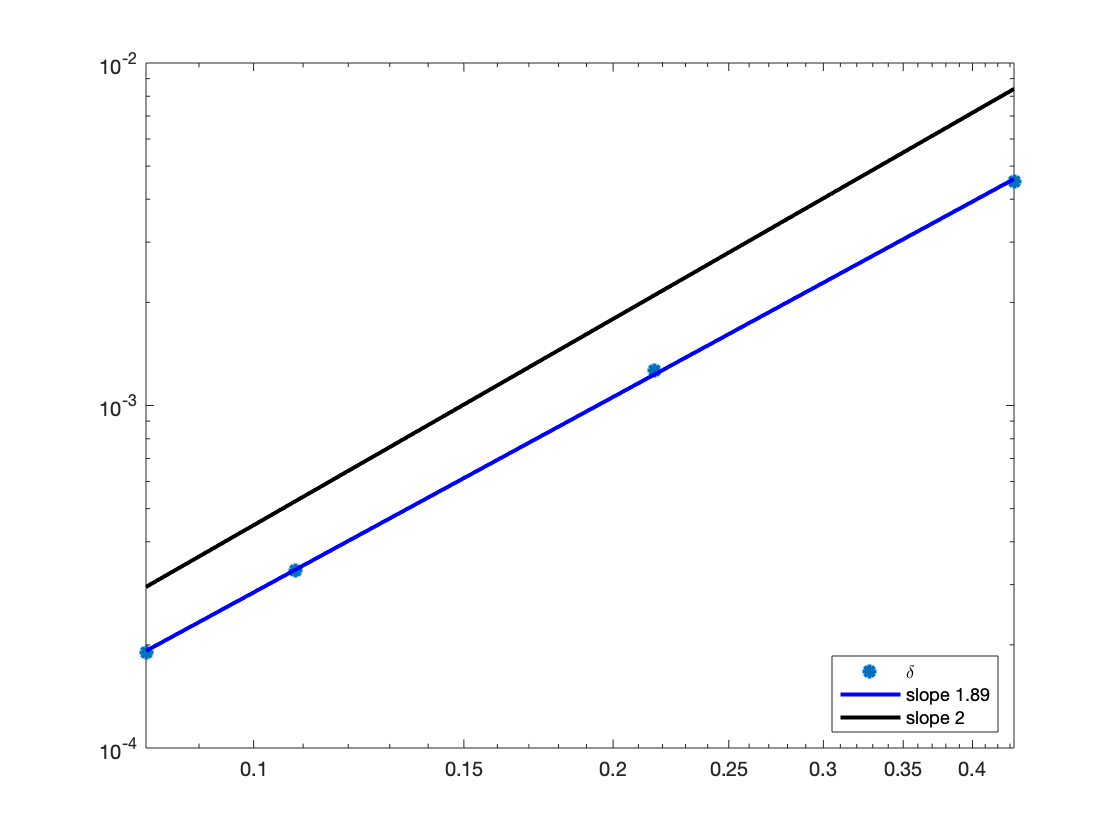}  & \includegraphics[height=50mm]{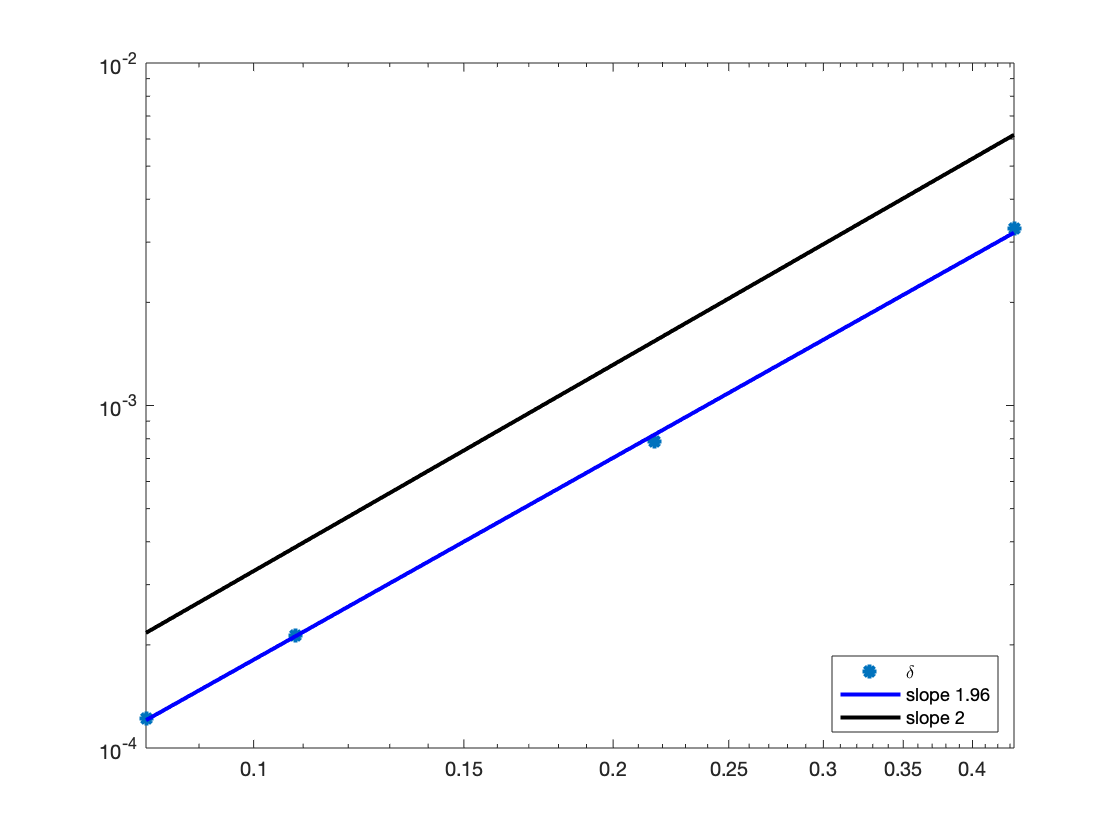}
\end{tabular}
  \caption{Return to equilibrium, linear case: convergence results for $\delta$ with the second-order-order scheme, $3\kappa^2 = 0.1$ (left) and $3\kappa^2 = 0.3$ (right).}
  \label{return_linear_MC}
     \end{figure}

 \begin{figure}[!ht]
\centering
\begin{tabular}{cc}
   \includegraphics[height=50mm]{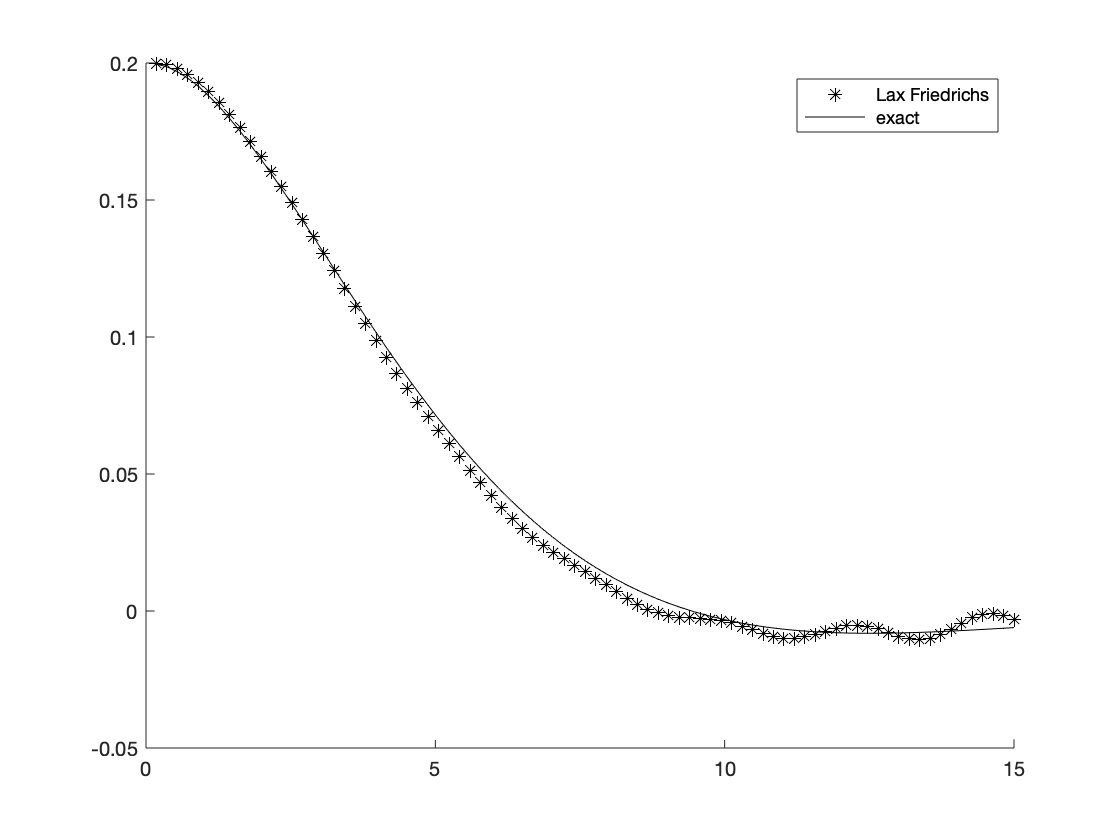}  & \includegraphics[height=50mm]{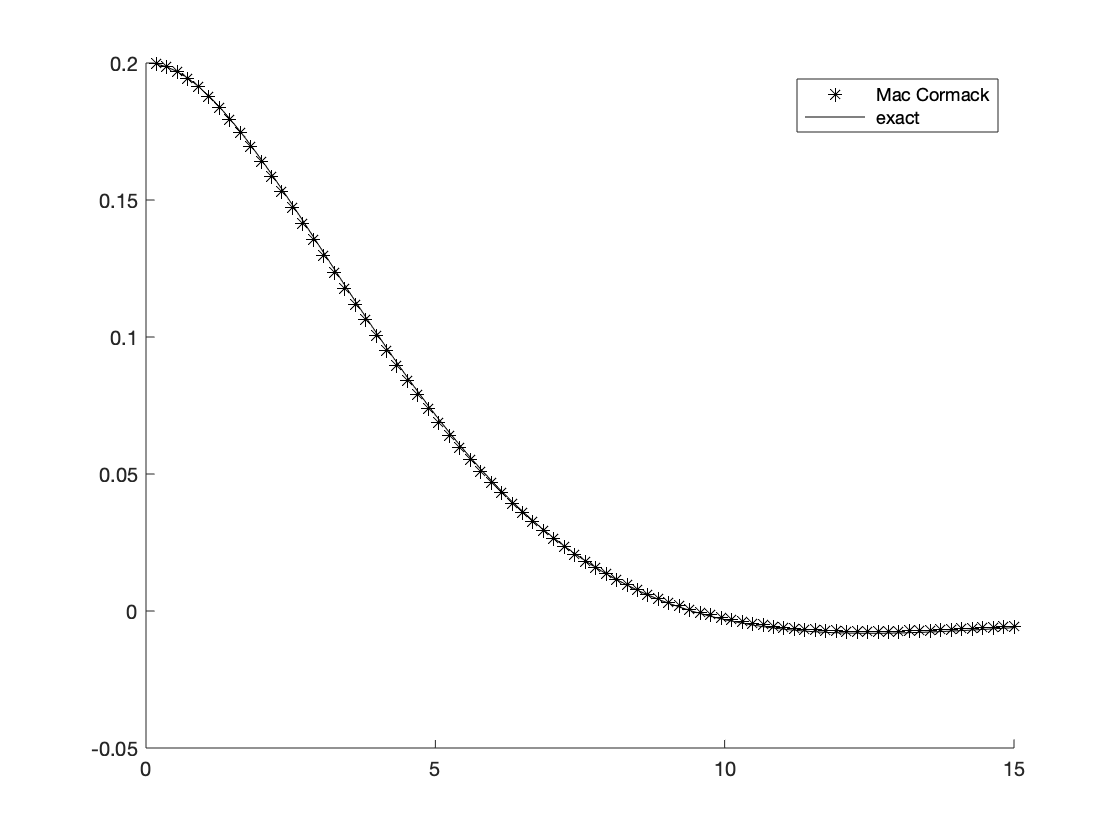}
\end{tabular}
  \caption{Return to equilibrium, linear case: temporal evolution of $\delta$, comparison between the first-order scheme and the second-order scheme, $3\kappa^2 = 0.3$ and $N = 100$.}
  \label{return_linear_compaLFMC}
     \end{figure}

\subsubsection{The nonlinear case}\label{sectNLrte}

We do not have any exact solution to compare with in the nonlinear case, and we therefore use mesh-convergence to study the order of our schemes. The reference solution is computed with a very refined mesh: $N = 2400$.
 We chose $\epsilon = 0.3$, $3\kappa^2 = 0.3$ or $3\kappa^2 = 0.1$, $h_{eq} = 1-0.3$, $l = 4$ and the size of the computational domain $L = 30$.
 The space steps $\Delta_x$ were computed as $\Delta_x = (L-l)/(N+1)$, with $N=160,200,240,300,400$ for the first-order scheme and $N = 120,160,200$ for the second-order scheme.
  The meshes are defined so that the points of the coarse meshes always coincide with the points of the very refined mesh of the reference solution.
 The time step was computed as $\Delta_t = 0.7 \Delta_x$.
The numerical results at final time $T_f = 15$ computed with the first-order and the second-order schemes show respectively a first-order convergence, see Figure \ref{return_nonlinear_LF} and a second-order convergence, see Figure \ref{return_nonlinear_MC}.

We perform another test to study qualitatively the influence of the dispersion on the nonlinear decay test. We compare the trajectories obtained for different values of $\kappa$ with the trajectory obtained in the non dispersive case ($\kappa=0$).
In the latter case, it was shown in \cite{BeckLannes} that the evolution of $\delta$ can be found, under some smallness assumptions, by solving a second order differential equation of the form
\begin{equation}\label{ODESW2bis}
{{\tau_0(\eps\delta)^2}}\ddot\delta+\ell \dot\delta+\delta +\eps B(\eps\delta) \dot\delta^2 =0,
\end{equation}
where $B(\cdot)$ is a smooth function whose exact expression can be found in Corollary 4.3  of \cite{BeckLannes}. This nondispersive solution is used as reference to illustrate the contribution of the dispersive terms.
On Figure \ref{return_nonlinear_compaLFMC} we compare this solution with the numerical results for $\eps=0.3$,  $h_{eq} = 1-0.3$, $l = 0.25$ and various values of $3\kappa^2$ ranging from $0.05$ to $1$.

 \begin{figure}[!ht]
\centering
\begin{tabular}{cc}
   \includegraphics[height=50mm]{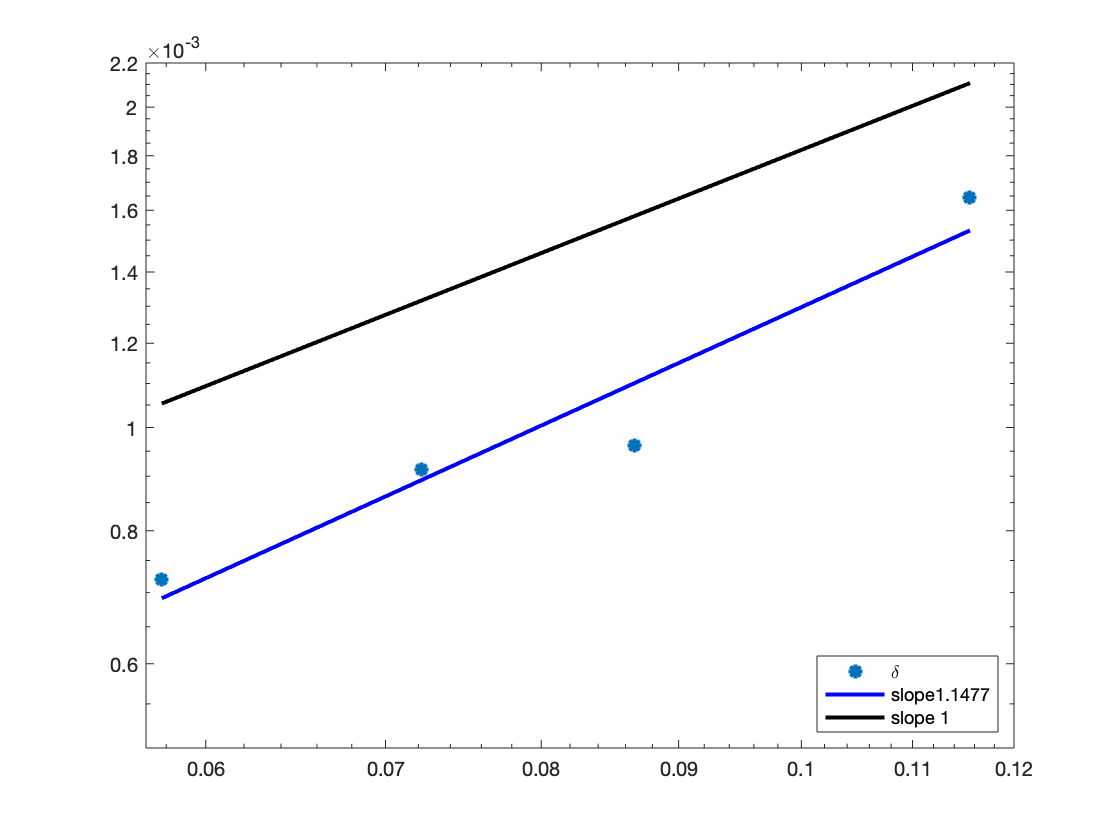}  & \includegraphics[height=50mm]{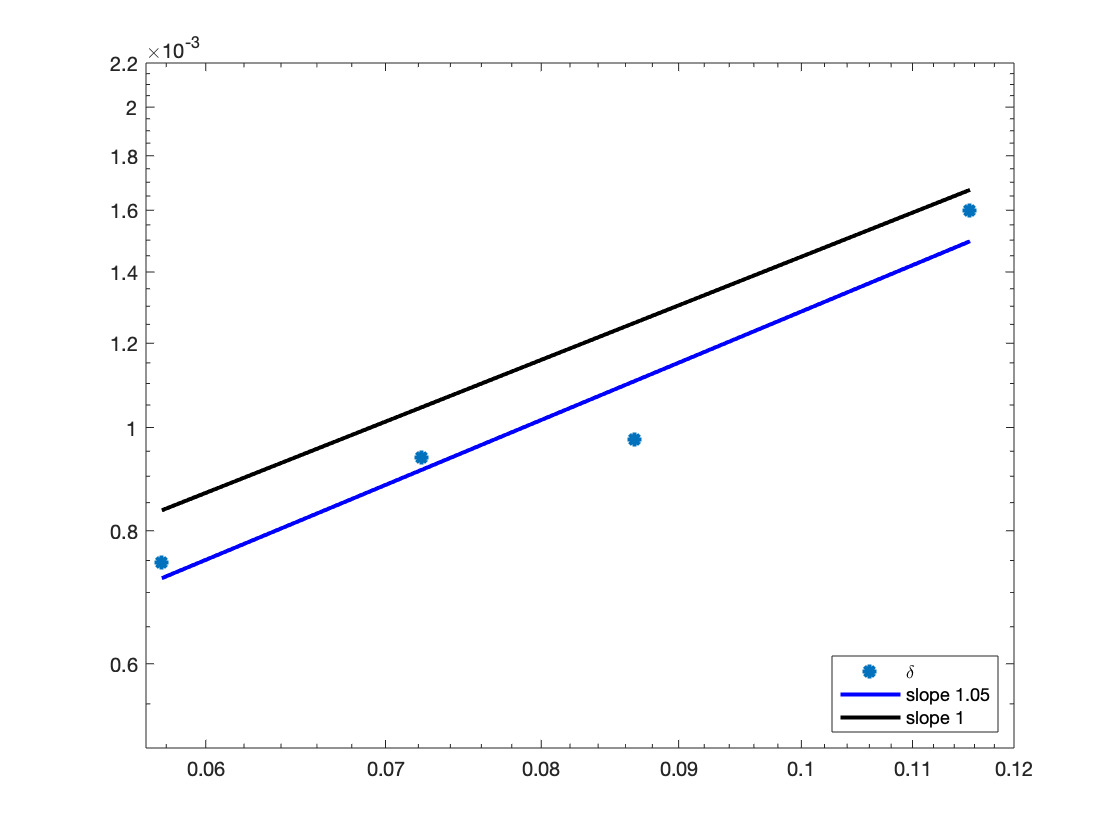}
\end{tabular}
  \caption{Return to equilibrium, non-linear case: mesh-convergence results for $\delta$ with the first-order scheme, $3\kappa^2 = 0.1$ (left) and $3\kappa^2 = 0.3$ (right).}
  \label{return_nonlinear_LF}
     \end{figure}

 \begin{figure}[!ht]
\centering
\begin{tabular}{cc}
   \includegraphics[height=50mm]{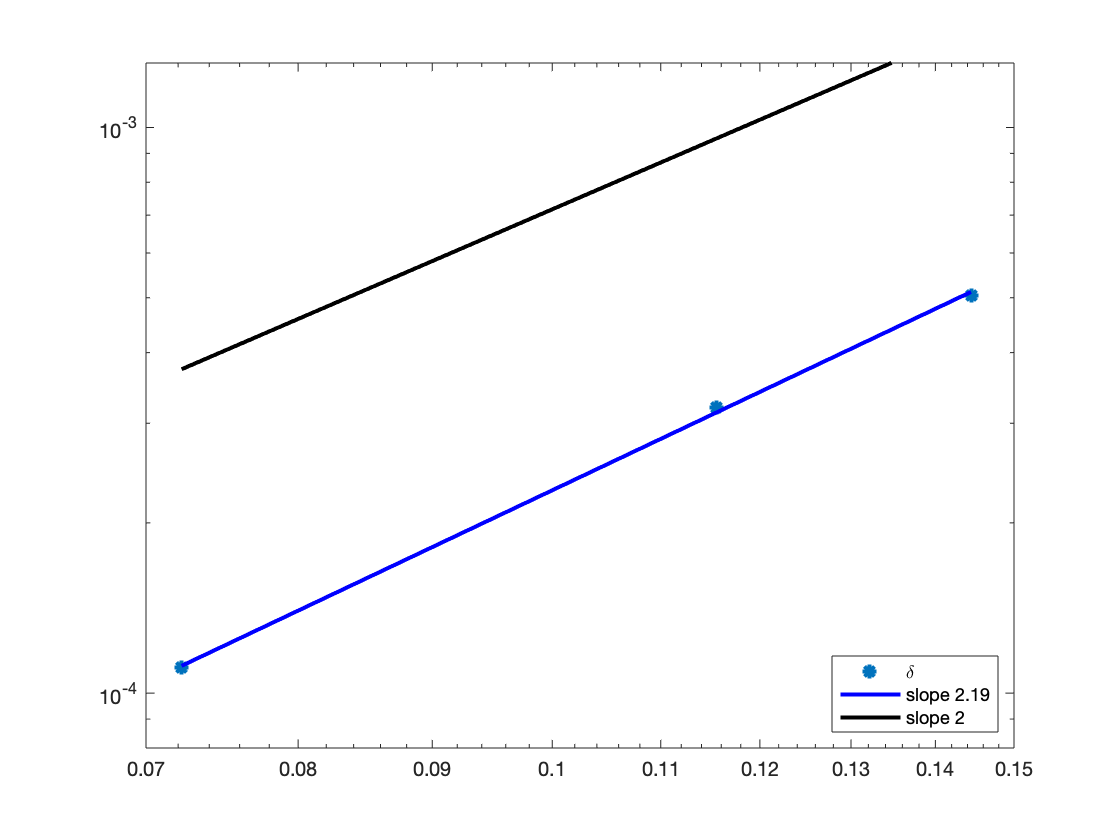}  & \includegraphics[height=50mm]{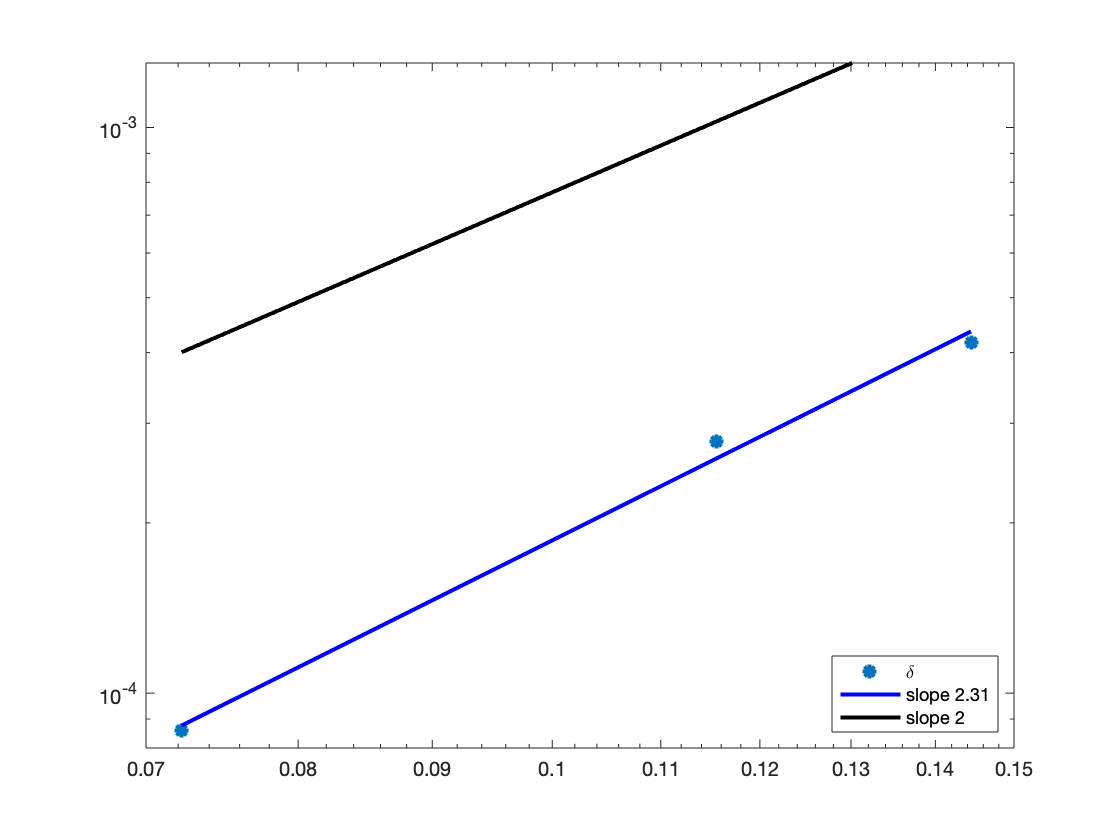}
\end{tabular}
  \caption{Return to equilibrium, non-linear case: mesh-convergence results for $\delta$ with the second-order-order scheme, $3\kappa^2 = 0.1$ (left) and $3\kappa^2 = 0.3$ (right).}
  \label{return_nonlinear_MC}
     \end{figure}

 \begin{figure}[!ht]
\centering
\begin{tabular}{c}
   \includegraphics[height=50mm]{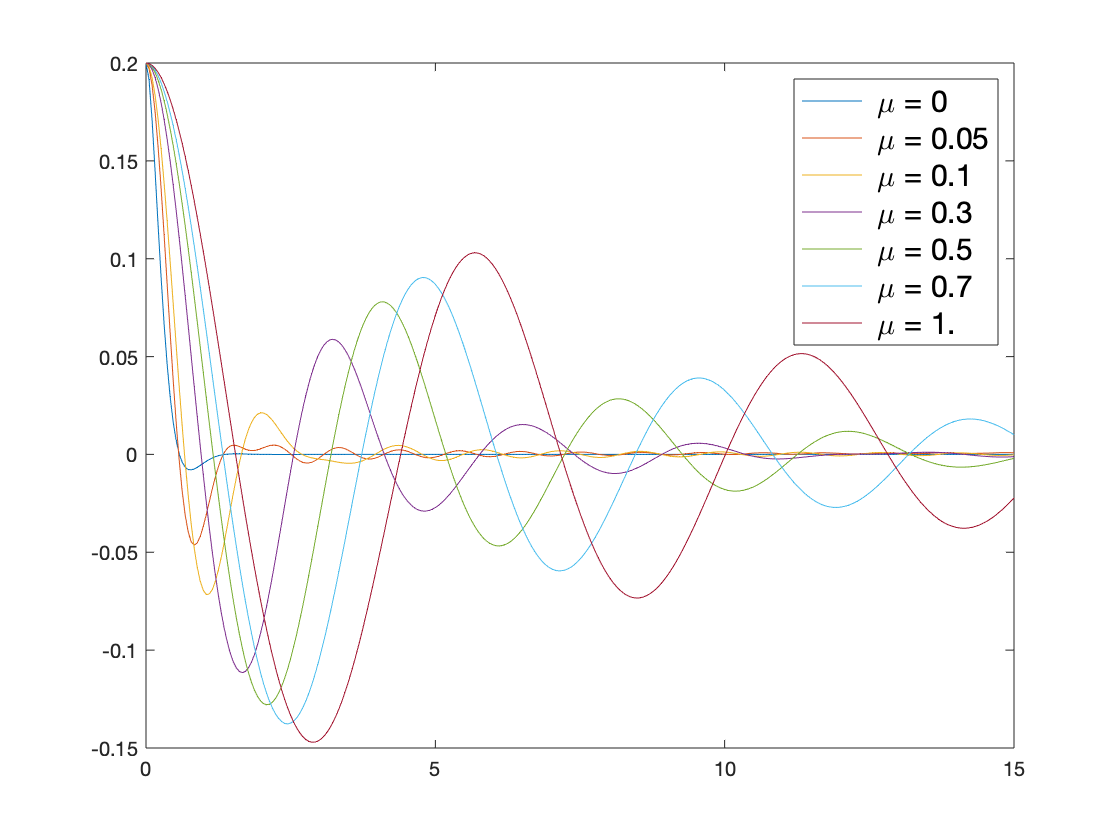}  
\end{tabular}
  \caption{Return to equilibrium, non-linear case: temporal evolution of $\delta$ for $\epsilon = 0.3$, comparison between results with the second-order scheme obtained with $N = 400$, different values of $\mu=3\kappa^2$, and the exact solution in the non dispersive case.}
  \label{return_nonlinear_compaLFMC}
     \end{figure}

\subsection{Waves interacting with a fixed object}\label{sectfixed}

We consider here waves that are interacting with a fixed partially immersed object. This is a particular example of prescribed motion ($\delta\equiv 0$), but contrary to the wave generation problem, the waves are not supposed to be symmetric with respect to the central axis $\{x=0\}$. It follows that the interior discharge $\av{q_{\rm i}}$ does not vanish identically and that it must be found by solving the ODE \eqref{transm3}. From a mathematical point of view, this physical configuration is somehow symmetric to the return to equilibrium problem in the sense that it allows one to focus on the coupling of the fluid equation with the dynamic of the interior discharge $\av{q_{\rm i}}$ (since $\delta\equiv 0$), while for the return to equilibrium problem the coupling was only with the vertical displacement $\delta$ (since in that case $\av{q_{\rm i}}\equiv 0$). In this case, the $7$-dimensional ODE on $\Theta$ of the augmented formulation given in Proposition \ref{propIVP} can be reduced to a $5$-dimensional ODE, as explained in \S \ref{appfixforced} of Appendix \ref{appcoeff}. We first study in \S \ref{sectcvlinfix} the linear case for which we exhibit a family of explicit solutions that we use to validate our code; the nonlinear case is then considered in \S \ref{sectNLfix}

\subsubsection{Convergence error in the linear case}\label{sectcvlinfix}
In order to investigate the ability of our scheme to correctly describe the coupling of the Boussinesq-Abbott equation with the average interior discharge $\av{q_{\rm i}}$ we exhibit an explicit solution of the equations in the linear case ($\eps=0$). In that case, the wave-structure equations \eqref{transm1}-\eqref{transm4} take the form
\begin{equation}\label{linfix1}
\begin{cases}
\dt \zeta+\dx q=0,\\
(1-\kappa^2\dx^2)\dt q+\dx \zeta=0
\end{cases}
\quad \mbox{ in }\quad \cE^\pm
\end{equation}
with transmission conditions
\begin{equation}\label{linfix2}
\jump{q}=0\quad\mbox{ and }\quad \av{q}=\av{q_{\rm i}}
\end{equation}
and where $\av{q_{\rm i}}$ solves the forced ODE
\begin{equation}\label{linfix3}
\alpha(0)\frac{d}{dt}\av{q_{\rm i}}=-\frac{1}{2\ell}\jump{\zeta+\kappa^2 \dt^2 \zeta}.
\end{equation}
A family of exact solutions which are periodic in time is given in the following proposition (which can be checked with basic computations omitted here).
\begin{proposition}
Let $k \neq 0$ and $\omega \neq 0$ satisfy the 
   the dispersion relation
$$
 \omega^2= \frac{k^2}{1+ \kappa^2 k^2}.
$$
For all $(\zeta_+^c, \zeta_-^c, \zeta^s, q_+^s, q_-^s, q^c)\in \RR^6$, the functions $(\zeta,q)$ defined in $\cE^\pm$ by
 $$
  \begin{cases}
 \zeta(x \mp \ell,t) =& \frac{k}{2} \big[ (\zeta_\pm^{c}+ q^{c}) \cos(kx-\omega t) +  (\zeta_\pm^{c}- q^{c}) \cos(kx+\omega t) 
  \\
  & + (\zeta^{s}+ q_\pm^{s}) \sin(kx-\omega t) + (\zeta^{s}- q_\pm^{s}) \sin(kx+\omega t) \big],  \\
q(x \mp \ell,t) =& \frac{\omega}{2} \big[ (\zeta_\pm^{c}+ q^{c}) \cos(kx-\omega t) -  (\zeta_\pm^{c}- q^{c}) \cos(kx+\omega t) 
  \\
  & + (\zeta^{s}+ q_\pm^{s}) \sin(kx-\omega t) - (\zeta^{s}- q_\pm^{s}) \sin(kx+\omega t) \big],
 \end{cases}
 $$
 solve \eqref{linfix1}-\eqref{linfix2} with  initial data
 $$
 \begin{cases}
 \zeta^{\text{\rm in}}(x \pm \ell) = k \zeta_\pm^{c} \cos(kx) + k \zeta^{s} \sin(kx), & x \in \mathcal{E}_\pm
 \\
 q^{\text{\rm in}}(x  \pm \ell) = \omega q^{c} \cos(kx) + \omega q_\pm^{s} \sin(kx), & x \in \mathcal{E}_\pm
 \end{cases}
 $$
 and with 
 $$
 \av{q_{\rm i}}(t)= 
 \omega \big[ q^c \cos(\omega t) - \zeta^s \sin(\omega t)  \big].
 $$
 If moreover
  $$
q^c= -\frac{1}{2 \ell \alpha(0) k} (q_+^s-q_-^s),
 \quad\mbox{ and }\quad 
 \zeta^s= \frac{1}{2 \ell \alpha(0) k} (\zeta_+^c-\zeta_-^c)
 $$
then \eqref{linfix3} is also satisfied, with initial data $\av{q_{\rm i}}^{\rm in}=\omega q^c$.
\end{proposition}

In our numerical tests we chose $3\kappa^2 = 0.3$ or $3\kappa^2 = 0.1$, $h_{eq} = 1-0.2$, $l = 1$, $k=2$ and the size of the computational domain $L = 10$.
 The space steps $\Delta_x$ were computed as $\Delta_x = (L-l)/(N+1)$, with $N=200,240,300,360,400$ for the second-order scheme.
 The time step was computed as $\Delta_t = 0.9 \Delta_x$.
 To impose the exact solution on both left and right outer boundaries we use the wave generation method described in \S \ref{sectWG}.
The numerical results at final time $T_f = 1$ show a second-order convergence, see Figure \ref{fixed_linear_MC}.
On Figure \ref{fixed_linear_MC-example} one can observe the shape of this exact solution, computed with $N = 400$ points.

 \begin{figure}[!ht]
\centering
\begin{tabular}{cc}
   \includegraphics[height=50mm]{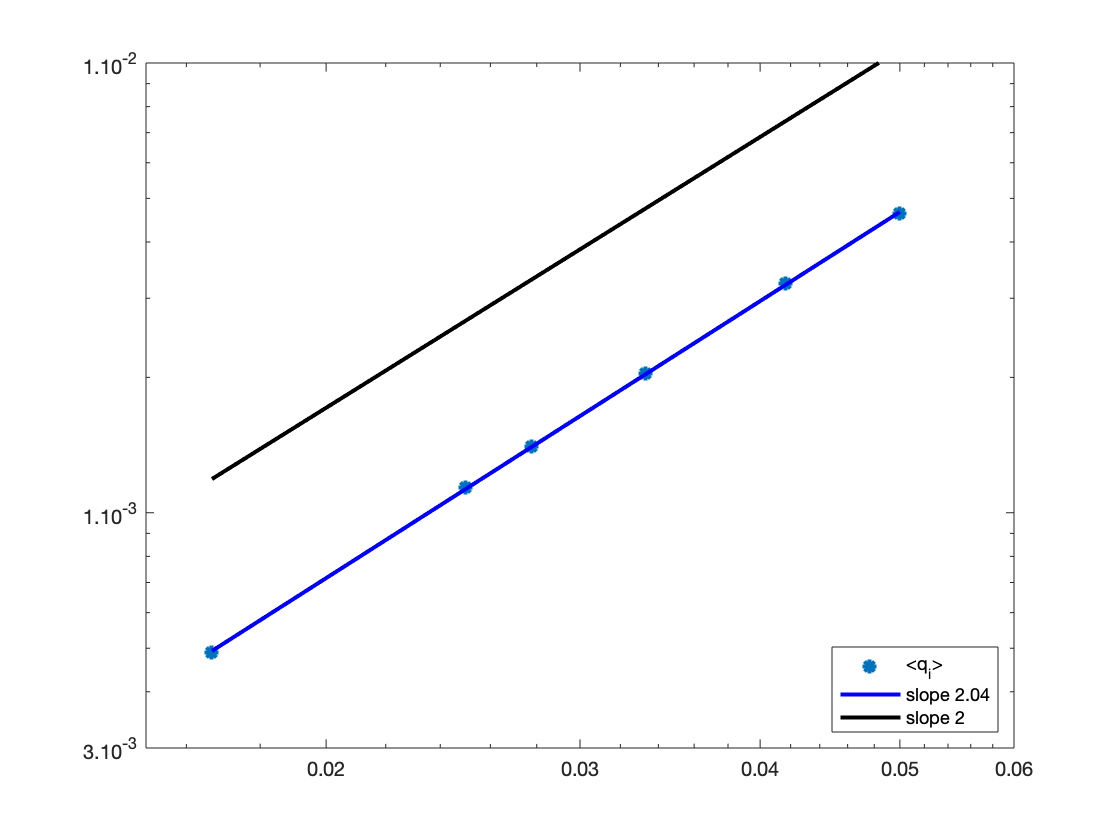}  & \includegraphics[height=50mm]{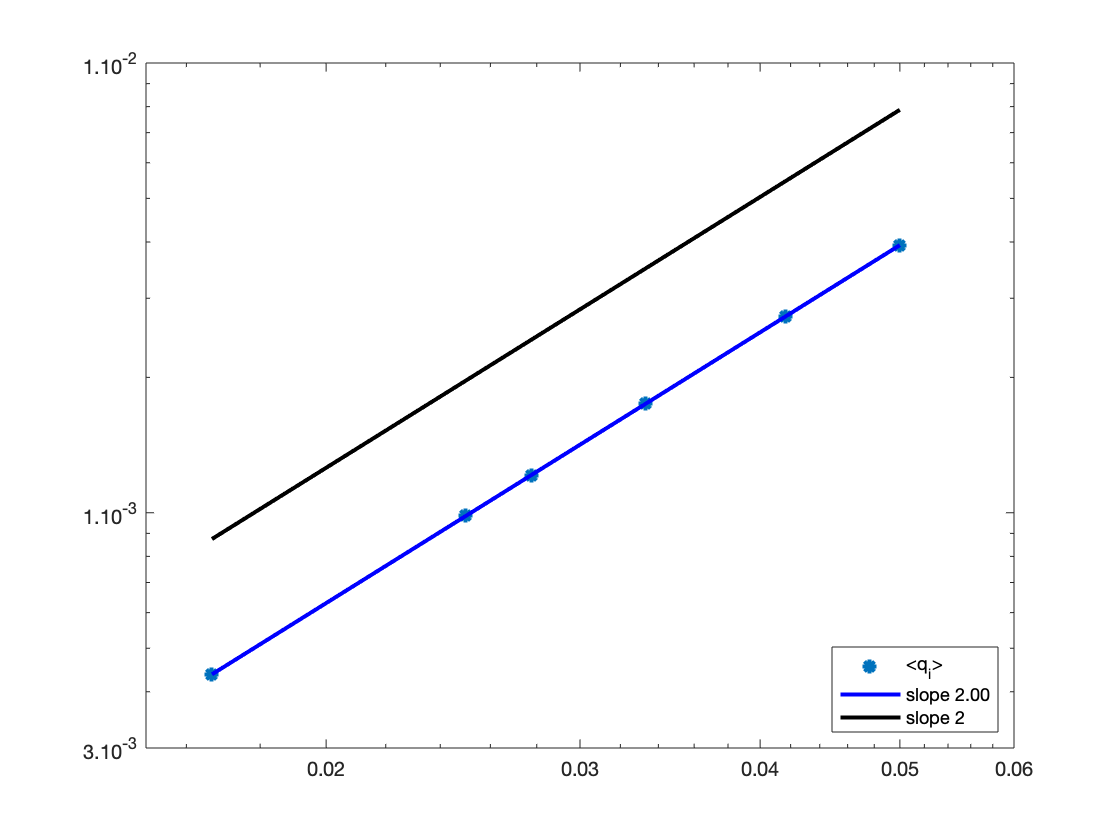}
\end{tabular}
  \caption{Interaction with a fixed object, linear case: convergence results for $\av{q_{\rm i}}$ with the second-order scheme, $3\kappa^2 = 0.1$ (left) and $3\kappa^2 = 0.3$ (right).}
  \label{fixed_linear_MC}
     \end{figure}

 \begin{figure}[!ht]
\centering
\begin{tabular}{cc}
   \includegraphics[height=50mm]{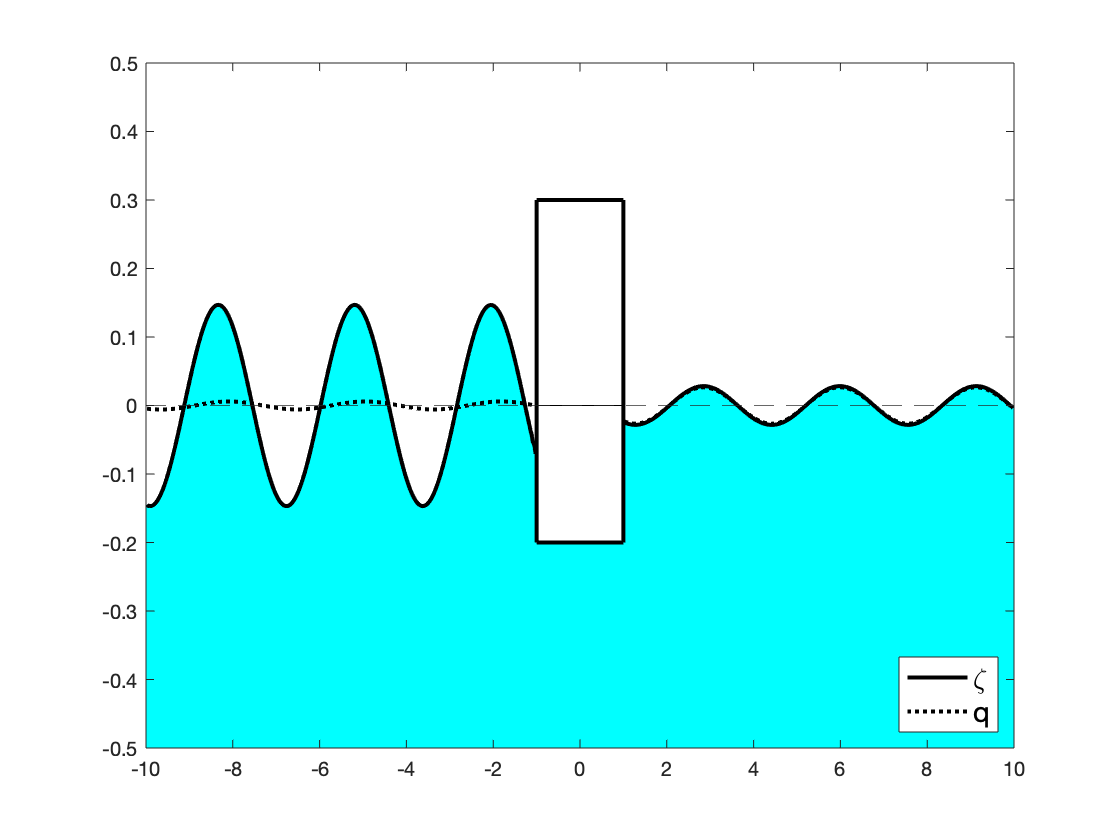}  & \includegraphics[height=50mm]{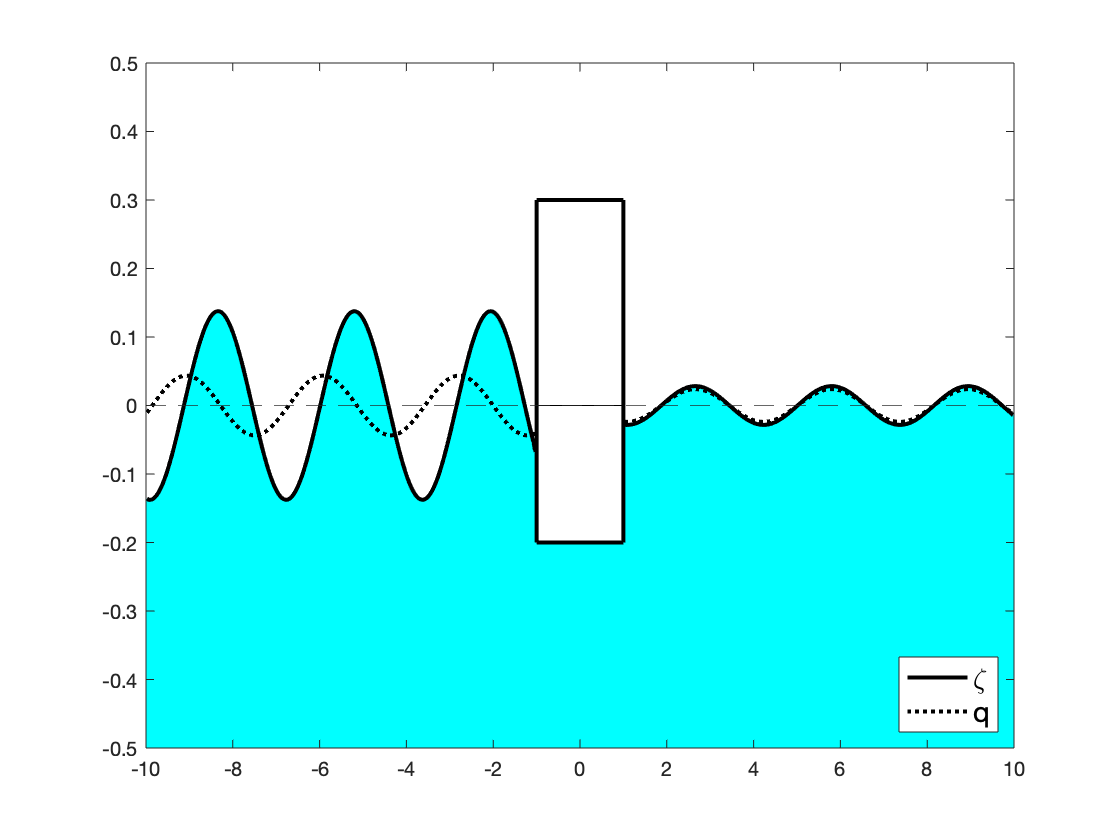}
\end{tabular}
  \caption{Interaction with a fixed object, linear case: profile of solutions, $3\kappa^2 = 0.1$ (left) and $3\kappa^2 = 0.3$ (right). 
  }
  \label{fixed_linear_MC-example}
     \end{figure}

\subsubsection{The nonlinear case}\label{sectNLfix}

In the absence of explicit solution in the nonlinear case, we use mesh-convergence to study the precision of our schemes. In this test the initial condition is the solitary wave described in \S \ref{sectWG}, with $\zeta_{\rm max} = 0.2$, centered at $x = -15$, at the left side of the fixed object. The size of the computational domain is $L = 30$. 
For this test the reference solution is computed with a very refined mesh: $N = 2400$.
 We chose $\epsilon = 0.3$, $3\kappa^2 = 0.3$ or $3\kappa^2 = 0.1$, $h_{eq} = 1-0.3$, $l = 4$.
 The space steps $\Delta_x$ were computed as $\Delta_x = (L-l)/(N+1)$, with $N=160,200,240,300,400$ for the first-order scheme and $N = 100,120,160$ for the second-order scheme.
  The meshes are defined so that the points of the coarse meshes always coincide with the points of the very refined mesh of the reference solution.
 The time step was computed as $\Delta_t = 0.7 \Delta_x$.
The numerical results at final time $T_f = 20$ computed with the first-order and the second-order schemes show respectively a first-order convergence for $\av{q_{\rm i}}$, see Figure \ref{fixed_nonlinear_LF} and a second-order convergence, see Figure \ref{fixed_nonlinear_MC}.
On Figure \ref{fixed_nonlinear_MC_profile} one can observe the shape of the numerical solution, computed with $N = 400$ points.

 \begin{figure}[!ht]
\centering
\begin{tabular}{cc}
   \includegraphics[height=50mm]{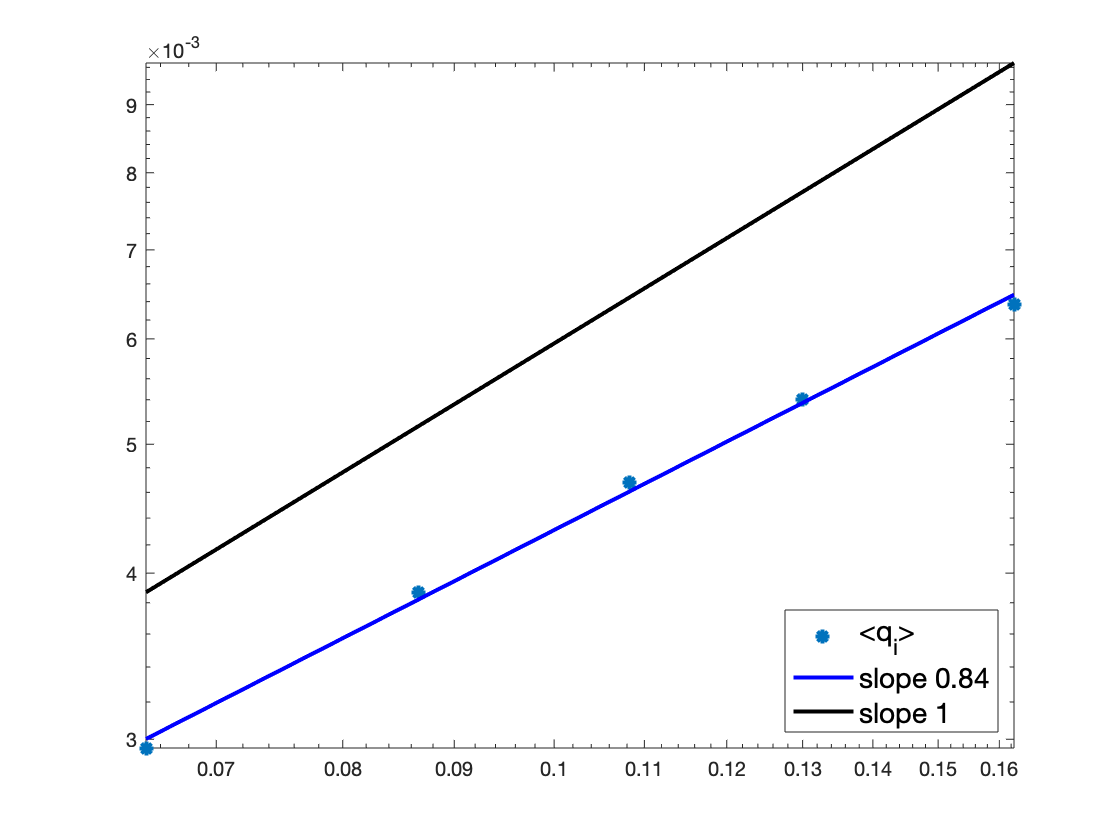}  & \includegraphics[height=50mm]{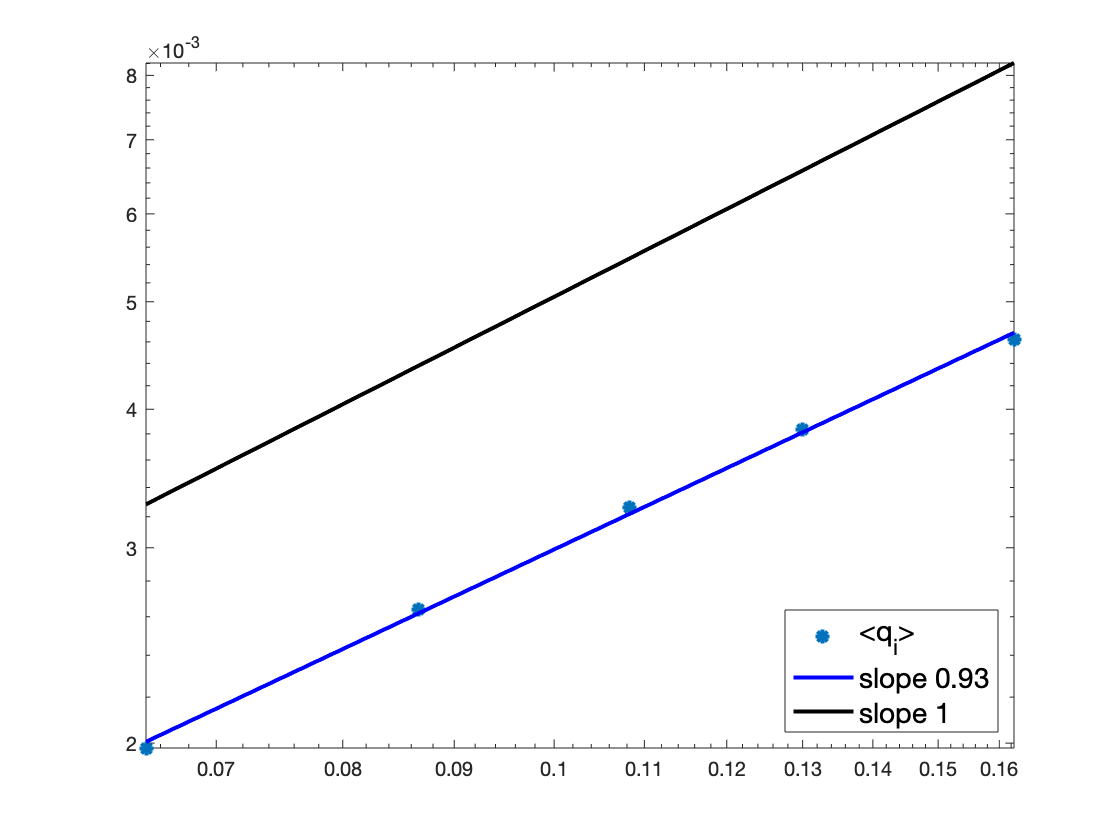}
\end{tabular}
  \caption{Interaction with a fixed object, non-linear case: mesh-convergence results for $\av{q_{\rm i}}$ with the first-order scheme, $3\kappa^2 = 0.1$ (left) and $3\kappa^2 = 0.3$ (right).}
  \label{fixed_nonlinear_LF}
     \end{figure}

 \begin{figure}[!ht]
\centering
\begin{tabular}{cc}
   \includegraphics[height=50mm]{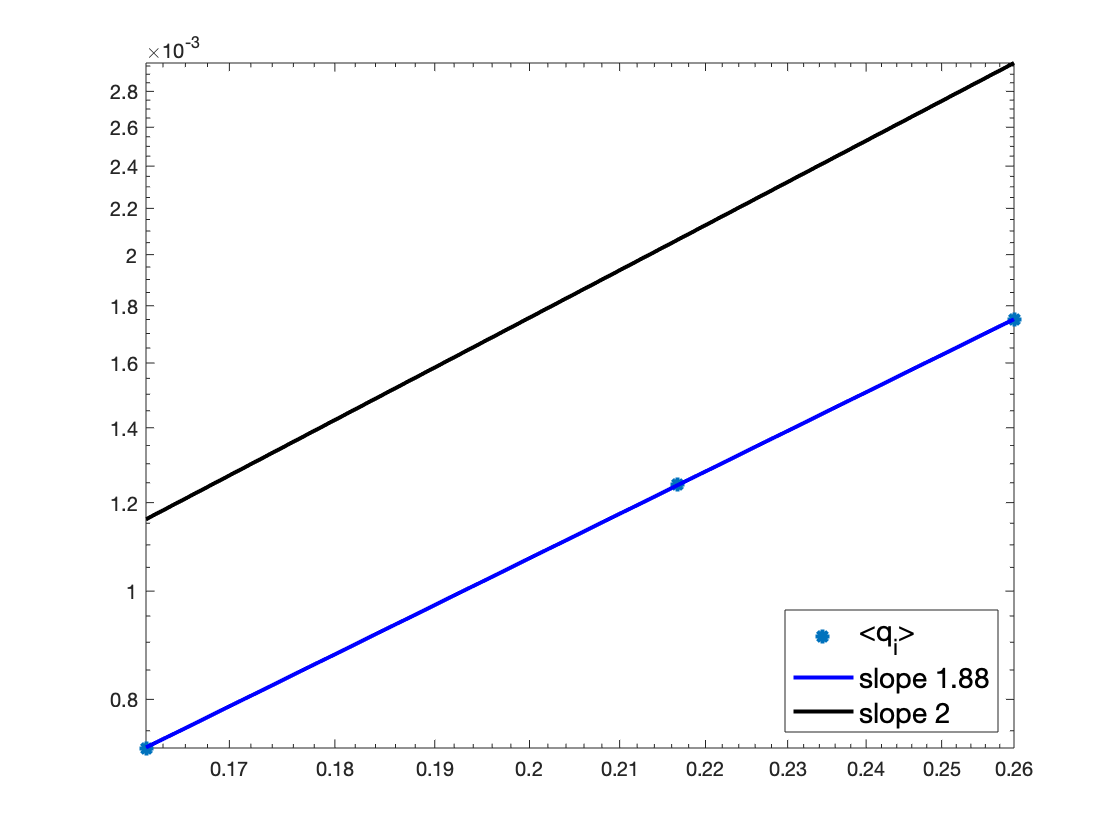}  & \includegraphics[height=50mm]{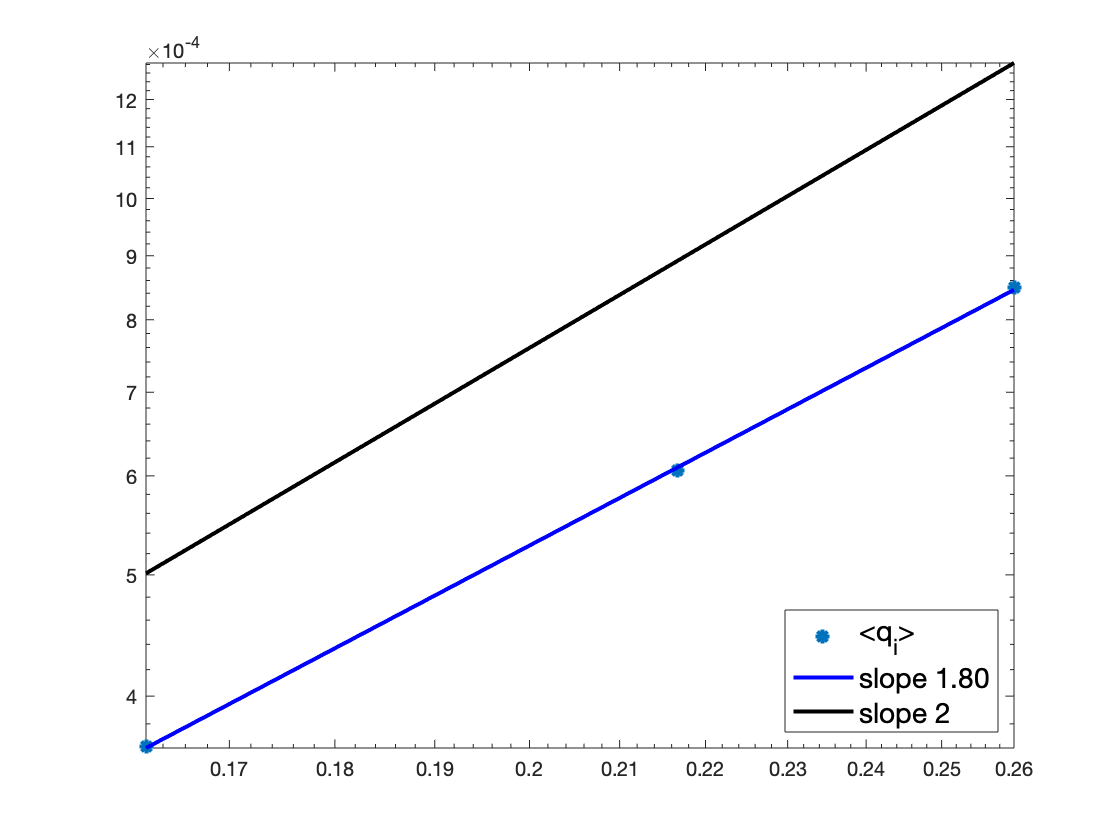}
\end{tabular}
  \caption{Interaction with a fixed object, non-linear case: mesh-convergence results for $\av{q_{\rm i}}$ with the second-order scheme, $3\kappa^2 = 0.1$ (left) and $3\kappa^2 = 0.3$ (right).}
  \label{fixed_nonlinear_MC}
     \end{figure}

 \begin{figure}[!ht]
\centering
\begin{tabular}{cc}
   \includegraphics[height=50mm]{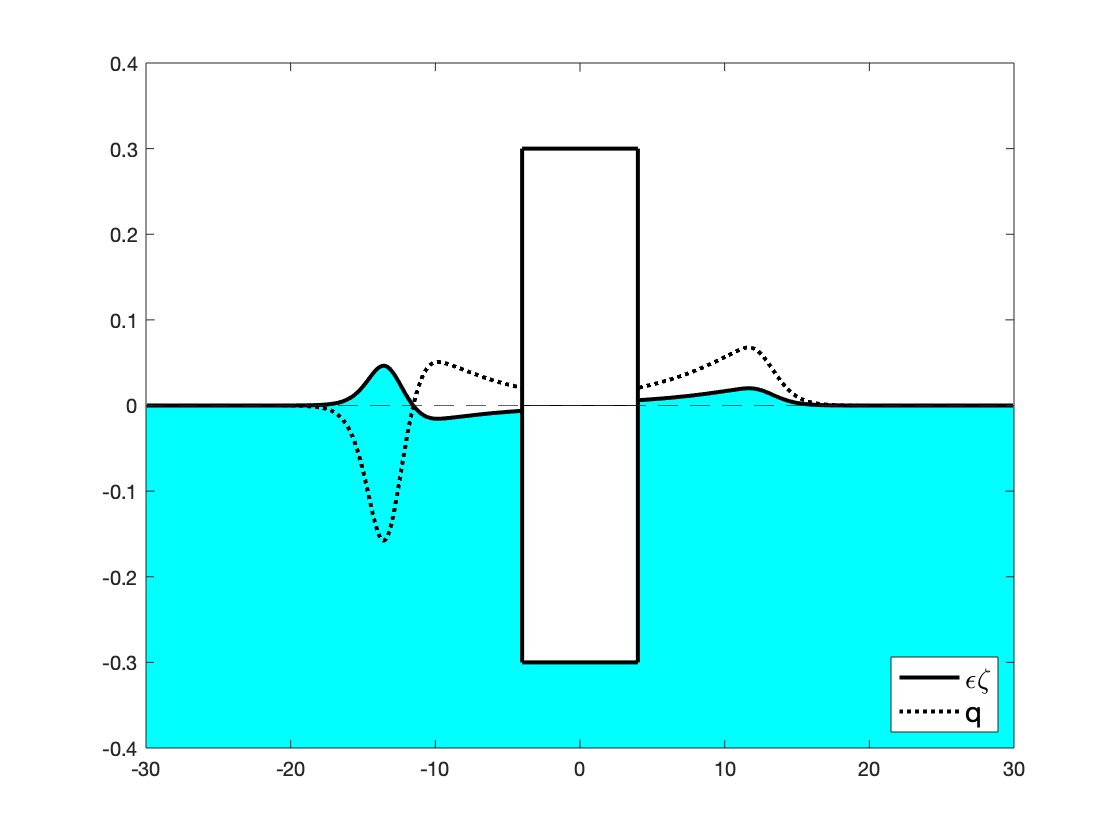}  & \includegraphics[height=50mm]{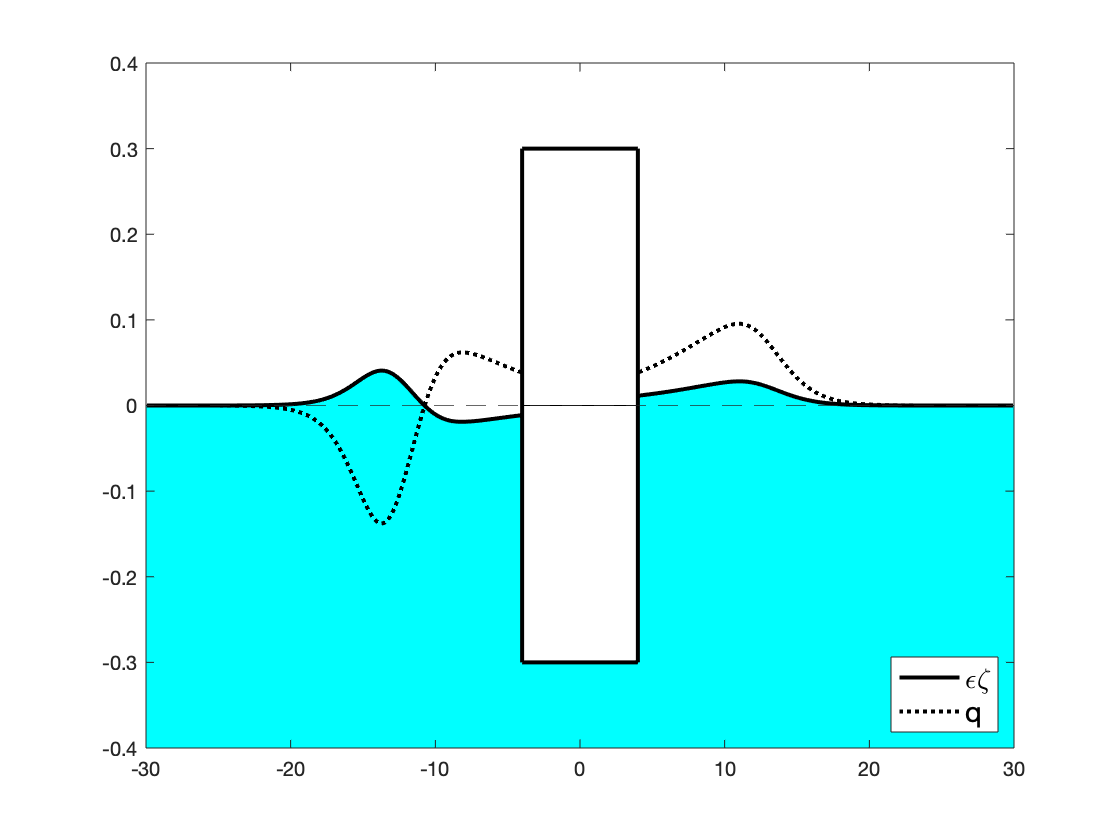}
\end{tabular}
  \caption{Interaction with a fixed object, non-linear case: profile of solutions, $3\kappa^2 = 0.1$ (left) and $3\kappa^2 = 0.3$ (right), computed with the second-order scheme and $N = 400$.}
  \label{fixed_nonlinear_MC_profile}
     \end{figure}

\subsection{Waves interacting with a freely floating object}\label{sectfreely}

In the general case, when the object moves under the influence of the waves and the waves are in return modified by the presence of the object, one has to consider the full augmented system presented in Proposition \ref{propIVP}.

We use exactly the same test as in \S \ref{sectNLfix}, the only difference being that the object is now allowed to move vertically under the action of the waves, so that we also have to study the convergence of $\delta$.
 The space steps $\Delta_x$ were computed as $\Delta_x = (L-l)/(N+1)$, with $N=240,300,400$ for the first-order scheme, $N = 120,160,200,240$ for the second-order scheme and  $3\kappa^2 = 0.1$, and $N = 80,100,120,160$ for the second-order scheme and $3\kappa^2 = 0.3$.
The numerical results at final time $T_f = 20$ computed with the first-order and the second-order schemes show respectively a first-order convergence, see Figures \ref{general_LF_delta}, \ref{general_LF_qi} and  \ref{general_LF_zeta}, and a second-order convergence, see Figures \ref{general_MC_delta}, \ref{general_MC_qi} and \ref{general_MC_zeta}.
On Figure \ref{general_MC_profile} one can observe the shape of the numerical solution, computed with $N = 400$ points.
A comparison between Figure \ref{general_MC_profile} and Figure \ref{fixed_nonlinear_MC_profile} shows that the profiles of the reflected and transmitted waves differ. In particular, when the object is allowed to move, the reflected and transmitted wave are preceded by a depression wave that is not present when the object if fixed.

\begin{figure}[!ht]
\centering
\begin{tabular}{cc}
   \includegraphics[height=50mm]{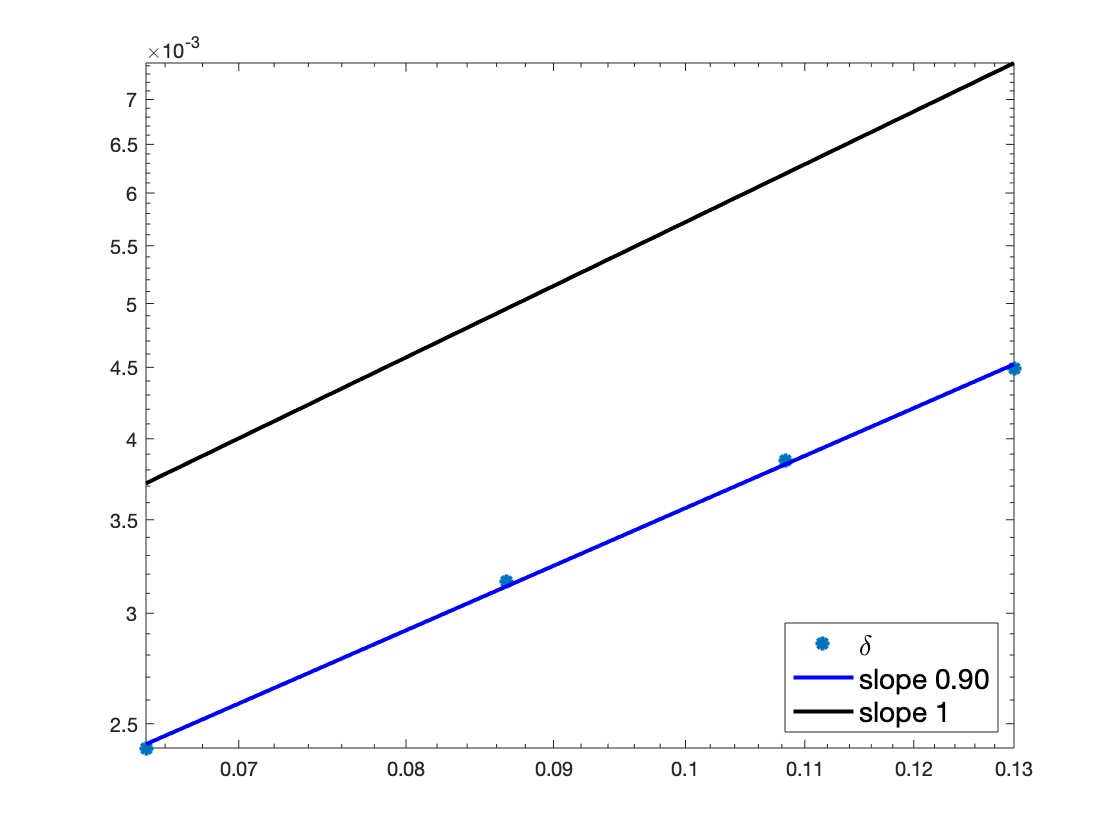}  & \includegraphics[height=50mm]{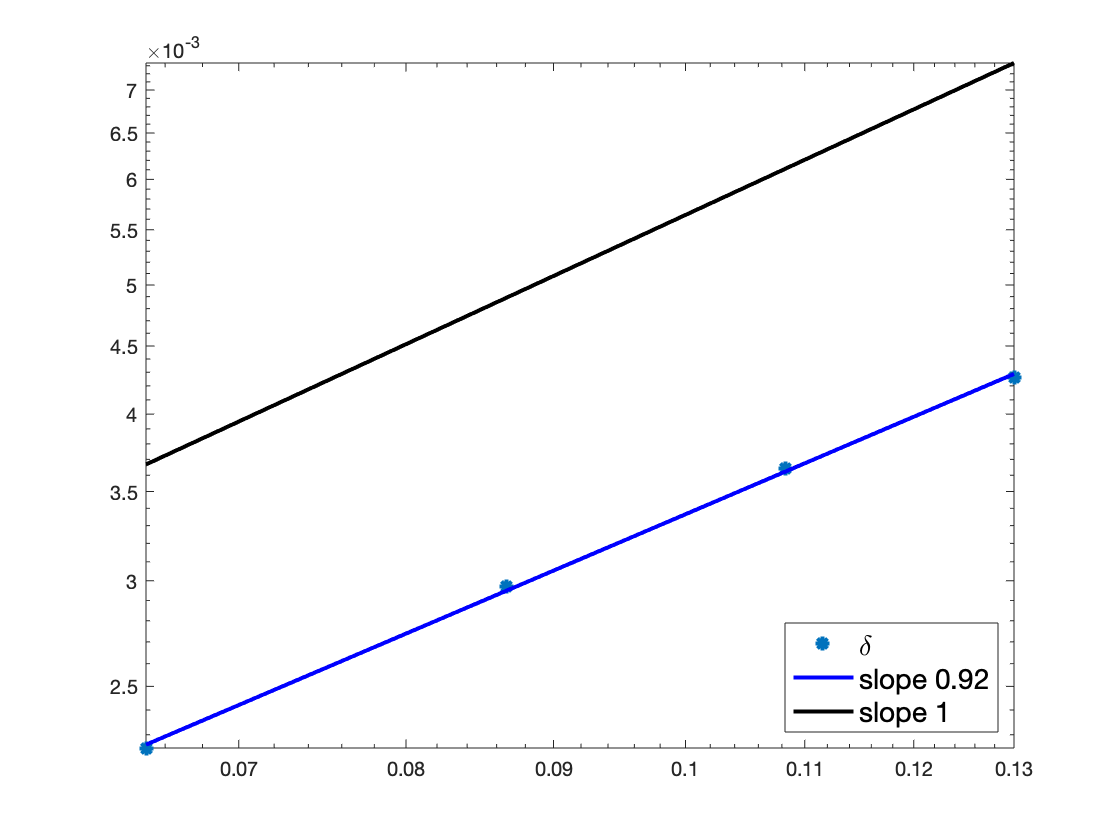}
\end{tabular}
  \caption{Interaction with a freely floating object, non-linear case: mesh-convergence results for $\delta$ with the first-order scheme, $3\kappa^2 = 0.1$ (left) and $3\kappa^2 = 0.3$ (right).}
  \label{general_LF_delta}
     \end{figure}

\begin{figure}[!ht]
\centering
\begin{tabular}{cc}
   \includegraphics[height=50mm]{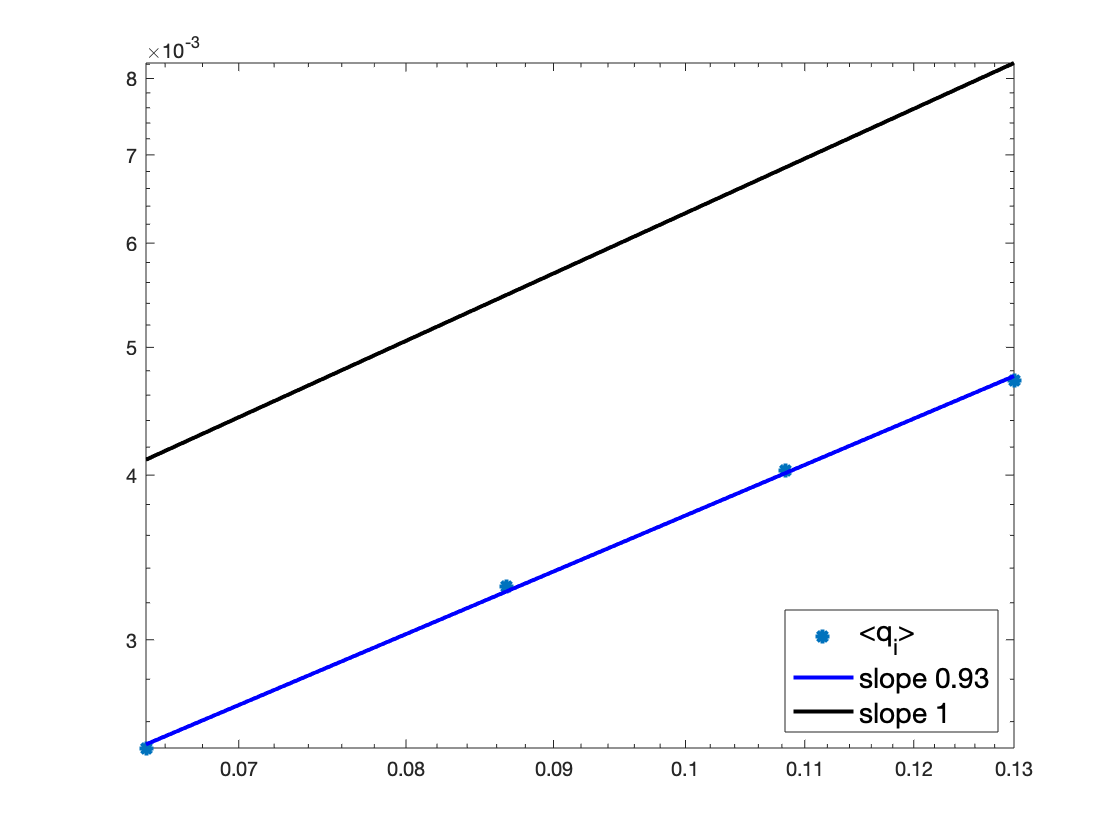}  & \includegraphics[height=50mm]{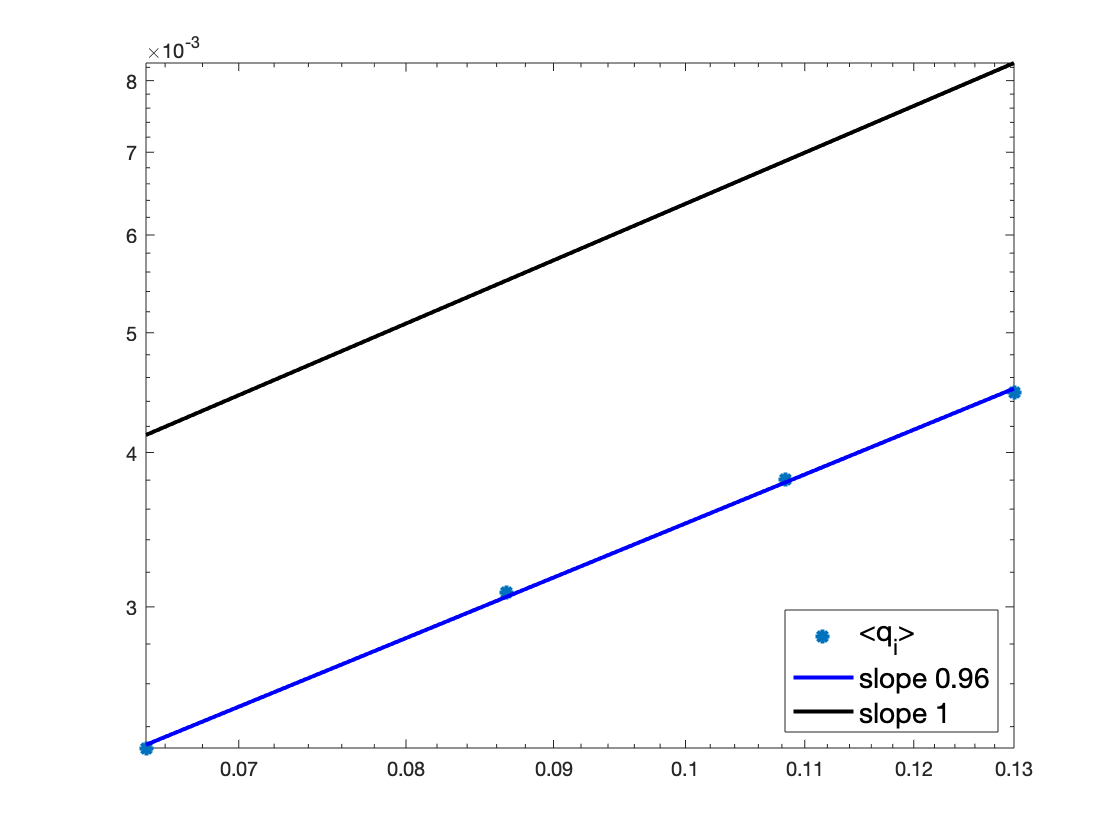}
\end{tabular}
  \caption{Interaction with a freely floating object, non-linear case: mesh-convergence results for $\av{q_{\rm i}}$ with the first-order scheme, $3\kappa^2 = 0.1$ (left) and $3\kappa^2 = 0.3$ (right).}
  \label{general_LF_qi}
     \end{figure}
     
     \begin{figure}[!ht]
\centering
\begin{tabular}{cc}
   \includegraphics[height=50mm]{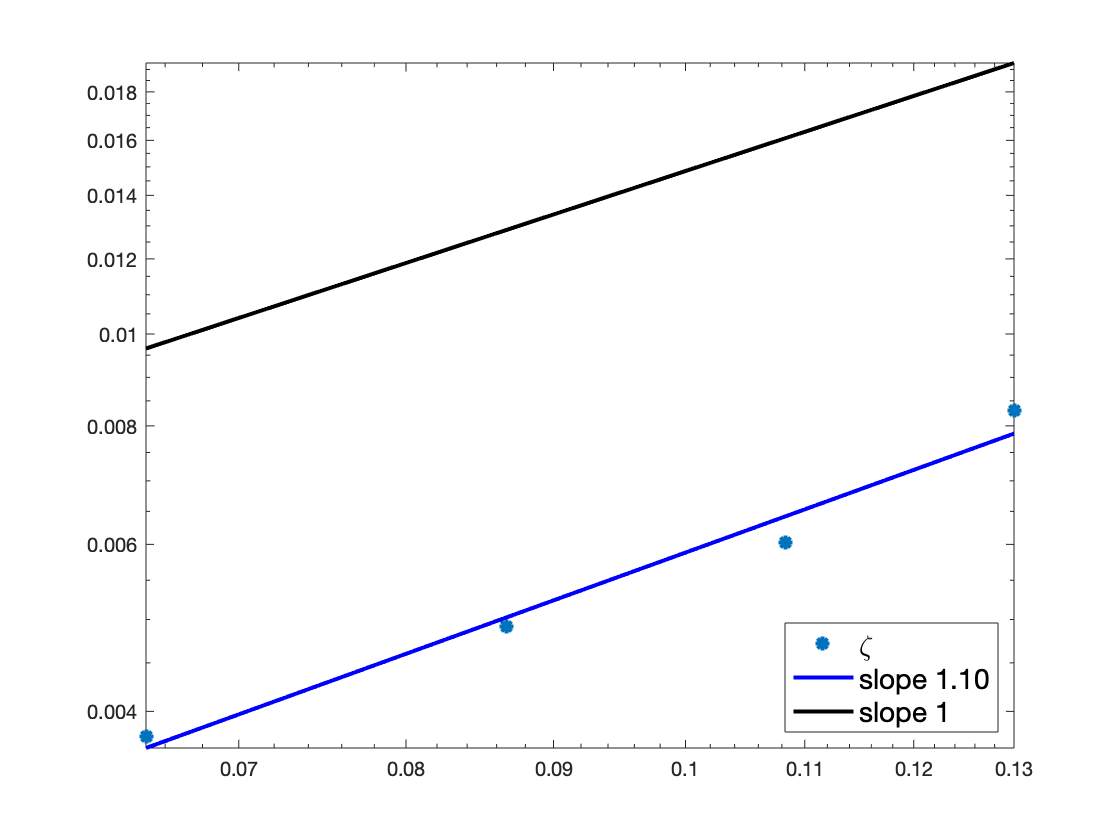}  & \includegraphics[height=50mm]{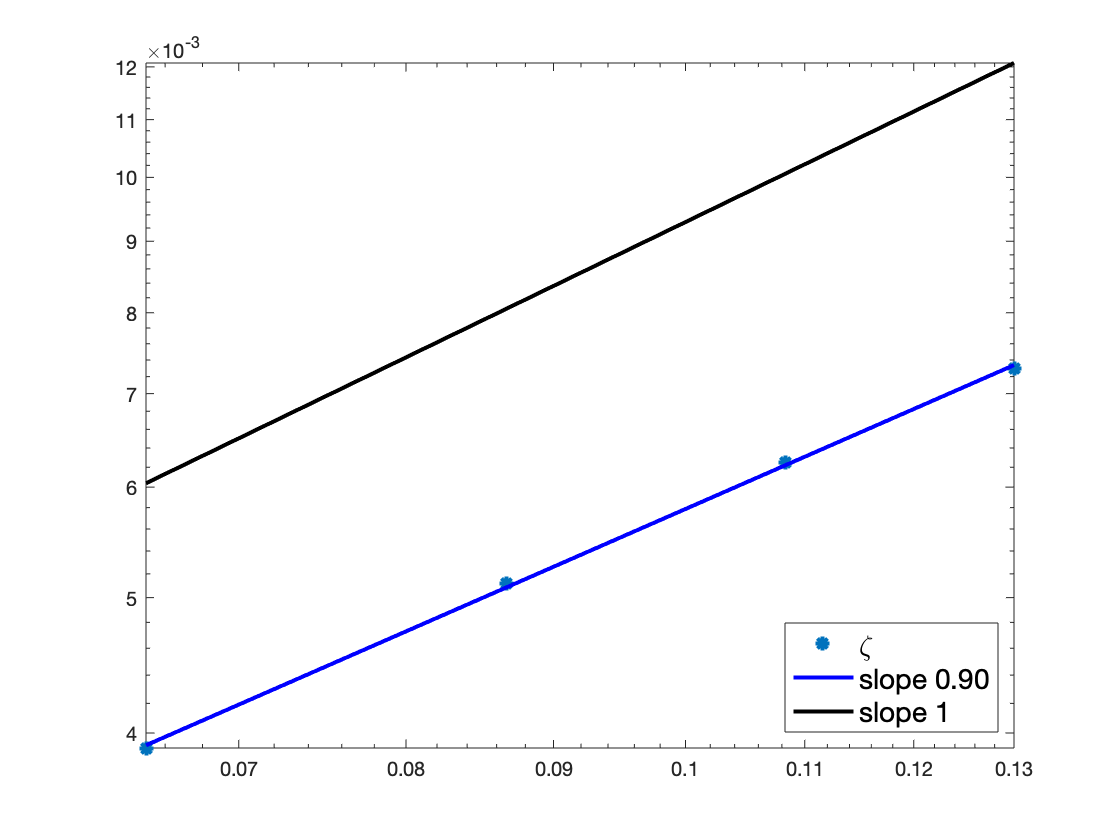}
\end{tabular}
  \caption{Interaction with a freely floating object, non-linear case: mesh-convergence results for $\zeta_{+}$ with the first-order scheme, $3\kappa^2 = 0.1$ (left) and $3\kappa^2 = 0.3$ (right).}
  \label{general_LF_zeta}
     \end{figure}

 \begin{figure}[!ht]
\centering
\begin{tabular}{cc}
   \includegraphics[height=50mm]{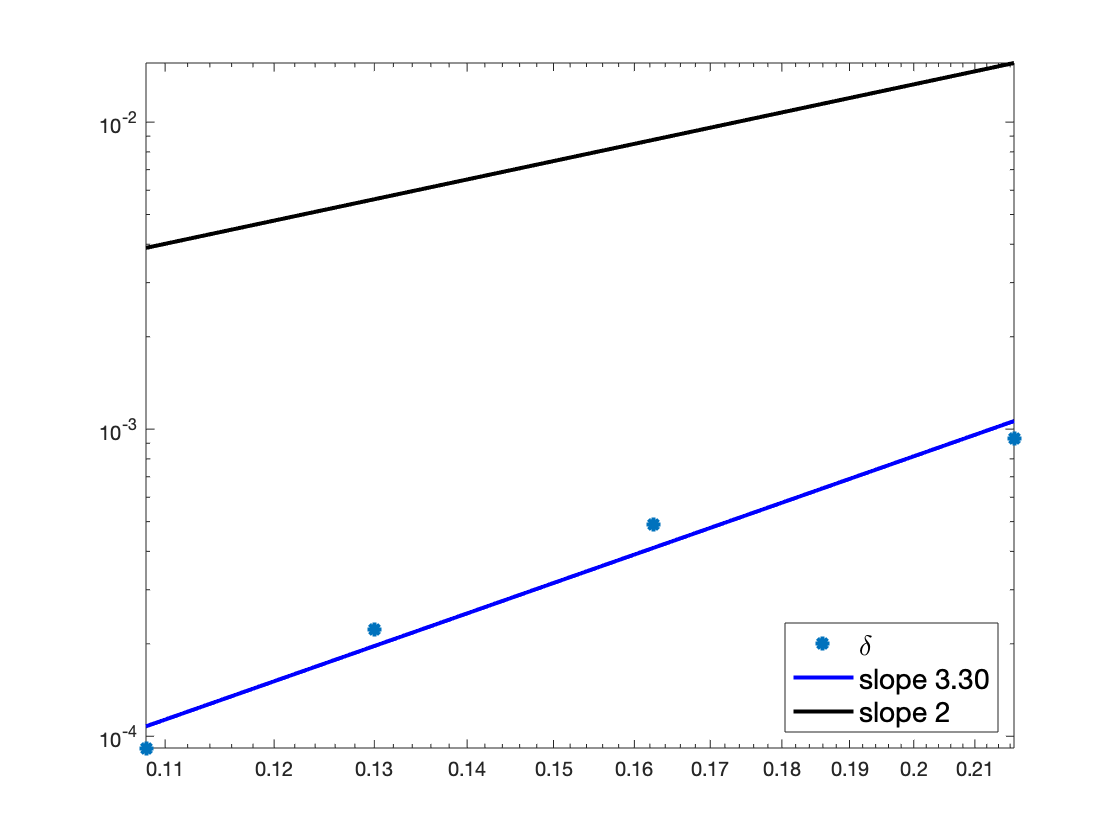}  & \includegraphics[height=50mm]{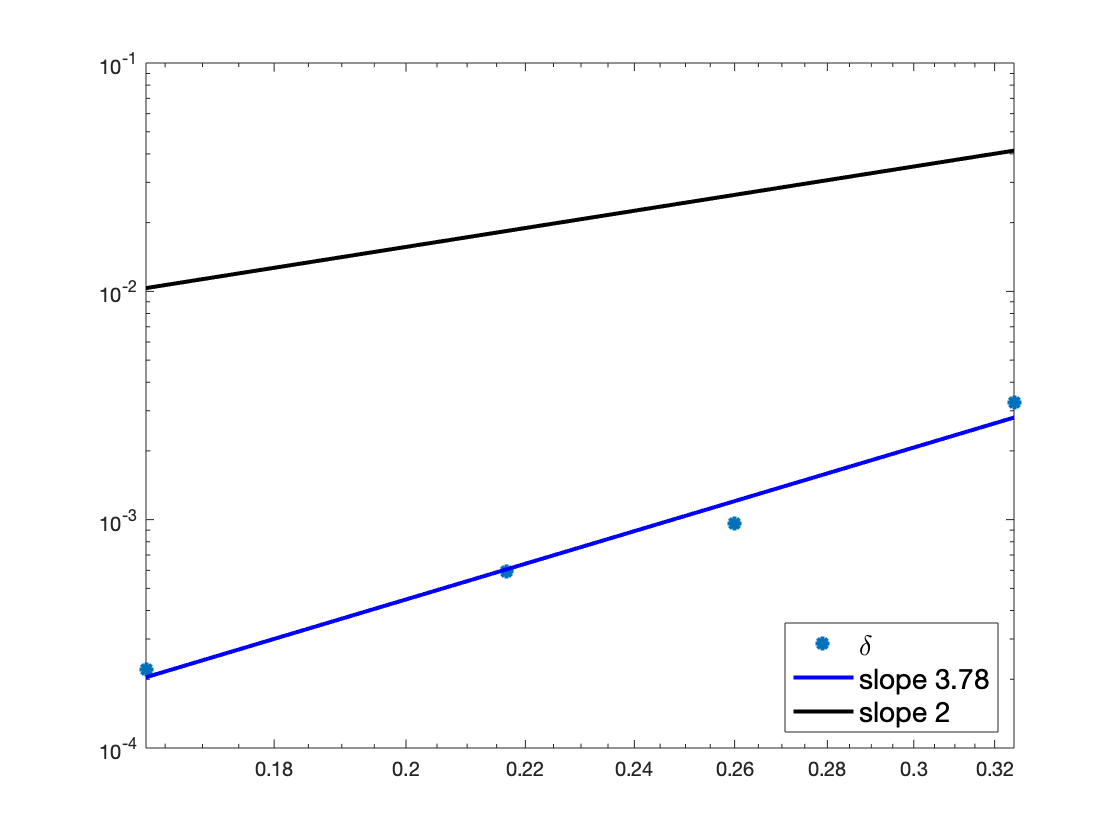}
\end{tabular}
  \caption{Interaction with a freely floating object, non-linear case: mesh-convergence results for $\delta$ with the second-order scheme, $3\kappa^2 = 0.1$ (left) and $3\kappa^2 = 0.3$ (right).}
  \label{general_MC_delta}
     \end{figure}

 \begin{figure}[!ht]
\centering
\begin{tabular}{cc}
   \includegraphics[height=50mm]{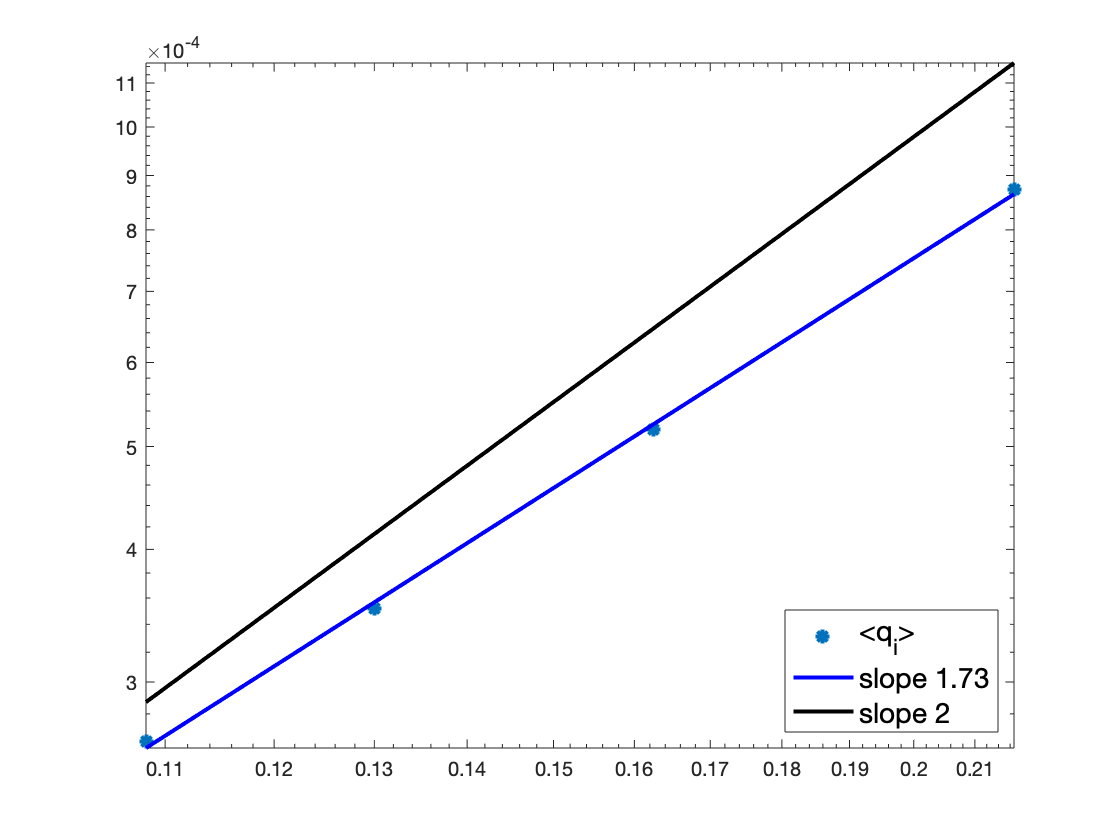}  & \includegraphics[height=50mm]{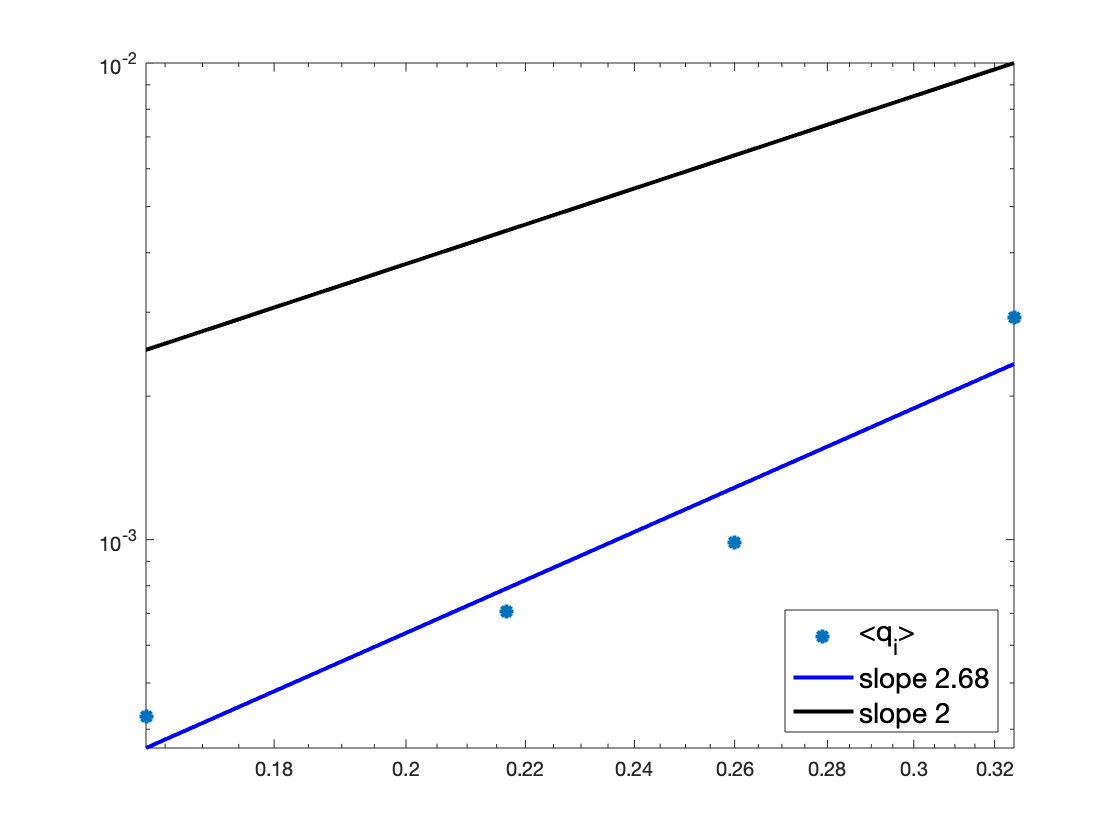}
\end{tabular}
  \caption{Interaction with a freely floating object, non-linear case: mesh-convergence results for $\av{q_{\rm i}}$ with the second-order scheme, $3\kappa^2 = 0.1$ (left) and $3\kappa^2 = 0.3$ (right).}
  \label{general_MC_qi}
     \end{figure}

 \begin{figure}[!ht]
\centering
\begin{tabular}{cc}
   \includegraphics[height=50mm]{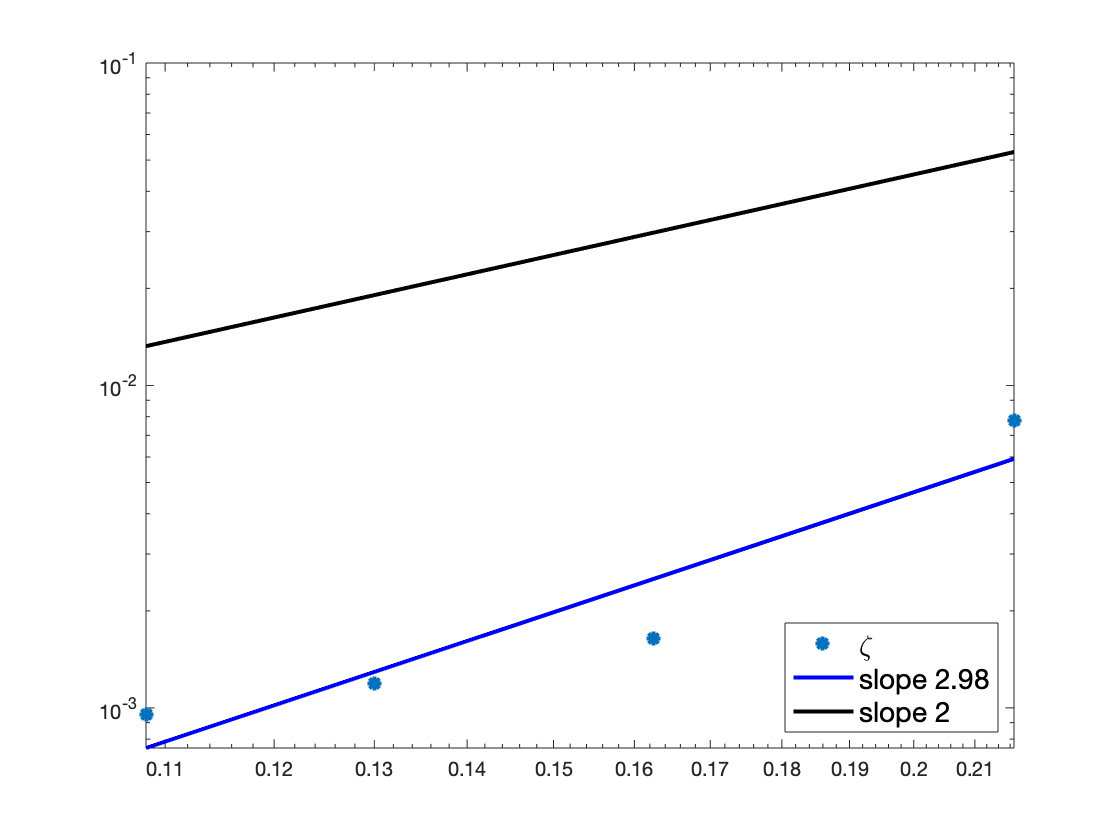}  & \includegraphics[height=50mm]{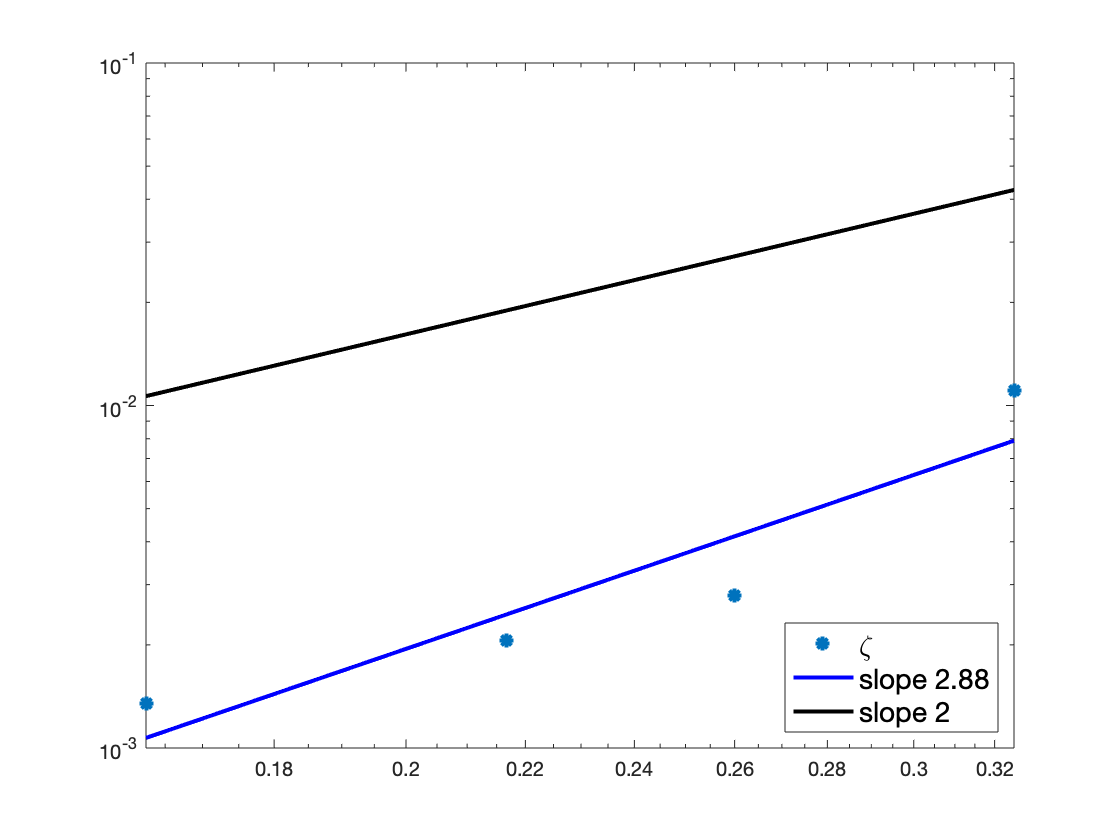}
\end{tabular}
  \caption{Interaction with a freely floating object, non-linear case: mesh-convergence results for $\zeta_{+}$ with the second-order scheme, $3\kappa^2 = 0.1$ (left) and $3\kappa^2 = 0.3$ (right).}
  \label{general_MC_zeta}
     \end{figure}

 \begin{figure}[!ht]
\centering
\begin{tabular}{cc}
   \includegraphics[height=50mm]{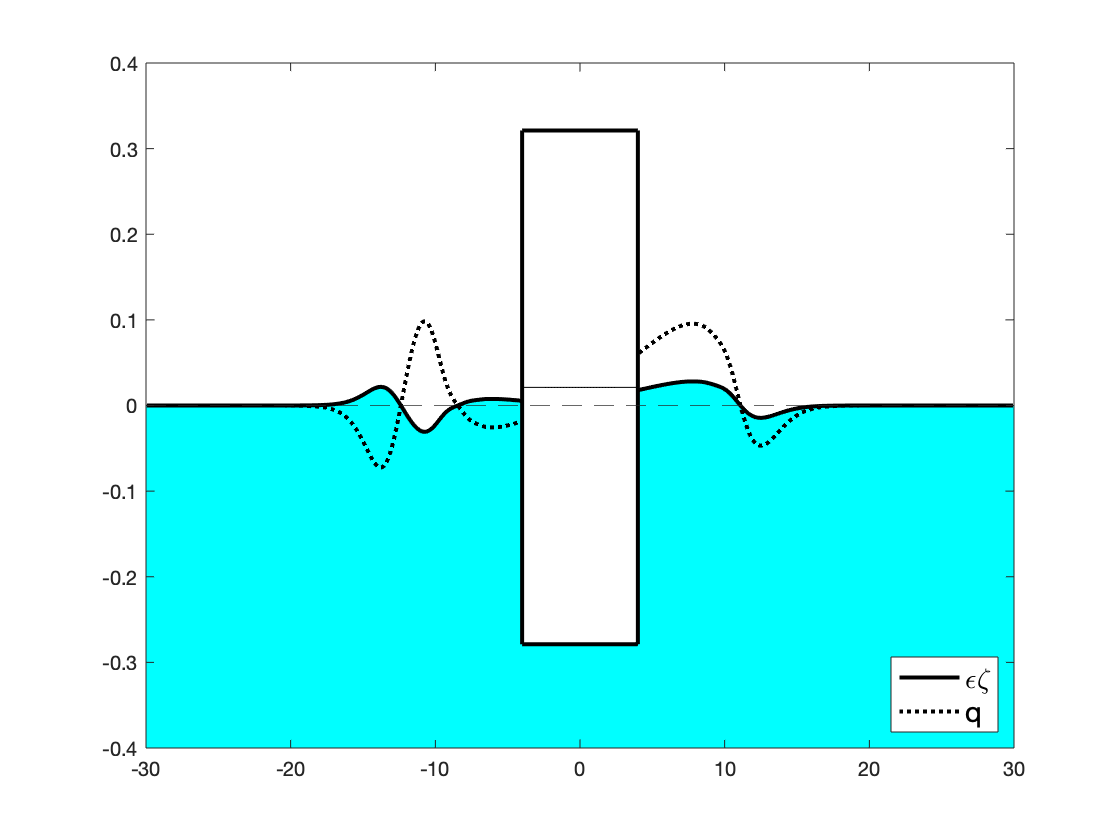}  & \includegraphics[height=50mm]{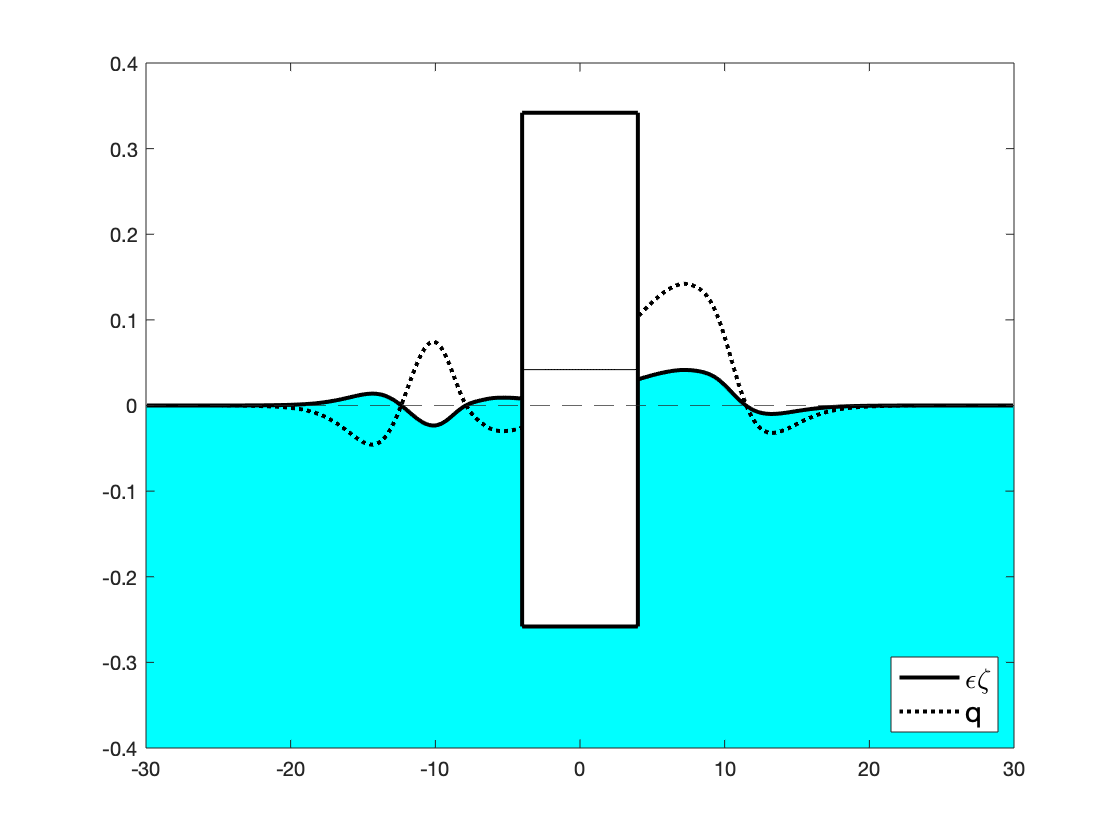}
\end{tabular}
  \caption{Interaction with a freely floating object, non-linear case: profile of solutions, $3\kappa^2 = 0.1$ (left) and $3\kappa^2 = 0.3$ (right), computed with the second-order scheme and $N = 400$.}
  \label{general_MC_profile}
     \end{figure}

\appendix

\section{Exact expressions of the quantities involved in the wave-structure models}\label{appcoeff}

We have  shown in \S \ref{sectaugmform} that the wave-structure interaction equations can be formulated as a transmission problem between the two components $\cE^-=(-\infty,-\ell)$ and $\cE^+=(\ell,\infty)$ of the fluid domain, with transmission conditions involving the vertical displacement $\delta$ of the object and the average horizontal discharge $\av{q_{\rm i}}$ under the object; we have also shown  that, in the augmented formulation, $\delta$ and $\av{q_{\rm i}}$ are found through the resolution of the  $7$-dimensional ODE  \eqref{ODEaugm} on $\Theta:=\big(\av{q_{\rm i}},\dot\delta,\dtzetap,\dtzetam,\delta,\tzetap,\tzetam\big)^{\rm T}$ that we wrote in abstract form as 
\begin{equation}\label{ODEaugmapp}
\frac{d}{dt}\Theta={\mathcal G}\big(\Theta, (R_1\mfsw)_+,(R_1\mfsw)_-,F_{\rm ext}\big),
\end{equation}
where ${\mathcal G}$ is a smooth function of its arguments. The goal of this section is to derive the explicit expression of the mapping ${\mathcal G}$ which is used for our numerical computations. We first provide in \S \ref{appexpcoeff} the explicit expression of various coefficients that appear in the wave-structure equations and then derive in \S \ref{appGC} the explicit expression of the mapping ${\mathcal G}$ in the most general case. We then point out  the simplifications that can be performed when the object is in fixed or forced motion (in \S \ref{appfixforced}) and when the system has a symmetry with respect to the vertical axis $\{x=0\}$ (in \S \ref{simplsym} ).

\medbreak

 \noindent
 {\bf N.B.} We recall that for the sake of simplicity, we assume throughout this article that $h_{\rm eq}(x)$ is an even function.

\subsection{Explicit expressions of some coefficients}\label{appexpcoeff}
We did not made explicit in the main text of the article most of the constants that appear in the wave-structure interaction equations studied in this paper and derived in \cite{BeckLannes} because they were not relevant for the mathematical and numerical analysis of these equations. Of course, they are necessary for realistic simulations of wave-structure interactions and we provide them here. 
Let us first remind that the  configuration under consideration is a floating object with vertical sidewalls located at $x=\pm \ell$ and that can only move in the vertical direction. In dimensionless variables, we denote by $h_{\rm eq}(x)$ the water depth below the object at equilibrium and by $\eps\delta(t)$ the displacement of the object at time $t$ from its equilibrium position, so that that the water depth under the object at time $t$ is $h_{\rm eq}(x)+\eps\delta(t)$. The dimensionless mass $m$ of the object can be defined through Archimedes' principle,
$$
m=\frac{1}{2\ell} \int_{-\ell}^\ell (1-h_{\rm eq})
$$
and  the formulas below will also involve two scalar functions $\alpha(\eps\delta)$ and  $\beta(\eps\delta)$ defined as
\begin{align*}
\alpha(\eps\delta)=\frac{1}{2\ell}\int_{-\ell}^\ell \frac{1}{h_{\rm eq}(x)+\eps\delta}{\rm d}x
\quad\mbox{ and }\quad
\beta(\eps\delta)=\frac{1}{2}\frac{1}{2\ell}\int_{-\ell}^\ell \frac{x^2}{(h_{\rm eq}(x)+\eps\delta)^2}{\rm d}x;
\end{align*}
the quantities $\tau_\kappa(\eps\delta)^2$ that appears in Newton's equation is given by
$$
\tau_\kappa(\eps\delta)^2=3\kappa^2 m +\frac{1}{2\ell}\int_{-\ell}^\ell \frac{x^2}{h_{\rm eq}+\eps\delta}{\rm d}x+\kappa^2 \av{\frac{1}{h_{\rm eq}+\eps\delta}}.
$$
\subsection{The general case}\label{appGC}

As seen in \eqref{GencoupleODE}, the first four components of $\Theta$ satisfy
\begin{equation}\label{GencoupleODEapp}
{\mathcal M}[\eps\delta,\eps\tzetapm]
\frac{d}{dt}\begin{pmatrix} \dsp  \av{q_{\rm i}} \\  \dsp \dot\delta \\ \dsp \dtzetap  \\ \dsp  \dtzetam \end{pmatrix}
+ \begin{pmatrix} \dsp \frac{1}{2\ell} \jump{\tzeta}\\  \dsp \delta-\av{\tzeta} \\ \dsp  \tzetap  \\ \dsp   \tzetam  \end{pmatrix}
= \eps {\boldsymbol {\mathfrak Q}}[\eps\delta,\eps\tzetapm](\av{q_{\rm i}},\dot\delta,\tzetapm)
+ \begin{pmatrix}  0 \\ F_{\rm ext} \\\dsp  ({R}_1 \mfsw)_+ \\ \dsp  ({R}_1 \mfsw)_-  \end{pmatrix}
\end{equation}
where $\av{\tzeta}:=\frac{1}{2}(\tzetap+\tzetam)$, $\jump{\tzeta}:=\zetap-\tzetam$, and ${\mathcal M}[\eps\delta,\eps\tzetapm]$ is the invertible matrix
\begin{equation}\label{defMtotapp}
{\mathcal M}[\eps\delta,\eps\tzetapm]:=
\left(
\begin{array}{cc|cc}
\alpha(\eps\delta)& 0  & \frac{\kappa^2}{2 \ell }\frac{1}{\uh_+} &  - \frac{\kappa^2}{2 \ell }\frac{1}{\uh_-} \\ 
0 & \tau_\kappa(\eps\delta)^2 & -  \frac{1}{2} \frac{\kappa^2}{\uh_+} & -  \frac{1}{2} \frac{\kappa^2}{\uh_-} \\ \hline
-{\kappa} & \ell \kappa & \kappa^2 &0  \\
\kappa & \ell \kappa & 0 & \kappa^2 
\end{array}
\right),
\end{equation}
with $\uh_\pm=1+\eps\tzeta_\pm$; simple computations also show that the quadratic term ${\boldsymbol {\mathfrak Q}}[\eps\delta,\eps\tzetapm]$ is of the form
$$
{\boldsymbol {\mathfrak Q}}[\eps\delta,\eps\tzetapm]=\big( {\mathfrak Q}_{\rm i}[\eps\delta,\eps\tzetapm],{\mathfrak Q}_{\delta}[\eps\delta,\eps\tzetapm],{\mathfrak Q}_+[\eps\delta,\eps\tzetapm],{\mathfrak Q}_-[\eps\delta,\eps\tzetapm]\big)^{\rm T}
$$
with (writing simply ${\mathfrak Q}_{\rm i}={\mathfrak Q}_{\rm i}[\eps\delta,\eps\tzetapm]$, etc.)
\begin{align*}
{\mathfrak Q}_{\rm i}(\av{q_{\rm i}},\dot\delta,\tzetapm)
&=-\alpha'(\eps\delta)\av{q_{\rm i}}\dot\delta-\frac{1}{4\ell}\big[ \Big( \frac{-\ell\dot\delta+\av{q_{\rm i}}}{\uh_+}\Big)^2-\Big(\frac{\ell\dot\delta+\av{q_{\rm i}}}{\uh_-}\Big)^2\big], \\
{\mathfrak Q}_{\delta}(\av{q_{\rm i}},\dot\delta,\tzetapm)&=
\beta(\eps\delta)\dot\delta^2+\frac{1}{2}\alpha'(\eps\delta)\av{q_{\rm i}}^2+\frac{1}{4}\big[ \Big( \frac{-\ell\dot\delta+\av{q_{\rm i}}}{\uh_+}\Big)^2+\Big(\frac{\ell\dot\delta+\av{q_{\rm i}}}{\uh_-}\Big)^2\big],\\
{\mathfrak Q}_{+}(\av{q_{\rm i}},\dot\delta,\tzetapm)&=\frac{1}{2} \tzetap^2-\frac{1}{\uh_+}\big(-\ell\dot\delta+\av{q_{\rm i}}\big)^2,\\
{\mathfrak Q}_{-}(\av{q_{\rm i}},\dot\delta,\tzetapm)&= \frac{1}{2}\tzetam^2-\frac{1}{\uh_-}\big(\ell\dot\delta+\av{q_{\rm i}}\big)^2.
\end{align*}
The matrix ${\mathcal M}$ is a $4\times 4$ matrix whose inverse is quite complicated; we therefore transform it into a block-triangular matrix by multiplying
\eqref{GencoupleODEapp} by the matrix
$$
\left(
\begin{array}{cc|cc}
 1 & 0  & -\frac{1}{2 \ell }\frac{1}{\uh_+} &   \frac{1}{2 \ell }\frac{1}{\uh_-} \\ 
0 & 1   &   \frac{1}{2} \frac{1}{\uh_+} &   \frac{1}{2} \frac{1}{\uh_-} \\ \hline
0& 0& 1 &0  \\
0 & 0 & 0 & 1
\end{array}
\right);
$$
the resulting equation takes the form
\begin{equation}\label{GencoupleODEapp2}
\widetilde{\mathcal M}[\eps\delta,\eps\tzetapm]
\frac{d}{dt}\begin{pmatrix} \dsp  \av{q_{\rm i}} \\  \dsp \dot\delta \\ \dsp \dtzetap  \\ \dsp  \dtzetam \end{pmatrix}
+ \begin{pmatrix} \dsp 0\\  \dsp \delta \\ \dsp  \tzetap  \\ \dsp   \tzetam \end{pmatrix}
= \eps \widetilde{\boldsymbol {\mathfrak Q}}(\av{q_{\rm i}},\dot\delta,\tzetapm)
+ \begin{pmatrix}  \dsp -\frac{1}{2\ell} \jump{\frac{1}{\uh}R_1\mfsw} \\ \dsp \av{\frac{1}{\uh}R_1\mfsw}+F_{\rm ext} \\\dsp  ({R}_1 \mfsw)_+ \\ \dsp  ({R}_1 \mfsw)_-  \end{pmatrix}
\end{equation}
where the matrix $\widetilde{\mathcal M}[\eps\delta,\eps\tzetapm]$ is block-triangular,
\begin{equation}\label{defMtotapp2}
\widetilde{\mathcal M}[\eps\delta,\eps\tzetapm]:=
\left(
\begin{array}{cc|cc}
\alpha(\eps\delta)+\frac{\kappa}{\ell}\av{\frac{1}{\uh}}& -\frac{\kappa}{2}\jump{\frac{1}{\uh}}   & 0 &  0\\ 
-\frac{\kappa}{2}\jump{\frac{1}{\uh}} & \tau_\kappa(\eps\delta)^2+\kappa\ell\av{\frac{1}{\uh}} & 0 & 0\\ \hline
-{\kappa} & \ell \kappa & \kappa^2 &0  \\
\kappa & \ell \kappa & 0 & \kappa^2 
\end{array}
\right),
\end{equation}
and the components $\widetilde{\mathfrak Q}_{\rm i}$, $ \widetilde{\mathfrak Q}_\delta$ and $\widetilde {\mathfrak Q}_\pm$ are given by
\begin{align*}
\nonumber
\widetilde{\mathfrak Q}_{\rm i}(\av{q_{\rm i}},&\dot\delta,\tzetapm)=- \frac{1}{4\ell}\jump{\frac{1}{\uh}\tzeta^2}\\
&+\frac{1}{4}\ell \jump{\frac{1}{\uh^2}}\dot\delta^2 +\frac{1}{4\ell}\jump{\frac{1}{\uh^2}}\av{q_{\rm i}}^2
-\big(\alpha'(\eps\delta)+\av{\frac{1}{\uh^2}}\big)\dot\delta \av{q_{\rm i}},\\
\nonumber
\widetilde{\mathfrak Q}_{\delta}(\av{q_{\rm i}},&\dot\delta,\tzetapm)= \frac{1}{2}\av{\frac{1}{\uh}\tzeta^2}\\
&+\big(\beta(\eps\delta)-\frac{1}{2}\ell^2 \av{\frac{1}{\uh^2}}\big)\dot\delta^2 
+\frac{1}{2}\big(\alpha'(\eps\delta)-\av{\frac{1}{\uh^2}}\big)\av{q_{\rm i}}^2
+\frac{1}{2}\ell \jump{\frac{1}{\uh^2}}\dot\delta \av{q_{\rm i}},
\end{align*}
and
\begin{align*}
\widetilde{\mathfrak Q}_+(\av{q_{\rm i}},\dot\delta,\tzetapm)&=
 -\frac{1}{2}\tzetap^2 - \frac{1}{\uh_+}\av{q_{\rm i}}^2 -\ell^2 \frac{1}{\uh_+}\dot\delta^2 +2\ell \frac{1}{\uh_+}\av{q_{\rm i}}\dot\delta,\\
\widetilde {\mathfrak Q}_-(\av{q_{\rm i}},\dot\delta,\tzetapm)&=
 -\frac{1}{2}\tzetam^2 - \frac{1}{\uh_-}\av{q_{\rm i}}^2 -\ell^2 \frac{1}{\uh_-}\dot\delta^2 -2\ell \frac{1}{\uh_-}\av{q_{\rm i}}\dot\delta.
\end{align*}
Since by definition of ${\mathcal G}$ one has
\begin{equation}\label{GencoupleODEapp3}
\frac{d}{dt}\big( \av{q_{\rm i}}, \dot\delta,\dtzetap, \dtzetam \big)^{\rm T}={\mathcal G}_{\rm I}(\Theta,(R_1\mfsw)_\pm,F_{\rm ext})
\end{equation}
with the notation ${\mathcal G}_{\rm I}:=\big( {\mathcal G}_1, {\mathcal G}_2,{\mathcal G}_3,{\mathcal G}_4\big)^{\rm T}$; we deduce from \eqref{GencoupleODEapp2} 
that
\begin{align*}
{\mathcal G}_{\rm I}&(\Theta,(R_1\mfsw)_\pm,F_{\rm ext})\\
&=\widetilde{\mathcal M}[\eps\delta,\eps\tzetapm]^{-1}
\Big[ -\begin{pmatrix} \dsp 0\\  \dsp \delta \\ \dsp  \tzetap  \\ \dsp   \tzetam \end{pmatrix}
+ \eps \widetilde{\boldsymbol {\mathfrak Q}}(\av{q_{\rm i}},\dot\delta,\tzetapm)
+ \begin{pmatrix}  \dsp -\frac{1}{2\ell} \jump{\frac{1}{\uh}R_1\mfsw} \\ \dsp \av{\frac{1}{\uh}R_1\mfsw} +F_{\rm ext}\\\dsp  ({R}_1 \mfsw)_+ \\ \dsp  ({R}_1 \mfsw)_-  \end{pmatrix}
\Big],
\end{align*}
while 
$$
{\mathcal G}_5(\Theta)=\dot\delta,\qquad {\mathcal G}_6(\Theta)=\dot\tzetap, \qquad  {\mathcal G}_7(\Theta)=\dot\tzetam.
$$
\begin{remark}
For the numerical computations, we use the explicit expression for the inverse of the matrix $\widetilde{\mathcal M}[\eps\delta,\eps\zetapm]$, namely, 
\begin{equation}\label{invMexpl}
\widetilde{\mathcal M}[\eps\delta,\eps\tzetapm]^{\rm -1}=
\left(
\begin{array}{cc|cc}
-\frac{4}{D} \big(\tau_\kappa(\eps\delta)^2+\kappa\ell \av{\frac{1}{\uh}}\big) &  -\frac{2\kappa}{D} \jump{\frac{1}{\uh}}     & 0 &0 \\
 -\frac{2\kappa}{D} \jump{\frac{1}{\uh}}  & - \frac{4}{D} \big(\alpha(\eps\delta)+\frac{\kappa}{\ell}\av{\frac{1}{\uh}}\big)& 0 & 0 \\ \hline
 -\frac{4}{\kappa D}\big(\tau_\kappa(\eps\delta)^2+\kappa\ell {\frac{1}{\uh_-}}\big) & \frac{4}{\kappa D}\big( \kappa {\frac{1}{\uh_-}}+\ell \alpha(\eps\delta)\big) & \frac{1}{\kappa^2} & 0 \\
\frac{4}{\kappa D}\big(\tau_\kappa(\eps\delta)^2+\kappa\ell {\frac{1}{\uh_+}}\big) &\frac{4}{\kappa D}\big( \kappa {\frac{1}{\uh_+}}+\ell \alpha(\eps\delta)\big) & 0 &  \frac{1}{\kappa^2}
\end{array}
\right)
\end{equation}
with
\begin{equation}\label{defDapp}
D=-4\big(\alpha(\eps\delta)+\frac{\kappa}{\ell}\av{\frac{1}{\uh}}\big)\times \big(\tau_\kappa(\eps\delta)^2+\kappa\ell \av{\frac{1}{\uh}}\big)+\kappa^2 \jump{\frac{1}{\uh}}^2.
\end{equation}
\end{remark}
\subsection{The case of an object fixed or in forced motion}\label{appfixforced}

When the object is fixed or in forced motion, the position of the center of mass is known and $\delta$ is therefore equal to some given function $\delta_{\rm forced}$ ($\delta_{\rm forced}\equiv 0$ if the solid is fixed). The ODE \eqref{ODEaugmapp} can be reduced to an ODE on $\RR^5$ instead of $\RR^7$. The variable $\Theta$ now stands for  $\Theta:=\big(\av{q_{\rm i}},\dot\tzetap,\dot\tzetam,\tzetap,\tzetam\big)^{\rm T}$ and \eqref{GencoupleODEapp2} can be simplified into
\begin{align}
\nonumber
\widetilde{\mathcal M}_{\rm forced}[t,\eps\tzetapm]
\frac{d}{dt}\begin{pmatrix} \dsp  \av{q_{\rm i}}  \\ \dsp \dtzetap  \\ \dsp  \dtzetam \end{pmatrix}
+ \begin{pmatrix} \dsp 0 \\ \dsp  \tzetap  \\ \dsp   \tzetam \end{pmatrix}
=& \eps \widetilde{\boldsymbol {\mathfrak Q}}_{\rm forced}[t,\eps\tzetapm](\av{q_{\rm i}},\zetapm)\\
\label{GencoupleODEapp2forced}
&
+ \begin{pmatrix}  \dsp -\frac{1}{2\ell} \jump{\frac{1}{\uh}R_1\mfsw} \\ \dsp  ({R}_1 \mfsw)_+ \\ \dsp  ({R}_1 \mfsw)_-  \end{pmatrix}
+\begin{pmatrix}
\frac{1}{2}\kappa \jump{\frac{1}{\uh}}\\
-\ell\kappa\\
-\ell\kappa
\end{pmatrix}
\ddot\delta_{\rm forced},
\end{align}
where the matrix $\widetilde{\mathcal M}_{\rm forced}[t,\eps\tzetapm]$ is block-triangular,
\begin{equation}\label{defMtotapp2forced}
\widetilde{\mathcal M}_{\rm forced}[t,\eps\tzetapm]:=
\left(
\begin{array}{cc|cc}
\alpha(\eps\delta_{\rm forced})+\frac{\kappa}{\ell}\av{\frac{1}{\uh}}  & 0 &  0\\ 
-{\kappa} & \kappa^2 &0  \\
\kappa  & 0 & \kappa^2 
\end{array}
\right),
\end{equation}
and
$$
\widetilde{\boldsymbol{\mathfrak Q}}_{\rm forced}[t,\eps\zetapm](\av{q_{\rm i}},\tzetapm):=
\begin{pmatrix}
\widetilde{\mathfrak Q}_{\rm i}[\eps\delta_{\rm forced},\eps\zetapm](\av{q_{\rm i}},\dot\delta_{\rm forced},\tzetapm)\\
\widetilde{\mathfrak Q}_{+}[\eps\delta_{\rm forced},\eps\zetapm](\av{q_{\rm i}},\dot\delta_{\rm forced},\tzetapm)\\
\widetilde{\mathfrak Q}_{-}[\eps\delta_{\rm forced},\eps\zetapm](\av{q_{\rm i}},\dot\delta_{\rm forced},\tzetapm)
\end{pmatrix},
$$
and ${\mathfrak Q}_{\rm i}$ and ${\mathfrak Q}_\pm$ as in the previous section.
\begin{remark}
We have made explicit the dependence of $\widetilde{\mathcal M}_{\rm forced}$ and $\widetilde{\boldsymbol{\mathfrak Q}}_{\rm forced}$ on the time variable $t$ because $\delta_{\rm forced}$ is now an explicit function of time, and is now an non autonomous contribution to the ODE for $\Theta$ (except of course if the object is fixed, in which case $\delta_{\rm forced}\equiv 0$).
\end{remark}
\noindent
With ${\mathcal G}_{\rm I}$ now being three dimensional (but with an extra dependence on $t$),  ${\mathcal G}_{\rm I}:=\big( {\mathcal G}_1, {\mathcal G}_2,{\mathcal G}_3\big)^{\rm T}$, we have therefore
\begin{equation}\label{GencoupleODEapp3forced}
\frac{d}{dt}\big( \av{q_{\rm i}},\dot\tzeta_+, \dot \tzeta_- \big)^{\rm T}={\mathcal G}_{\rm I}(t,\Theta,(R_1\mfsw)_\pm)
\end{equation}
and
\begin{align*}
{\mathcal G}_{\rm I}&(t,\Theta,(R_1\mfsw)_\pm)
=\widetilde{\mathcal M}_{\rm forced}[t,\eps\tzetapm]^{-1}\\
\times&\Big[ -\begin{pmatrix} \dsp 0 \\ \dsp  \tzetap  \\ \dsp   \tzetam \end{pmatrix}
+ \eps \widetilde{\boldsymbol {\mathfrak Q}}_{\rm forced}[t,\eps\tzetapm](\av{q_{\rm i}},\tzetapm)
+ \begin{pmatrix}  \dsp -\frac{1}{2\ell} \jump{\frac{1}{\uh}R_1\mfsw}  \\ \dsp  ({R}_1 \mfsw)_+ \\ \dsp  ({R}_1 \mfsw)_-  \end{pmatrix}
+\begin{pmatrix}
\frac{1}{2}\kappa \jump{\frac{1}{\uh}}\\
-\ell\kappa\\
-\ell\kappa
\end{pmatrix}
\ddot\delta_{\rm forced}
\Big],
\end{align*}
while ${\mathcal G}_4(\Theta)=\dot\tzetap$,  ${\mathcal G}_5(\Theta)=\dot\tzetam$ and 
$$
\widetilde{M}_{\rm forced}[t,\eps\tzetapm]^{-1}=\begin{pmatrix}
\dsp \frac{1}{\alpha+\frac{\kappa}{\ell}\av{\frac{1}{\uh}}} & 0 & 0\\
\dsp \frac{1}{\kappa} \frac{1}{\alpha+\frac{\kappa}{\ell}\av{\frac{1}{\uh}}} & \frac{1}{\kappa^2} &0\\
\dsp -\frac{1}{\kappa} \frac{1}{\alpha+\frac{\kappa}{\ell}\av{\frac{1}{\uh}}} & 0 &\frac{1}{\kappa^2}
\end{pmatrix}.
$$
\begin{remark}\label{remcontrol}
The second component of \eqref{GencoupleODEapp3} (evolution equation on $\delta$) does not appear any longer in \eqref{GencoupleODEapp3forced} but remains of course valid. It can be used to answer the following control problem: what external force should we apply to the object so that the vertical displacement of its center of gravity coincides with $\delta_{\rm forced}$? The answer is explicitly given (writing ${\mathfrak Q}_{\delta}={\mathfrak Q}_{\delta}[\eps\delta_{\rm forced},\eps\zetapm]$, etc.)
\begin{align*}
F_{\rm ext}=&\delta-\av{\frac{1}{h}R_1\mfsw}-\eps{\mathfrak Q}_{\delta}(\av{q_{\rm i}},\dot\delta_{\rm forced},\zetapm)\\
&-\frac{1}{4}\frac{D_{\rm forced}}{\alpha(\eps\delta_{\rm forced})+\frac{\kappa}{\ell}\av{\frac{1}{h}}}\ddot\delta_{\rm forced}\\
&-\frac{1}{2}\frac{\kappa}{\alpha(\eps\delta_{\rm forced})+\frac{\kappa}{\ell}\av{\frac{1}{h}}}\jump{\frac{1}{h}}\Big[ \eps {\mathfrak Q}_{\rm i}(\av{q_{\rm i}},\dot\delta_{\rm forced},\zetapm)-\frac{1}{2\ell}\jump{\frac{1}{h}R_1\mfsw}\Big],
\end{align*}
where $D_{\rm forced}$ is deduced from the expression given for $D$ in \eqref{defDapp} by substituting $\delta_{\rm forced}$ to $\delta$.
\end{remark}

\subsection{Simplifications in the symmetric case}\label{simplsym}

When the object is symmetric with respect to the vertical axis $\{x=0\}$ (i.e. if $h_{\rm eq}$ is an even function), as assumed throughout this article, it is possible to consider symmetric flows for which $\zeta$ is an even function, $q$ is odd, and $\av{q_{\rm i}}\equiv 0$  (such conditions are propagated by the equations from the initial data). This is for instance the case for waves generated by a floating object in a fluid initially at rest. By symmetry, the augmented transmission problem \eqref{transm1augm}-\eqref{ODEaugm} reduces to an augmented initial boundary value problem on the half-line $\cE^+=(\ell,\infty)$, 
\begin{equation}\label{transm1augmsym}
\begin{cases}
\dt \zeta+\dx q=0,\\
(1-\kappa^2\dx^2)\dt q +\eps \dx \big( \frac{1}{h}q^2\big)+h\dx\zeta=0
\end{cases}
\quad\mbox{ for }\quad t>0, \quad x\in (\ell,\infty)
\end{equation}
with boundary condition
\begin{equation}\label{transm2augmsym}
q_{\vert_{x=\ell}}=-\ell \dot\delta,
\end{equation}
where  $\delta$ is a  function of time determined by the first order ODE 
\begin{equation}\label{ODEaugmsym}
\frac{d}{dt}\Theta={\mathcal G}\big(\Theta, (R_1\mfsw)_+,F_{\rm ext}\big),
\end{equation}
with $\Theta:=\big(\dot\delta,\dot\tzetap,\delta,\tzetap\big)^{\rm T}$ and where ${\mathcal G}=({\mathcal G}_1,{\mathcal G}_2,{\mathcal G}_3,{\mathcal G}_4)^{\rm T}$ is given by
\begin{align*}
\begin{pmatrix}
{\mathcal G}_1\\
{\mathcal G}_2
\end{pmatrix}
=
{\mathcal M}_{\rm sym}[\eps\delta,\eps\zetap]^{-1}
\Big[&-\begin{pmatrix} \delta \\ \zetap \end{pmatrix}+\eps \begin{pmatrix} {\mathfrak Q}_\delta^{\rm sym}[\eps\delta,\eps\zetap](\dot\delta,\zetap) \\
{\mathfrak Q}_+^{\rm sym}[\eps\delta,\eps\zetap](\dot\delta,\zetap)
\end{pmatrix}\\
&+
\begin{pmatrix}
\frac{1}{h_+}(R_1\mfsw)_++F_{\rm ext}\\
(R_1\mfsw)_+
\end{pmatrix}\Big]
\end{align*}
and ${\mathcal G}_3=\dot\delta$, ${\mathcal G}_4=\dot\zetap$, and with
$$
{\mathcal M}_{\rm sym}[\eps\delta,\eps\zetap]=
\begin{pmatrix}
\tau_\kappa(\eps\delta)^2+\kappa\ell{\frac{1}{h_+}} & 0 \\
\ell\kappa & \kappa^2
\end{pmatrix},
$$
while ${\mathfrak Q}_\delta^{\rm sym}$ and ${\mathfrak Q}_+^{\rm sym}$ are obtained by replacing $\av{q_{\rm i}}=0$ and $\zetam=\zetap$ in the formula derived above for ${\mathfrak Q}_\delta$ and ${\mathfrak Q}_\delta$.

Two particular physical situations of particular interest fit into the symmetric framework and are investigated in this paper:
\begin{itemize}
\item The return to equilibrium. An object is released from an out of equilibrium position in a fluid initially at rest. This situation is described by \eqref{transm1augmsym}-\eqref{ODEaugmsym} with $F_{\rm ext}=0$ and  initial conditions $(\zeta,q)_{\vert_{t=0}}=(0,0)$ and $(\dot\delta,\dot\zeta_+,\delta,\zeta_+)_{\vert_{t=0}}=(0,0,\delta^{\rm in},0)$.
\item Wave generation. Waves are generated in a fluid initially at rest by moving the object up and down with a prescribed motion $\delta_{\rm forced}$. The problem then reduces to an initial boundary value problem with boundary condition on $q$, namely, $q_{\vert_{x=\ell}}=g$, with $g=-\ell\dot\delta_{\rm forced}$. This boundary data is explicitly given and does not require the resolution of a first order ODE as the other problems considered here. 
\end{itemize}

\section*{Acknowledgment}
D. L. was partially supported by the grant ANR-18-CE40-0027-01 Singflows.


\end{document}